\documentclass[11pt,reqno]{amsart}
\usepackage{amssymb,amsmath,amsfonts,amsthm,enumerate,stmaryrd,tensor,mathtools,dsfont,upgreek,bbm}
\usepackage{mathrsfs}
\usepackage{nameref,hyperref,cleveref}
\usepackage{microtype}
\numberwithin{equation}{section}

\usepackage[left=0.75 in, right=0.75 in,top=0.75 in, bottom=0.75 in]{geometry}
\renewcommand{\Bar}{\overline}

\newcommand{\R}{\mathbb{R}}
\newcommand{\N}{\mathbb{N}}
\newcommand{\Z}{\mathbb{Z}}

\newcommand{\C}{\mathbb{C}}

\newcommand{\imp}{\;\Rightarrow\;}
\newcommand{\m}{\mathrm}

\newcommand{\lv}{\lVert}
\newcommand{\rv}{\rVert}
\newcommand{\Lm}{\Lambda}

\newcommand{\al}{\alpha}
\newcommand{\be}{\beta}
\newcommand{\es}{\varnothing}
\newcommand{\lra}{\;\Leftrightarrow\;}
\newcommand{\ep}{\varepsilon}
\newcommand{\f}{\frac}
\newcommand{\sig}{\sigma}

\newcommand{\del}{\delta}

\newcommand{\pd}{\partial}

\newcommand{\grad}{\nabla}
\newcommand{\bpm}{\begin{pmatrix}}
\newcommand{\epm}{\end{pmatrix}}

\newcommand{\loc}{\m{loc}}
\renewcommand{\bar}{\overline}
\newcommand{\Le}{\mathfrak{L}}

\newcommand{\emb}{\hookrightarrow}

\newcommand{\norm}[1]{\left\lv#1\right\rv}

\newcommand{\qnorm}[1]{\left \vert \mspace{-1.8mu} \left\vert
\mspace{-1.8mu} \left \lvert #1 \right \vert \mspace{-1.8mu} \right\vert
\mspace{-1.8mu} \right\vert}

\newcommand{\bnorm}[1]{\Big\lv#1\Big\rv}
\newcommand{\snorm}[1]{\big\lv#1\big\rv}
\newcommand{\bp}[1]{\Big(#1\Big)}

\newcommand{\abs}[1]{\left|#1\right|}
\newcommand{\p}[1]{\left(#1\right)}

\newcommand{\br}[1]{\left\langle{#1}\right\rangle}

\renewcommand{\sb}[1]{\left[{#1}\right]}
\newcommand{\bsb}[1]{\Big[{#1}\Big]}
\newcommand{\ssb}[1]{\big[{#1}\big]}

\newcommand{\cb}[1]{\left\{{#1}\right\}}

\renewcommand{\bf}[1]{\mathbf{#1}}
\newcommand{\ii}{i}

\DeclareMathOperator{\supp}{supp}
\DeclareMathOperator{\card}{card}

\providecommand{\br}[1]{\langle #1 \rangle}

\renewcommand{\le}{\leqslant}
\renewcommand{\ge}{\geqslant}

\newtheorem{prop}{Proposition}[section]
\newtheorem{thm}[prop]{Theorem}
\newtheorem{defn}[prop]{Definition}
\newtheorem{lem}[prop]{Lemma}
\newtheorem{coro}[prop]{Corollary}

\newtheorem{innercustomthm}{Theorem}
\newenvironment{customthm}[1]
  {\renewcommand\theinnercustomthm{#1}\innercustomthm}
  {\endinnercustomthm}
  
\def\XXint#1#2#3{{\setbox0=\hbox{$#1{#2#3}{\int}$ }
\vcenter{\hbox{$#2#3$ }}\kern-.6\wd0}}

\theoremstyle{definition}
\newtheorem{exax}[prop]{Example}
\newenvironment{exa}
  {\pushQED{\qed}\exax}
  {\popQED\endexax}

\title[Truncated interpolation and screened Sobolev spaces]{
A truncated real interpolation method and characterizations of screened Sobolev spaces
}

\author{Noah Stevenson}
\address{
Department of Mathematical Sciences\\
Carnegie Mellon University\\
Pittsburgh, PA 15213, USA
}
\email[N. Stevenson]{nwsteven@andrew.cmu.edu}

\author{Ian Tice}
\address{
Department of Mathematical Sciences\\
Carnegie Mellon University\\
Pittsburgh, PA 15213, USA
}
\email[I. Tice]{iantice@andrew.cmu.edu}
\thanks{I. Tice was supported by an NSF CAREER Grant (DMS \#1653161). N. Stevenson was supported by the summer research support provided by this grant. }


\begin{document}

\begin{abstract} 

In this paper we prove structural and topological characterizations of the screened Sobolev spaces with screening functions bounded below and above by positive constants. We generalize a method of interpolation to the case of seminormed spaces.  This method, which we call the truncated method, generates the screened Sobolev subfamily and a more  general screened Besov scale. We then prove that the screened Besov spaces are equivalent to the sum of a Lebesgue space and a homogeneous Sobolev space and provide a Littlewood-Paley frequency space characterization.

\end{abstract}

\maketitle

\section{Introduction}

\subsection{Background}
The study of partial differential equations on unbounded domains is a catalyst for the development of new analytical tools and spaces of functions.  One reason for this is that the classical scale of inhomogeneous Sobolev spaces fails to provide a suitable functional setting for PDEs on these domains.   This is evident in the basic exterior Dirichlet problem in $\R^2$, where $u\p{x}=\log\abs{x}$ solves
\begin{equation}\label{this is the first PDE}
    \begin{cases}
    -\Delta u=0&\text{in }\R^2\setminus\Bar{B\p{0,e}}\\
    u=1&\text{on }\pd B\p{0,e}.
    \end{cases}
\end{equation}
While we have that $u\in\dot{W}^{1,p}(\R^2\setminus\Bar{B\p{0,e}})\cap\dot{W}^{2,q}(\R^2\setminus\Bar{B\p{0,e}})$ for all $2<p\le\infty$ and $1<q\le\infty$, $u$ does not belong to  $L^r(\R^2\setminus\Bar{B\p{0,e}})$ for any $1\le r\le\infty$. 

One approach for dealing with this problem is to switch to weighted inhomogeneous Sobolev spaces: see, for instance \cite{MR554783,MR312241,MR380094,MR802206,MR2753293,MR1923391}.  An alternative approach is to directly utilize homogeneous Sobolev spaces, but for this to be fruitful in the study of boundary value problems it is essential to know their associated trace spaces.  The trace spaces associated to homogeneous Sobolev spaces on infinite strip-like domains of the form $\{x \in \R^{n+1} : \eta^-(x') < x_{n+1} < \eta^+(x')\}$, for $\eta^\pm : \R^n \to \R$ Lipschitz with $\eta^- < \eta^+$, were recently characterized by Leoni and Tice \cite{leoni2019traces}.  They used this to characterize the solvability of certain quasilinear elliptic boundary value problems in these domains.  This trace theory has also been used in recent studies of the Muskat problem by Nguyen and Pausader \cite{nguyen_pausader}, Nguyen \cite{2019arXiv190711552N}, and Flynn and Nguyen \cite{2020arXiv200110473F}.  

A curious feature of this trace theory is that regularity of the trace function is measured with a fractional Sobolev seminorm involving a screening effect.  These screened Sobolev seminorms were first studied by  Strichartz~\cite{MR3481175}, who proved that the fractional regularity associated to traces of $\dot{H}^1(\R\times\p{0,1})$  is characterized by the seminorm
\begin{equation}\label{strichartz}
    \sb{f}_{\tilde{H}^{\f12}\p{\R}}=\Big(\int_{\R}\int_{\p{x-1,x+1}}\abs{f\p{x}-f\p{y}}^2\abs{x-y}^{-2}\;\m{d}y\;\m{d}x\Big)^{1/2} \text{ for }f\in L^1_\loc\p{\R}.
\end{equation}
Comparing the expression in~\eqref{strichartz} to the seminorm on a homogeneous Sobolev-Slobodeckij space $\dot{H}^{\f12}\p{\R}$, one sees that the moniker `screening' is justified since in the former only small difference quotients are allowed, and larger ones are screened away.

The generalization of this result in~\cite{leoni2019traces} required the introduction more general seminorms.  For an open set $\es\neq U\subseteq\R^n$, a lower semicontinuous function $\sig:U\to(0,\infty]$, $s\in\p{0,1}$ (called the screening function), and $1\le p<\infty$, \cite{leoni2019traces} defines the screened Sobolev space $\tilde{W}^{s,p}_{\p{\sig}}\p{U}$ as the collection of locally integrable functions $f$, defined on $U$, for which
\begin{equation}\label{screed space defn}
    \sb{f}_{\tilde{W}^{s,p}_{\p{\sig}}} = \Big(\int_{U}\int_{B\p{x,\sig\p{x}}\cap U}\abs{f\p{x}-f\p{y}}^p\abs{x-y}^{-n-sp}\;\m{d}y\;\m{d}x\Big)^{1/p} < \infty.
\end{equation}
Variants of these screened spaces have appeared in recent work on fractional Sobolev-type seminorms \cite{bourgain-brezis-mironescu2001,bourgain-brezis-mironescu2002,brezis-nguyen2018,ponce2004,ponce-spector2017} and in weak formulations of nonlocal elliptic equations \cite{MR3023366,MR3318251,MR2733097}.  However, the space $\tilde{W}^{s,p}_{\p{\sig}}$ above did not appear in previous literature, so \cite{leoni2019traces} established its basic properties: completeness, strict inclusion of $\dot{W}^{1,p}$, and a partial frequency space characterization in the case that $p=2$, $\sig=1$: for $f\in\mathscr{S}\p{\R^n;\R}$ we have the equivalence
\begin{equation}\label{tice's partial fourier characterization}
    \sb{f}_{\tilde{W}^{s,2}_{\p{1}}}\asymp\Big(\int_{\R^n}\min\{\abs{\xi}^2,\abs{\xi}^{2s}\}\abs{\hat{f}\p{\xi}}^2\;\m{d}\xi\Big)^{1/2}.
\end{equation}
Deeper questions related to density, embeddings, traces, a more robust frequency space characterization, and interpolation were left open in \cite{leoni2019traces}, and a central goal of this paper is to fill that gap.

The key to unlocking these deeper properties is the characterization of the screened spaces in terms of interpolation theory, specifically the real method of abstract interpolation (see~\cite{MR0482275,MR3753604,MR0243340}).
One expects such a characterization, as this is the case for the  Sobolev-Slobodeckij and Besov spaces.  We refer to the works
\cite{adams-fournier2003book, MR0482275, besov-ilin-nikolskii1979book,  besov-ilin-nikolskii1978book, burenkov1998book, dinezza-palatucci-valdinoci2012,  grisvard2011book, MR3726909, mazya2011book,necas,MR0461123,MR503903}  and their references for a thorough study of these spaces and their interpolation properties.  However, the standard methods of abstract interpolation only generate Banach interpolation spaces intermediate to an appropriately compatible pair of Banach spaces.  The screened Sobolev spaces are only seminormed spaces with non-Hausdorff topologies and thus, without appropriate modification, interpolation methods cannot generate these spaces.  Literature regarding the theory of interpolation of seminormed spaces appears to be sparse, aside from a technical report of Gustavsson \cite{gustavsson}, which is difficult to find in print.

\subsection{Primary results and discussion}

We survey the principle new results regarding the screened spaces obtained in this paper. For brevity's sake, we do not provide fully detailed statements and only record their abbreviated forms. The proper statements can be found later in the indicated theorems. 

In order to study the screened Sobolev spaces, we actually introduce a more general scale of screened Besov spaces, $\tilde{B}^{s,p}_{q}\p{\R^n}$, for $s\in\p{0,1}$ and $1\le p,q\le\infty$.   See Definition~\ref{screened besov spaces defn} for the precise definition.  Our first result finds sufficient conditions that identify the screened Sobolev spaces within the screened Besov scale.

\begin{customthm}{1}[Proved in Corollary~\ref{sum characterization of screened sobolev spaces}]\label{1}
If $\sig:\R^n\to\R^+$ is a lower semicontinuous screening function bounded above and below by positive constants,  $s\in\p{0,1}$, and $1\le p<\infty$, then the  screened Sobolev space $\tilde{W}^{s,p}_{\p{\sig}}\p{\R^n}$ is equivalent to the screened Besov space $\tilde{B}^{s,p}_p\p{\R^n}$.
\end{customthm}
With the established connection between the scales of screened Sobolev and screened Besov spaces, we move to structurally and topologically characterize the latter. We find that there is a method of interpolation of seminormed spaces that generates the screened Besov spaces.

\begin{customthm}{2}[Proved in Theorem~\ref{screened sobolev spaces are truncated interpolation spaces}]\label{2}
There is a method of interpolation of seminormed spaces, called the truncated real-method, that generates the screened Besov spaces as interpolation spaces with respect to a Lebesgue space and a homogeneous Sobolev space.
\end{customthm}

The interpolation characterization of the screened Besov spaces leads to the following characterization.

\begin{customthm}{3}[Proved in Theorem
~\ref{fundamental decomposition of screened sobolev spaces}]\label{3}
For $s\in\p{0,1}$ and $1\le p,q\le\infty$, the screened Besov space $\tilde{B}^{s,p}_q\p{\R^n}$ is  equivalent to the sum of the inhomogeneous Besov space $B^{s,p}_q\p{\R^n}$ and the homogeneous Sobolev space $\dot{W}^{1,p}\p{\R^n}$.
\end{customthm}

The sum characterization from the previous theorem allows us to characterize when the subspace of compactly supported smooth functions is dense.

\begin{customthm}
{4}[Proved in Corollaries ~\ref{density of com} and~\ref{lack of density}]\label{4}
For $1\le p,q<\infty$ and $s\in\p{0,1}$ the set $C^\infty_c\p{\R^n}$ is dense in the screened Besov space $\tilde{B}^{s,p}_q\p{\R^n}$ if and only if $1<p$ or $2\le n$.
\end{customthm}

We next generalize~\eqref{tice's partial fourier characterization} by giving a Littlewood-Paley characterization of the screened spaces.

\begin{customthm}{5}[Proved in Corollary~\ref{general fourier char of screened spaces}]\label{5}
Let $1<p<\infty$, $1\le q\le\infty$, $s\in\p{0,1}$. For all functions $f\in\tilde{B}^{s,p}_q\p{\R^n}$ we have that $f$ is a tempered distribution and that the following equivalence holds:
\begin{equation}
    \sb{f}_{\tilde{B}^{s,p}_q}\asymp\bnorm{\bp{\textstyle{\sum_{j\in\Z\setminus\N}}\p{2^j\abs{\uppi_jf}}^2}^{1/2}}_{L^p}+\norm{\cb{2^{sj}\norm{\uppi_j f}_{L^p}}_{j\in\N}}_{\ell^q\p{\N}},
\end{equation}
where $\cb{\uppi_j}_{j\in\Z}$ are a family of dyadic localization operators, as in Definition~\ref{dyadic localization}.
\end{customthm}

Theorem \ref{5} follows from a somewhat more general result that provides a Littlewood-Paley characterization of the interpolant between two Riesz potential spaces, $\dot{H}^{r_i,p}\p{\R^n}$ for $i\in\{1,2\}$ (see Definition~\ref{riesz potential spaces}).

\begin{customthm}{6}[Proved in Theorem \ref{truncated interpolation of Riesz potential spaces}]\label{6}
Let $1<p<\infty$, $1\le q\le\infty$, $\al\in\p{0,1}$, $\sigma \in \R^+$, and $r,s\in\R$ with $r<s$.  Set $t=\p{1-\al}r+\al s$.  Then the interpolation space $\big(\dot{H}^{r,p}\p{\R^n},\dot{H}^{s,p}\p{\R^n}\big)^{\p{\sigma}}_{\al,q}$ is characterized by the Littlewood-Paley seminorm
\begin{equation}
    f \mapsto \bnorm{\bp{\textstyle{\sum_{j\in\Z\setminus\N}}\p{2^{sj}\abs{\uppi_j f}}^2}^{1/2}}_{L^p}+\snorm{\cb{2^{tj}\norm{\uppi_jf}_{L^p}}_{j\in\N}}_{\ell^q\p{\N;\R}}.
\end{equation}
\end{customthm}

The interesting feature of Theorems \ref{5} and \ref{6} is that the Littlewood-Paley characterization changes form between low frequencies and high frequencies.  For low frequencies, a Triebel-Lizorkin type of seminorm arises, but for high frequencies it is of Besov type.  Note also that the power $2^{sj}$ in the low frequencies is inherited from the second factor in the interpolation.  This explains why the low frequency Fourier multiplier in \eqref{tice's partial fourier characterization}, $\abs{\xi}^2$, matches that associated to $\dot{H}^1(\R^n)$.

Our final result concerns embedding and restriction (trace) results for these spaces.

\begin{customthm}{7}[Proved in Section \ref{sec:embeddings} and Theorem \ref{restriction theorem}]\label{7}
The screened Besov spaces enjoy various embeddings, and, provided that $n\ge 2$, $1<p<\infty$, $1\le q\le\infty$, and  $p^{-1}<s<1$, they admit well-defined restriction operators with continuous right-inverses.
\end{customthm}

Broadly speaking, our strategy for proving the above results is to take the analytical high road: these results are consequences of our development of a more general abstract theory.  We begin the paper, in Section~\ref{interpolation of seminormed spaces}, with the development of interpolation methods of seminormed spaces in the abstract. Recall that the $K$-method for Banach spaces takes a pair of spaces $X_0$ and $X_1$ and constructs their intermediate $s,q$-interpolation space, $\p{X_0,X_1}_{s,q}$, as the collection of all elements $x$ belonging to the sum of $X_0$ and $X_1$ for which the map $\R^+\ni t\mapsto t^{-s}K\p{t,x}\in\R$ belongs to $L^q\p{\R^+;\mu}$,  were $\mu$ is the Haar measure associated to the multiplicative group on $\R^+$ (see Section \ref{notation stuff}). We find that this $K$-method is not quite right to produce the screened spaces as interpolation spaces. However, it is nearly correct.  We need only consider a slight generalization to seminormed spaces and allow for a larger family of domains of integration. 

For a parameter $\sig\in(0,\infty]$ we study the `truncated' interpolation space $\p{Y_0,Y_1}_{s,q}^{\p{\sig}}$ with respect to seminormed spaces $Y_0$ and $Y_1$. The truncated spaces are characterized as the set of $y$ belonging to the sum of $Y_0$ and $Y_1$ for which the map $\p{0,\sig}\ni t\mapsto t^{-s}K\p{t,y}\in\R$ belongs to $L^q\p{\p{0,\sig},\mu}$. We find that for $\sig=\infty$ the seminormed interpolation mirrors that of the interpolation theory of Banach spaces, with only a few more subtleties regarding notions of compatibility. On the other hand, when $\sig<\infty$ the truncated method does give interpolation spaces; however it is interestingly asymmetric in the roles of $Y_0$ and $Y_1$. The upshot of studying these methods of abstract seminormed interpolation is that we obtain a general relationship between the methods for $\sig<\infty$ and $\sig=\infty$.  More precisely, in Theorem \ref{label sum characterization of the truncated method} we find an abstract sum characterization: for $\sig <\infty$ the truncated interpolation space $\p{Y_0,Y_1}^{\p{\sig}}_{s,q}$ is equivalent to the sum of $\p{Y_0,Y_1}^{\p{\infty}}_{s,q}$ and the second factor, $Y_1$.

Section~\ref{Homogeneous Sobolev and homogeneous Besov spaces} is a three-fold development of vital analytical tools utilized in the later study of screened Sobolev and screened Besov spaces. The inspiration for this section is the Littlewood-Paley theory in Chapter 6 of~\cite{MR3243734}, the applications of harmonic analysis to study smoothness in Chapter 1 of~\cite{MR3243741}, and the interpolation of Sobolev and Besov spaces in Chapter 6 of~\cite{MR0482275}. First, we define the homogeneous Sobolev spaces and the Riesz potential spaces. The latter is a two parameter space, $\dot{H}^{s,p}\p{\R^n}$,  for $s\in\R$ and $1<p<\infty$ (see Definition~\ref{riesz potential spaces}) where, roughly speaking, a tempered distribution $f$ belongs to $\dot{H}^{s,p}\p{\R^n}$ if $[\abs{\cdot}^s\hat{f}]^{\vee}$ defines a function in $L^p\p{\R^n}$. Note that this scale is intimately related to the homogeneous Sobolev spaces; however, we work directly with seminorms rather than quotient by polynomials to obtain a normed space.  We prove a frequency space characterization of $\dot{W}^{1,p}\p{\R^n}$ that says that the former space is essentially  equivalent to the Riesz potential space $\dot{H}^{1,p}\p{\R^n}$.  We then pair this with a Littlewood-Paley decomposition of the Riesz potential spaces to deduce a Littlewood-Paley characterization of the homogeneous Sobolev space $\dot{W}^{1,p}\p{\R^n}$.

Next, we study the $L^p$-modulus of continuity and its relationship with the $K$-functional on the sum of $L^p\p{\R^n}$ and $\dot{W}^{1,p}\p{\R^n}$. Having established this, we use the interpolation of seminormed spaces developed in Section~\ref{interpolation of seminormed spaces} to show that the homogeneous Besov spaces (see Definition~\ref{defintion of besov spaces}), $\dot{B}^{s,p}_q\p{\R^n}$, are generated via $\big(L^p\p{\R^n},\dot{W}^{1,p}\p{\R^n}\big)_{s,q}^{\p{\infty}}$ for $s\in\p{0,1}$ and $1\le q\le\infty$.

As a final development in Section~\ref{Homogeneous Sobolev and homogeneous Besov spaces}, we explore the homogeneous Besov-Lipschitz scale of spaces, ${_\wedge}\dot{B}^{s,p}_q\p{\R^n}$ with parameters $s\in\R$, $1<p<\infty$, and $1\le q\le\infty$ (see Definition~\ref{homogeneous besov-lipschitz spaces}). With the Littlewood-Paley decomposition of the Riesz-Potential spaces, we see that the theory of seminorm interpolation realizes the equivalence: ${_\wedge}\dot{B}^{r,p}_q\p{\R^n}=\big(\dot{H}^{s,p}\p{\R^n},\dot{H}^{t,p}\p{\R^n}\big)^{\p{\infty}}_{\al,q}$ for $r=\p{1-\al}s+\al t$. This interpolation result, supplemented with the interpolation characterization of the homogeneous Besov spaces, reveals a Littlewood-Paley characterization of the latter scale.

Section~\ref{Screened Sobolev and screened Besov spaces} synthesizes the abstract seminormed space interpolation of Section~\ref{interpolation of seminormed spaces} and the analysis of Section~\ref{Homogeneous Sobolev and homogeneous Besov spaces} to obtain a deeper understanding of the screened Sobolev spaces. First we generalize the scale of screened Sobolev spaces by defining the screened Besov spaces in Definition~\ref{screened besov spaces defn}. Having already developed the homogeneous Besov spaces and the connection between the $L^p$ modulus of continuity and the $K$-functional associated to the sum of $L^p\p{\R^n}$ and $\dot{W}^{1,p}\p{\R^n}$, the claims in Theorem~\ref{2} above are now immediate. Then the abstract sum characterization of the truncated interpolation method gives, with a little more work, Theorems~\ref{3} and~\ref{4}. We next apply the truncated interpolation method to general pairs of Riesz potential spaces and from this analysis we obtain the claims of Theorem~\ref{5}. Finally, we use the sum characterization of the screened Besov spaces to quickly read off some results on embeddings and traces.


\subsection{Conventions of notation}\label{notation stuff}
We record our conventions of notation used throughout this paper. The number sets $\N$, $\Z$, $\R$, and $\C$ are the natural numbers, integers, reals, and complex numbers, respectively. We assume that $0\in\N$ and write $\N^+=\N\setminus\cb{0}$ and $\R^+=\p{0,\infty}$.  In writing $\R^n$ we always assume $n \in \N^+$.

Throughout the paper we denote the field $\mathbb{K}\in\cb{\R,\C}$.  The spaces of rapidly decreasing and analytic functions taking values in $\mathbb{K}$ are denoted by $\mathscr{S}\p{\R^n;\mathbb{K}}$ and $C^\omega\p{\R^n;\mathbb{K}}$, respectively. The space of tempered distributions valued in $\mathbb{K}$ is denoted $\mathscr{S}^{\ast}\p{\R^n;\mathbb{K}}$. The Fourier transform is denoted either as $\hat{\cdot}$ or $\mathscr{F}$. For $0<\al\le1$, the homogeneous H\"older space (homogeneous Lipschitz space when $\al=1$) $\dot{C}^{0,\al}\p{\R^n;\mathbb{K}}$ is the space of functions $f:\R^n\to\mathbb{K}$ such that $\sb{f}_{\dot{C}^{0,\al}}=\sup\cb{\abs{f\p{x}-f\p{y}}/\abs{x-y}^{\al}\;:\;x,y\in\R^n,\;x\neq y} < \infty$.

We let $\mu$ denote the standard Haar measure with respect to the multiplicative structure on $\R^+$, i.e. $\mu\p{E}=\int_{\R^+}\chi_E\p{t}t^{-1}\;\m{d}t$ for Lebesgue measurable sets $E\subseteq\R^+$. The $n$-dimensional Lebesgue measures and $s$-dimensional Hausdorff measures are denoted $\mathfrak{L}^n$ and $\mathcal{H}^s$, respectively. Moreover we choose the normalization of $\mathcal{H}^s$ so that when $s=n$ we have $\mathcal{H}^n = \mathfrak{L}^n$. 

Finally,  whenever the expression $a \lesssim b$ appears in a proof of a result, it means that there is a constant $C\in\R^+$, depending only on the parameters quantified in the statement of the result such that $a\le Cb$.  We may sharpen this by occasionally writing the explicit dependence of the constant $C$ as a subscript on $\lesssim$, i.e. $a\lesssim_{s,p,q}b$. We write $a \asymp b$ to mean $a \lesssim b$ and $b \lesssim a$.


\section{Interpolation of seminormed spaces}\label{interpolation of seminormed spaces}
In this section we present two distinct methods of generating interpolation spaces intermediate to a pair of seminormed spaces satisfying certain compatibility conditions. Both generalize known methods of interpolating between couples of Banach spaces.  The first method is a seminorm generalization of the well-known `real method of interpolation' (see for instance, the paper~\cite{MR0243340}, Chapter 3 in~\cite{MR0482275}, or Chapter 1 in~\cite{MR503903,MR3753604}), and as such we refer to this method as the real method of interpolation of seminormed spaces.  The real seminorm method has its origins in the work of Gustavsson \cite{gustavsson}, and here we essentially follow his approach, with a few embellishments.  

The second method, which we call the truncated real method, generates spaces in a seemingly similar way to the real method; however, it is bizarrely asymmetric and generates larger spaces than the non-truncated method.  The truncated method has its origins in the work of Gomez and Milman~\cite{gomez_milman_1986}, who employed it to study the extreme parameter regime of Peetre's interpolation theory for nested Banach spaces, with the aim of proving certain estimates for singular integral operators.  A more thorough study of the interpolation properties of the limiting spaces in the nested case commenced with the paper of Cobos, Fern\'{a}ndez-Cabrera,  K\"{u}hn, and Ullrich~\cite{cobos_etal_2009}.  Recent work of Astashkin, Lykov, and Milman~\cite{astashkin_etal_2019} removed the nested assumption, considered a more general parameter regime, and uncovered a deep connection between extrapolation theory (see the book of Jawerth and Milman~\cite{MR1046185}) and the limiting case of the real method.  We were led to consider a seminorm version of this theory in studying the trace theory of homogeneous Sobolev spaces on certain unbounded domains.

The spaces obtained from the non-truncated method appear crucially at a few points in the truncated theory, so it is important for us to have a careful enumeration of their properties.  The technical report \cite{gustavsson} is not available in journals or online, so we have recorded a number of its results below and indicated how to obtain the proofs from the arguments used in the second method.

A concise review of relevant topological notions in seminormed spaces is presented in Appendix~\ref{seminorm topology}.  Throughout the following section all generic seminormed spaces are over a fixed common field - either real or complex.
\subsection{Topology of compatible couples}\label{topology of compatible couples}
We begin by exploring notions of compatibility of seminormed spaces. In this first subsection we consider what happens when two seminormed spaces are simultaneously contained within some larger vector space. We can then consider their sum and intersection and give each of those a seminorm in a natural way.

\begin{defn}[Compatibility of seminormed spaces]\label{admissable semi normed spaces}
Suppose that $\p{X_0,\sb{\cdot}_0}$ and  $\p{X_1,\sb{\cdot}_1}$ are seminormed spaces.
\begin{enumerate}
    \item We say that they are a strongly compatible pair if there exists a topological vector space $\p{Y,\tau}$ such that $X_i\emb Y$ for $i\in\cb{0,1}$, and $\mathfrak{A}\p{Y}=\mathfrak{A}\p{ X_0} \cup \mathfrak{A}\p{X_1}$, where the annihilator $\mathfrak{A}$ is defined in Definition \ref{defn of topological vector space}. 
    Note that, due to Proposition \ref{seminormed spaces are TVS}, the second condition implies that either $\mathfrak{A}\p{X_0}\subseteq\mathfrak{A}\p{X_1}$ or $\mathfrak{A}\p{X_1}\subseteq\mathfrak{A}\p{X_0}$.
    \item We say that $\p{X_0,\sb{\cdot}_0}$ and $\p{X_1,\sb{\cdot}_1}$ are a weakly compatible pair if there is a vector space $Y$ with $X_0,X_1\subseteq Y$. Note that every strongly compatible pair is automatically a weakly compatible pair.
    \item In the case that $X_0$ and $X_1$ are either a strong or weak compatible pair of seminormed spaces we form their sum and their intersection in the usual way:
    \begin{equation}
        \Sigma\p{X_0,X_1}=\cb{x\in Y\;:\;\exists\p{x_0,x_1}\in X_0\times X_1,\;x=x_0+x_1}\text{ and }
        \Delta\p{X_0,X_1}=X_0\cap X_1.
    \end{equation}
    We endow these spaces with the seminorms $\sb{\cdot}_\Sigma:\Sigma\p{X_0,X_1}\to\R$ and $\sb{\cdot}_{\Delta}:\Delta\p{X_0,X_1}\to\R$ defined by 
    \begin{equation}
         \sb{x}_\Sigma =\inf\cb{\sb{x_0}_0+\sb{x_1}_1\;:\;\p{x_0,x_1}\in X_0\times X_1,\;x=x_0+x_1}\text{ and }
        \sb{x}_\Delta = \max\cb{\sb{x}_0,\sb{x}_1}.
    \end{equation}
    Observe that we have the continuous embeddings  $\Delta\p{X_0,X_1}\emb X_i\emb\Sigma\p{X_0,X_1}$ for each $i\in\cb{0,1}$.
\end{enumerate}
\end{defn}

\begin{defn}[Intermediate and interpolation spaces]\label{intermediate and interpolation spaces}
Suppose that $\p{X_0,\sb{\cdot}_0}$ and  $\p{X_1,\sb{\cdot}_1}$ are  a weakly compatible couple of seminormed spaces.
\begin{enumerate}
    \item We say that a seminormed space $\p{Y,\sb{\cdot}}$ is intermediate with respect to the couple $\p{X_0,\sb{\cdot}_0}$, $\p{X_1,\sb{\cdot}_1}$ if it holds that
        $\Delta\p{X_0,X_1}\emb Y\emb\Sigma\p{X_0,X_1}$.
    \item Suppose that $\p{Y_0,\sb{\cdot}_0}$ and $\p{Y_1,\sb{\cdot}_1}$  are another weakly compatible couple of seminormed spaces, and that $\p{X,\sb{\cdot}_X}$ and $\p{Y,\sb{\cdot}_Y}$ are another pair seminormed spaces, with $X\subseteq\Sigma\p{X_0,X_1}$ and $Y\subseteq\Sigma\p{Y_0,Y_1}$. We say that $X$ and $Y$ are a pair of interpolation spaces if for every linear map $T:\Sigma\p{X_0,X_1}\to\Sigma\p{Y_0,Y_1}$ that is continuous with values in $Y_i$ when restricted to $X_i$, $i\in \{0,1\}$, it holds that $TX \subseteq Y$ and $T: X \to Y$ is continuous.
\end{enumerate}
\end{defn}

The definition of strong compatibility that we give ensures that the intersection of a compatible couple behaves well with respect to completeness.

\begin{prop}\label{completeness of sum and intersection}
Suppose that $\p{X_0,\sb{\cdot}_0}$ and $\p{X_1,\sb{\cdot}_1}$ are a weakly compatible pair of semi-Banach spaces.  Then the space $\p{\Sigma\p{X_0,X_1},\sb{\cdot}_{\Sigma}}$ is semi-Banach.  If we assume that the pair is strongly compatible, then $\p{\Delta\p{X_0,X_1},\sb{\cdot}_{\Delta}}$ is semi-Banach.
\end{prop}
\begin{proof}
If $\cb{x_k}_{k=0}^\infty\subset\Delta\p{X_0,X_1}$ is Cauchy, then the continuous embedding $\Delta\p{X_0,X_1}\emb X_i$ for $i\in\cb{0,1}$ paired with completeness of $X_i$ implies that there are $\p{a,b}\in X_0\times X_1$ such that
\begin{equation}\label{convergence 1}
    \lim_{k\to\infty}\p{\sb{x_k-a}_0+\sb{x_k-b}_1}=0.
\end{equation}
By definition of strongly admissible pair, there is some topological vector space $Y$ such that $X_0,X_1\emb Y$ and $\mathfrak{A}\p{Y}=\mathfrak{A}\p{X_0}\cup\mathfrak{A}\p{X_1}$. Hence $x_k\to a,b$ as $k\to\infty$ in $Y$ as $k\to\infty$. Thus, by Proposition~\ref{kernel measures failure of a space to be hausdorff}, $a-b\in\mathfrak{A}\p{Y}=\mathfrak{A}\p{X_0}\cup\mathfrak{A}\p{X_1}$. Let's handle the case that $a-b\in\mathfrak{A}\p{X_0}$ (the other case is similar). Then, $b=a+\p{b-a}\in X_0+\mathfrak{A}\p{X_0}\subseteq X_0$, and hence $b\in\Delta\p{X_0,X_1}$. Moreover for any $k \in \N$ we have the bound
\begin{equation}\label{bound 1}
    \sb{x_k-b}_{\Delta}\le\sb{x_k-b}_0+\sb{x_k-b}_1\le\sb{b-a}_0+\sb{x_k-a}_0+\sb{x_k-b}_1=\sb{x_k-a}_0+\sb{x_k-b}_1.
\end{equation}
Thus~\eqref{bound 1} paired with~\eqref{convergence 1} shows that $x_k\to b$ in $\Delta\p{X_0,X_1}$, showing this space to be complete.

To prove completeness of $\Sigma\p{X_0,X_1}$, we use Lemma~\ref{characterizations of completeness in seminormed spaces}. Suppose that $\cb{x_k}_{k=0}^\infty\subset\Sigma\p{X_0,X_1}$ is a sequence such that $\sum_{k=0}^\infty\sb{x_k}_{\Sigma}<\infty$. For each $k \in \N$, we can find, by the definition of the seminorm on $\Sigma\p{X_0,X_1}$, a pair $\p{a_k,b_k}\in X_0\times X_1$ for which $a_k+b_k=x_k$ and $\sb{a_k}_0+\sb{b_k}_1\le 2^{-k}+\sb{x_k}_{\Sigma}$. Summing over $k$ in the inequality reveals that $
    -2+\sum_{k=0}^\infty\sb{a_k}_0+\sum_{k=0}^\infty\sb{b_k}_1\le\sum_{k=0}^\infty\sb{x_k}_{\Sigma}<\infty$.
Consequently, there are $\p{a,b}\in X_0\times X_1$ for which 
    $\lim_{K\to\infty}\ssb{-a+\sum_{k=0}^Ka_k}_0=0$ and $\lim_{K\to\infty}\ssb{-b+\sum_{k=0}^Kb_k}_1=0$.
Set $x=a+b\in\Sigma\p{X_0,X_1}$. For any $K \in \N$ we may then estimate $
    \ssb{-x+\sum_{k=0}^Kx_k}_{\Sigma}\le\ssb{-a+\sum_{k=0}^Ka_k}_0+\ssb{-b+\sum_{k=0}^Kb_k}_1$.
As $K\to\infty$, the previous right hand side vanishes, completing the proof.
\end{proof}

We can also say something about the annihilators of the sum and intersection.

\begin{prop}[Annihilators of sum and intersection]\label{kernel of the sum and intersection}
Suppose that $\p{X_0,\sb{\cdot}_0}$ and $\p{X_1,\sb{\cdot}_1}$ are a pair of weakly admissible seminormed spaces. Then $\mathfrak{A}\p{\Delta\p{X_0,X_1}} = \Delta\p{\mathfrak{A}\p{X_0},\mathfrak{A}\p{X_1}}$.  Also, we have the inclusion    $\Sigma\p{\mathfrak{A}\p{X_0},\mathfrak{A}\p{X_1}} \subseteq \mathfrak{A}\p{\Sigma\p{X_0,X_1}}$, and if if we additionally assume that $\p{X_0,\sb{\cdot}_0}$ and $\p{X_1,\sb{\cdot}_1}$ are strongly compatible or $\Delta\p{X_0,X_1}$ is semi-Banach, then equality holds.
\end{prop}
\begin{proof}
The first assertion is trivial, so we only prove the second. If $x\in\Sigma\p{\mathfrak{A}\p{X_0},\mathfrak{A}\p{X_1}}$, then there are $\p{x_0,x_1}\in\mathfrak{A}\p{X_0}\times\mathfrak{A}\p{X_1}$ such that $x=x_0+x_1$. Hence $0\le\sb{x}_{\Sigma}\le\sb{x_0}_0+\sb{x_1}_1=0$, and the inclusion is shown.

Suppose first that the pair of seminormed spaces are strongly admissible. Thus we may find $\p{Y,\tau}$ a topological vector space such that $\forall\;i\in\cb{0,1}$ we have $X_i\emb Y$, and $\mathfrak{A}\p{Y}=\mathfrak{A}\p{X_0}\cup\mathfrak{A}\p{X_1}$. Thus, if $x\in\mathfrak{A}\p{\Sigma\p{X_0,X_1}}$, then we may find sequences $\cb{y_m}_{m\in\N}\subset X_0$ and $\cb{z_m}_{m\in\N}\subset X_1$ such that $x=y_m+z_m$ for all $m\in\N$, and $\sb{y_m}_0+\sb{z_m}_1\le2^{-m}$ for $m\in\N$. The continuous embeddings $X_0,X_1\emb Y$ imply that $y_m,z_m\to0$ in $Y$ as $m\to\infty$, and hence $x\in\mathfrak{A}\p{Y}\subseteq\Sigma\p{\mathfrak{A}\p{X_0},\mathfrak{A}\p{X_1}}$.

Suppose next that $\Delta\p{X_0,X_1}$ is semi-Banach, in which case we employ an argument from \cite{gustavsson}.  Let $x\in\mathfrak{A}\p{\Sigma\p{X_0,X_1}}$.  Then for each $n\in\N$, we can find a decomposition of $x$ with the following property:
\begin{equation}
    x=y_n+z_n,\;\p{y_n,z_n}\in X_0\times X_1,\;\sb{y_n}_0+\sb{z_n}_1\le2^{-n}.
\end{equation}
Observe that for $n\in\N$ we have $w_n=y_n-y_0=z_0-z_n\in\Delta\p{X_0,X_1}$. Hence, for $m,n\in\N$ we may estimate:
\begin{equation}
    \sb{w_n-w_m}_{\Delta}=\sb{\p{y_n-y_0}-\p{y_m-y_0}}_{\Delta}=\max\cb{\sb{y_n-y_m}_0,\sb{z_n-z_m}_1}\le 2^{-m}+2^{-n}.
\end{equation}
Then $\cb{w_n}_{n\in\N}\subseteq\Delta\p{X_0,X_1}$ is Cauchy. Since $\Delta\p{X_0,X_1}$ is semi-Banach by hypothesis, there is $w\in\Delta\p{X_0,X_1}$ such that $w_n\to w$ as $n\to\infty$ in $\Delta\p{X_0,X_1}$. Now we observe that $y_0+w\in X_0$, $z_0-w\in X_1$, and $\p{y_0+w}+\p{z_0-w}=x$. For $n\in\N$ it holds that
\begin{equation}
    \begin{cases}
    \sb{y_0+w}_0\le\sb{y_0+w_n}_0+\sb{w-w_n}_{\Delta}\le 2^{-n}+\sb{w-w_n}_{\Delta}\\
    \sb{z_0-w}_1\le\sb{z_0-w_n}_1+\sb{w_n-w}_{\Delta}\le 2^{-n}+\sb{w-w_n}_{\Delta}
    \end{cases}\to0\quad\text{as}\quad n\to\infty.
\end{equation}
Hence $x\in\Sigma\p{\mathfrak{A}\p{X_0},\mathfrak{A}\p{X_1}}$, as desired.
\end{proof}

This result tells us that one of the downsides to having a weak, but not strong, compatible pair is that the annihilator of the sum may grow larger than one desires.

\subsection{The $K$-methods}
We define the $K$-functional in the same way as the normed space theory.
\begin{defn}[$K$-functional]\label{K functional for semi normed spaces}
Given $\p{X_0,\sb{\cdot}_0}$ and $\p{X_1,\sb{\cdot}_1}$, a weakly compatible pair of seminormed spaces (see Definition~\ref{admissable semi normed spaces}), we define the functional $\mathscr{K}:\R^+\times\Sigma\p{X_0,X_1}\to\R$ via
\begin{equation}
    \mathscr{K}\p{t,x}=\inf\cb{\sb{x_0}_0+t\sb{x_1}_1\;:\;\p{x_0,x_1}\in X_0\times X_1,\;x=x_0+x_1}.
\end{equation}
\end{defn}

The following proposition contains the most basic properties of the $K$-functional from Definition~\ref{K functional for semi normed spaces}.

\begin{prop}\label{basic properties of K for seminormed}
Given a weakly compatible pair of seminormed spaces $\p{X_0,\sb{\cdot}_0}$ and $\p{X_1,\sb{\cdot}_1}$, the following hold:
\begin{enumerate}
    \item $\forall\;t\in\R^+$, $\mathscr{K}\p{t,\cdot}$ is a seminorm on $\Sigma\p{X_0,X_1}$, and $\mathscr{K}\p{1,\cdot}=\sb{\cdot}_{\Sigma}$.
    \item For all $x\in\Sigma\p{X_0,X_1}$ and for all $t,s\in\R^+$ we have the estimates
    \begin{equation}\label{equivalence}
        \min\cb{1,t/s}\mathscr{K}\p{s,x}\le\mathscr{K}\p{t,x}\le\max\cb{1,t/s}\mathscr{K}\p{s,x}.
    \end{equation}
    \item $\forall\;x\in\Sigma\p{X_0,X_1}$, the mapping $\R^+\ni t\mapsto\mathscr{K}\p{t,x}\in\R$ is concave, increasing, and measurable.
\end{enumerate}
\end{prop}
\begin{proof}
These three items are immediate from the definition of $\mathscr{K}$.
\end{proof}

Using the $K$-functional, we can define the following families of extended seminorms on the sum.

\begin{defn}[$K$-methods' interpolation spaces]\label{interpolation spaces}
Let $\p{X_0,\sb{\cdot}_0}$ and $\p{X_1,\sb{\cdot}_1}$ be a weakly compatible pair of seminormed spaces, $s\in\p{0,1}$, $1\le q\le\infty$, $\sig\in(0,\infty]$.   We define $\sb{\cdot}_{s,q}^{\p{\sig}}:\Sigma\p{X_0,X_1}\to\sb{0,\infty}$ via
\begin{equation}
    \sb{x}_{s,q}^{\p{\sig}}
    =
    \bp{\int_{\p{0,\sig}}\p{t^{-s}\mathscr{K}\p{t,x}}^qt^{-1}\;\m{d}t}^{1/q}\text{ if }q<\infty,
    \text{ and }\sup\cb{t^{-s}\mathscr{K}\p{t,x}\;:\;t\in\p{0,\sig}}\text{ when }q=\infty.
\end{equation}
We define the $K$-methods' interpolation spaces to be the sets $\p{X_0,X_1}_{s,q}^{\p{\sig}}=\{x\in\Sigma\p{X_0,X_1}\;:\;\sb{x}_{s,q}^{\p{\sig}}<\infty\}$, which are a vector spaces thanks to Proposition \ref{basic properties of K for seminormed} and Minkowski's inequality on $L^q(\p{0,\sig},\mu)$ ($\mu$ is as in Section~\ref{notation stuff}).  Moreover, we equip the space $\p{X_0,X_1}_{s,q}^{\p{\sig}}$ with the seminorm $\sb{\cdot}_{s,q}^{\p{\sig}}$.

In the case that $\sig=\infty$, we often write $\p{X_0,X_1}_{s,q}$ and $\sb{\cdot}_{s,q}$ in place of $\p{X_0,X_1}_{s,q}^{\p{\infty}}$ and $\sb{\cdot}_{s,q}^{\p{\infty}}$. This is in agreement with the existing notation for the usual $K$-method on normed vector spaces. In the case that $\sig<\infty$, we also refer to this method of generating spaces as the truncated method of interpolation.

\end{defn}
\subsection{Basic properties}
We now study basic properties of the $K$-methods of interpolation. In particular, we will prove that they are intermediate and interpolation spaces in the sense of Definition~\ref{intermediate and interpolation spaces}, then we study various inclusion, embedding, and completeness properties, and finally we exhibit equivalent discrete seminorms. Along the way, we will see that for fixed $s$ and $q$, the $K$-methods' interpolation spaces are only topologically distinct for $\sig$ finite and $\sig$ infinite. 

\begin{prop}[$K$-method spaces are intermediate]\label{truncated interpolation spaces are intermediate}
Suppose that $\p{X_0,\sb{\cdot}_0}$ and $\p{X_1,\sb{\cdot}_1}$ are a weakly compatible pair of seminormed spaces. Then for all $s\in\p{0,1}$, $\sig\in(0,\infty]$, and $1\le q\le\infty$, we have continuous embeddings: $
    \Delta\p{X_0,X_1}\emb\p{X_0,X_1}^{\p{\sig}}_{s,q}\emb\Sigma\p{X_0,X_1}$.
\end{prop}
\begin{proof}
We only prove the case that $\sig<\infty$, as the case $\sig=\infty$ is proved in the exact same way as the real method of interpolation of normed vector spaces.

First we consider the case when $q=\infty$. Then for any $x\in\Delta\p{X_0,X_1}$ and any $t\in\p{0,\sig}$, we have the estimate $\mathscr{K}\p{t,x}\le\min\cb{1,t}\sb{x}_{\Delta}$. Hence $\sb{x}^{\p{\sig}}_{s,\infty}\le\min\cb{1,\sig^{1-s}}\sb{x}_{\Delta}$. If now $x\in\p{X_0,X_1}^{\p{\sig}}_{s,\infty}$, we use the fact that for $t\in\p{0,\sig}$ $\mathscr{K}\p{t,x}\ge\min\cb{1,t}\sb{x}_{\Sigma}$. Hence $\sb{x}^{\p{\sig}}_{s,\infty}\ge\min\cb{1,\sig^{1-s}}\sb{x}_{\Sigma}$.

Next, we handle $1\le q<\infty$. If $x\in\Delta\p{X_0,X_1}$ we can use the same estimate as before:
\begin{equation}\label{helium}
    \sb{x}^{\p{\sig}}_{s,q}=\bp{\int_{\p{0,\sig}}\p{t^{-s}\mathscr{K}\p{t,x}}^qt^{-1}\;\m{d}t}^{1/q} \le\sb{x}_{\Delta}\bp{\int_{\p{0,\sig}}\min\cb{t^{-sq},t^{q\p{1-s}}}t^{-1}\;\m{d}t}^{1/q}=C_{s,q,\sig}\sb{x}_{\Delta}.
\end{equation}
And if $x\in\p{X_0,X_1}^{\p{\sig}}_{s,q}$ we obtain:
\begin{equation}
    \sb{x}^{\p{\sig}}_{s,q}=\bp{\int_{\p{0,\sig}}\p{t^{-s}\mathscr{K}\p{t,x}}^qt^{-1}\;\m{d}t}^{1/q} \ge\sb{x}_{\Sigma}\bp{\int_{\p{0,\sig}}\min\cb{t^{-sq},t^{q\p{1-s}}}t^{-1}\;\m{d}t}^{1/q}=C_{s,q,\sig}\sb{x}_{\Sigma}.
\end{equation}
This completes the proof.
\end{proof}


Next, we show that the $K$-methods' interpolation spaces preserve completeness.
\begin{prop}\label{completeness of truncated spaces}
Suppose that $\p{X_0,\sb{\cdot}_0}$ and $\p{X_1,\sb{\cdot}_1}$ are a weakly compatible pair of semi-Banach spaces. Then for all $\sig\in (0,\infty]$, $s\in\p{0,1}$, and $1\le q\le\infty$ the seminormed space $\big(\p{X_0,X_1}^{\p{\sig}}_{s,q},\sb{\cdot}^{\p{\sig}}_{s,q}\big)$ is semi-Banach.
\end{prop}
\begin{proof}
We again only prove the case for $\sig<\infty$, as the case for $\sig=\infty$ follows with a similar argument. We verify completeness through the series characterization in Lemma~\ref{characterizations of completeness in seminormed spaces}. Let $\cb{x_k}_{k=0}^\infty\subset\p{X_0,X_1}_{s,q}^{\p{\sig}}$ be a sequence such that $\sum_{k=0}^\infty\sb{x_k}^{\p{\sig}}_{s,q}<\infty$. By Proposition~\ref{truncated interpolation spaces are intermediate} it follows that $\sum_{k=0}^\infty\sb{x_k}_{\Sigma}<\infty$. Then, Proposition~\ref{completeness of sum and intersection} implies that there exists $x\in\Sigma\p{X_0,X_1}$ such that $\lim_{K\to\infty}\ssb{-x+\sum_{k=0}^Kx_k}_{\Sigma}=0$.  For $K,M \in \N$ with $M > K$ we may use Proposition~\ref{basic properties of K for seminormed} to bound 
\begin{equation}\label{this is a math equation that will indeed be referenced later in this paper}
\mathscr{K}\p{t,-x+\textstyle{\sum_{k=0}^K} x_k}\le  \mathscr{K}\p{t,-x + \textstyle{\sum_{k=0}^M} x_k} +  \textstyle{\sum_{k=K+1}^M} \mathscr{K}\p{t,x_k},    
\end{equation}
and since $\mathscr{K}(t,\cdot)$ is an equivalent seminorm on $\Sigma\p{X_0,X_1}$ we may send $M \to \infty$ in ~\eqref{this is a math equation that will indeed be referenced later in this paper} and then multiply by $t^{-s}$ to deduce that $t^{-s}\mathscr{K}\big(t,-x+\textstyle{\sum_{k=0}^K} x_k\big)\le\sum_{k=K+1}^\infty t^{-s}\mathscr{K}\p{t,x_k}$,
for all $K \in \N$ and all $t\in\p{0,\sig}$.  In the case that $q=\infty$ we deduce immediately that $[-x+\sum_{k=0}^Kx_k]^{\p{\sig}}_{s,\infty}\le\sum_{k=K+1}^\infty\sb{x_k}_{s,\infty}^{\p{\sig}}$.
Hence $x\in\p{X_0,X_1}^{\p{\sig}}_{s,\infty}$. Since the right-hand-side tends to zero as $K\to\infty$, completeness is established in this case. 

We now consider the case $1\le q<\infty$.   For $b\in\N^+$ with $2^{-b}<\sig$ we integrate \eqref{this is a math equation that will indeed be referenced later in this paper}, apply Minkowski's inequality, and employ the bound
$\mathscr{K}\p{t,\cdot}\le\max\cb{1,t^{-1}}\sb{\cdot}_{\Sigma}$ (see Proposition~\ref{basic properties of K for seminormed}) to estimate:
\begin{multline}\label{a good estimate 2}
    \bp{\int_{\p{2^{-b},\sig}}\p{t^{-s}\mathscr{K}\p{t,-x+\textstyle{\sum_{k=0}^K} x_k}}^qt^{-1}\;\m{d}t}^{1/q}
    \le\bp{\int_{\p{2^{-b},\sig}}\p{t^{-s}\mathscr{K}\p{t,-x+\textstyle{\sum_{k=0}^{M}} x_k}}^qt^{-1}\;\m{d}t}^{1/q} 
    \\+\sum_{k=K+1}^{M} \bp{\int_{\p{2^{-b},\sig}}\p{t^{-s}\mathscr{K}\p{t, x_k}}^q\f{1}{t}\;\m{d}t}^{1/q}
    \le C_b\sb{-x+\textstyle{\sum_{k=0}^{M}x_k}}_{\Sigma}+\sum_{k=K+1}^\infty\sb{x_k}^{\p{\sig}}_{s,q}.
\end{multline}
The number $C_b=\big(\int_{\p{2^{-b},\sig}}\max\cb{t^{-qs},t^{q\p{1-s}}}t^{-1}\;\m{d}t\big)^{1/q}$ is finite, so we may send $M \to \infty$  in~\eqref{a good estimate 2} and use the convergence in $\Sigma\p{X_0,X_1}$ to see that
\begin{equation}\label{a good equation 3}
    \bp{\int_{\p{2^{-b},\sig}}\p{t^{-s}\mathscr{K}\p{t,-x+\textstyle{\sum_{k=0}^K} x_k}}^qt^{-1}\;\m{d}t}^{1/q}\le\sum_{k=K+1}^\infty\sb{x_k}^{\p{\sig}}_{s,q}.
\end{equation}
Letting $b\to\infty$ and using the monotone convergence theorem shows that~\eqref{a good equation 3} continues to hold with $0$ in place of $2^{-b}$. Hence $x\in\p{X_0,X_1}^{\p{\sig}}_{s,q}$. Finally, sending $K\to\infty$ shows that  $[-x+\sum_{k=0}^Kx_k]^{\p{\sig}}_{s,q} \to 0$, and we conclude that the space $\p{X_0,X_1}^{\p{\sig}}_{s,q}$ is complete.
\end{proof}
We now examine the inclusion relations among the interpolation and truncated interpolation spaces.
\begin{prop}[Inclusions and embeddings of $K$-methods' spaces]\label{inclusion relations of truncated interpolation spaces}
Suppose that $\p{X_0,\sb{\cdot}_0}$ and $\p{X_1,\sb{\cdot}_1}$ are a weakly compatible pair of seminormed spaces, $s\in\p{0,1}$, $1\le p,q\le\infty$, and $\sig\in(0,\infty],\;\rho\in\R^+$. The following hold:
\begin{enumerate}
    \item  We have the continuous embedding $\p{X_0,X_1}_{s,q} = \p{X_0,X_1}_{s,q}^{\p{\infty}} \hookrightarrow  \p{X_0,X_1}_{s,q}^{\p{\rho}}$.  Moreover, for all $x \in \p{X_0,X_1}_{s,q}$ we have that $
        \sb{x}_{s,q}^{\p{\rho}} \le         \sb{x}_{s,q}$.
    \item     If $\sig<\infty$, then we have the equality of spaces, $\p{X_0,X_1}^{\p{\sig}}_{s,q}=\p{X_0,X_1}_{s,q}^{\p{\rho}}$, with equivalence of seminorms. In fact, $\forall\;x\in\Sigma\p{X_0,X_1}$ it holds
            $\sb{x}^{\p{\sig}}_{s,q}\le\max\cb{\rho^s\sig^{-s},\sig^{1-s}\rho^{s-1}}\sb{x}^{\p{\rho}}_{s,q}$.
    \item If $\sig<\infty$ and $s<t$, then we have the continuous embedding $\p{X_0,X_1}^{\p{\sig}}_{t,q}\emb\p{X_0,X_1}^{\p{\sig}}_{s,q}$, with the following estimate  for all $x\in\p{X_0,X_1}^{\p{\sig}}_{t,q}$: $\sb{x}^{\p{\sig}}_{s,q}\le\sig^{t-s}\sb{x}_{t,q}^{\p{\sig}}$.
    \item If $\sig<\infty$, then we have the continuous embedding $X_1\emb\p{X_0,X_1}^{\p{\sig}}_{s,q}$, with the following estimate  for all $x\in X_1$: $\sb{x}^{\p{\sig}}_{s,q}\le D_{s,q,\sig}\sb{x}_1$ where $D_{s,q,\sig}=
        \sig^{1-s}q^{-1/q}\p{1-s}^{-1/q
        }$ when $q<\infty$ and $D_{s,q,\infty}=\sig^{1-s}$.
    \item If $p<q$, then we have the continuous embedding $\p{X_0,X_1}^{\p{\sig}}_{s,p}\emb\p{X_0,X_1}^{\p{\sig}}_{s,q}$.
\end{enumerate}
\end{prop}
\begin{proof}
For the first four items we only prove the case for $1\le q<\infty$, as the case for $q=\infty$ is proved analogously. The first item follows trivially from the definitions.  Given $x\in\p{X_0,X_1}^{\p{\sig}}_{s,q}$, we estimate via a change of variables and Proposition~\ref{basic properties of K for seminormed}:
    \begin{equation}
        \sb{x}_{s,q}^{\p{\sig}}=\bp{\rho^{qs}\sig^{-qs}\int_{\p{0,\rho}}\p{t^{-s}\mathscr{K}\p{\sig t/\rho,x}}^{q}t^{-1}\;\m{d}t}^{1/q}\le\max\cb{\rho^s\sig^{-s},\sig^{1-s}\rho^{s-1}}\sb{x}^{\p{\rho}}_{s,q}.
    \end{equation}
This proves the second item.   Next, for $x\in\p{X_0,X_1}^{\p{\sig}}_{t,q}$ we bound
    \begin{equation}
        \sb{x}^{\p{\sig}}_{s,q}=\bp{\int_{\p{0,\sig}}\p{\tau^{-s}\mathscr{K}\p{\tau,x}}^q\tau^{-1}\;\m{d}\tau}^{1/q}\le\sig^{t-s}\bp{\int_{\p{0,\sig}}\p{\tau^{-t}\mathscr{K}\p{\tau,x}}^q\tau^{-1}\;\m{d}\tau}^{1/q}=\sig^{t-s}\sb{x}^{\p{\sig}}_{t,q},
    \end{equation}
which proves the third item.  If $x\in X_1$, then $x=0+x\in\Sigma\p{X_0,X_1}$ is a decomposition, and so for $t\in\p{0,\sig}$ we have $\mathscr{K}\p{t,x}\le t\sb{x}_1$. Thus if $1\le q<\infty$, then $
        \sb{x}^{\p{\sig}}_{s,q}=\big(\int_{\p{0,\sig}}\p{t^{-s}\mathscr{K}\p{t,x}}^qt^{-1}\;\m{d}t\big)^{1/q}\le\sb{x}_1\p{\int_{\p{0,\sig}}t^{q\p{1-s}-1}\;\m{d}t}^{1/q}$,
and the fourth item is proved.

We will only prove the fifth item in the case that $\sig<\infty$, as the case $\sig=\infty$ follows similarly.  Let $x\in\p{X_0,X_1}^{\p{\sig}}_{s,q}$.  We first consider $q=\infty$.  For $t,\tau\in\p{0,\sig}$ we may use Proposition~\ref{basic properties of K for seminormed} to bound $t^{-s} \mathscr{K}\p{\tau,x}\le \max\cb{1,\tau t^{-1}} t^{-s}\mathscr{K}\p{t,x}$.
In turn,
\begin{multline}\label{more lables are the coolest ever}
\sb{x}^{\p{\sig}}_{s,p}
\ge \mathscr{K}\p{\tau,x}\bp{\int_{\p{0,\sig}}\p{t^{-s}\min\cb{1,t\tau^{-1}}}^pt^{-1}\;\m{d}t}^{1/p}
\\\ge \mathscr{K}\p{\tau,x}\bp{\int_{\p{0,\tau}}\p{t^{1-s}\tau^{-1}}^pt^{-1}\;\m{d}t}^{1/p} = \tau^{-s}\mathscr{K}\p{\tau,x}((1-s)p)^{-1/p}.  
\end{multline}
Upon taking the supremum in $\tau\in\p{0,\sig}$, we deduce that $\sb{x}^{\p{\sig}}_{s,\infty}\le (p\p{1-s})^{1/p} \sb{x}^{\p{\sig}}_{s,p}$.   On the other hand, if $1\le p< q<\infty$, then we can use estimate~\eqref{more lables are the coolest ever} to bound
    \begin{equation}\label{hydrogen}
        \sb{x}^{\p{\sig}}_{s,q}\le\big(\sb{x}_{s,\infty}^{\p{\sig}}\big)^{\f{q-p}{q}}\big(\sb{x}_{s,p}^{\p{\sig}}\big)^{\f{p}{q}}
        \le\big(p\p{1-s})^{1/p}\sb{x}^{\p{\sig}}_{s,p}\big)^{\f{q-p}{q}}\big(\sb{x}_{s,p}^{\p{\sig}}\big)^{\f{p}{q}}
        =\big(p\p{1-s}\big)^{\f{q-p}{pq}}\sb{x}_{s,p}^{\p{\sig}}.
    \end{equation}
The fifth item is proved.

\end{proof}

The next theorem shows that the spaces $\p{X_0,X_1}_{s,q}^{\p{\sig}}$ and $\p{X_0,X_1}_{s,q}$ are interpolation spaces in the sense of Definition~\ref{intermediate and interpolation spaces}.

\begin{thm}\label{truncated and bounded linear mappings}
Suppose that $\big(X_0,\sb{\cdot}_{X_0}\big)$,  $\big(X_1,\sb{\cdot}_{X_1}\big)$ and $\big(Y_0,\sb{\cdot}_{Y_0}\big)$,  $\big(Y_1,\sb{\cdot}_{Y_1}\big)$ are two pairs of weakly compatible seminormed spaces. Suppose that  $T:\Sigma\p{X_0,X_1}\to\Sigma\p{Y_0,Y_1}$ is a linear mapping with the following property: for $i\in\cb{0,1}$ there exist $c_i\in\R^+$ such that for all $x\in X_i$ we have $Tx_i\in Y_i$ and $\sb{Tx}_{Y_i}\le c_i\sb{x}_{X_i}$.  Then for all $s\in\p{0,1}$, $\sig \in (0,\infty]$, and $1\le q\le\infty$ we have that $T\p{X_0,X_1}_{s,q}^{\p{\sig}}\subseteq\p{Y_0,Y_1}^{\p{\sig}}_{s,q}$; moreover, for  $x\in\p{X_0,X_1}^{\p{\sig}}_{s,q}$ we have the estimate $\sb{Tx}^{\p{\sig}}_{s,q}\le c_0^{1-s}c_1^s\sb{x}^{(\sig{c_1}/{c_0})}_{s,q}$.
\end{thm}
\begin{proof}
We only present the proof for $\sigma <\infty$, as the proof with $\sigma =\infty$ follows similarly.  Let $x\in\p{X_0,X_1}^{\p{\sig}}_{s,q}$. Proposition~\ref{truncated interpolation spaces are intermediate} implies that $x\in\Sigma\p{X_0,X_1}$. Let $\p{x_0,x_1}\in X_0\times X_1$ be a decomposition of $x$, i.e. $x=x_0+x_1$. Then for $t\in\p{0,\sig}$ we may bound $\mathscr{K}\p{t,Tx}\le c_0\sb{x_0}_0+c_1t\sb{x_1}_1$. Taking the infimium over all such decompositions and multiplying by $t^{-s}$ yields the bound $t^{-s}\mathscr{K}\p{t,Tx}\le c_0t^{-s}\mathscr{K}\p{t{c_1}{c_0}^{-1},x}$. In the case $q = \infty$ we take the supremum over $t\in\p{0,\sig}$, and in the case $q< \infty$ we take the $q^{\m{th}}$ power, integrate, and employ a change of variables; in either case we arrive at the bound: $
\sb{Tx}_{s,q}^{\p{\sig}}\le c_0^{1-s}c_1^s\sb{x}_{s,q}^{(\sig{c_1/c_0})}$.
\end{proof}
We can also quantify the annihilators of the $K$-methods' interpolation spaces.
\begin{prop}[Annihilators]\label{kernels prop}
Let $\p{X_0,\sb{\cdot}_0}$ and $\p{X_1,\sb{\cdot}_1}$ be a pair of weakly compatible seminormed spaces.
Then the following hold for $s\in\p{0,1}$, $1\le q\le\infty$, and $\sig\in(0,\infty]$:
\begin{enumerate}
    \item $\mathfrak{A}\p{\p{X_0,X_1}_{s,q}^{\p{\sig}}}=\mathfrak{A}\p{\Sigma\p{X_0,X_1}}\supseteq\Sigma\p{\mathfrak{A}\p{X_0},\mathfrak{A}\p{X_1}}$.
    \item If strong compatibility holds (see Definition~\ref{admissable semi normed spaces}) or $\Delta\p{X_0,X_1}$ is semi-Banach, then the latter inclusion is an equality.
\end{enumerate}
\end{prop}
\begin{proof}
The first item follows from Propositions  \ref{basic properties of K for seminormed}, \ref{truncated interpolation spaces are intermediate}, and \ref{inclusion relations of truncated interpolation spaces}. The second item follows from Proposition~\ref{kernel of the sum and intersection}.
\end{proof}

Finally, we can characterize these spaces with discrete seminorms.

\begin{prop}[Discrete seminorms]\label{discrete seminorm on truncated interpolation spaces}
Suppose that $\p{X_0,\sb{\cdot}_0}$ and $\p{X_1,\sb{\cdot}_1}$ are a weakly compatible pair of seminormed spaces, $\sig\in\R^+$, $s\in\p{0,1}$, $1 < r < \infty$, and $1\le q\le\infty$. The following hold:
\begin{enumerate}
    \item For $x\in\Sigma\p{X_0,X_1}$ we have that $x\in\p{X_0,X_1}_{s,q}^{\p{\sig}}$ if and only if $\cb{r^{sk}\mathscr{K}\p{\sig r^{-k},x}}_{k\in\N} \in \ell^q\p{\N;\R}$.
    In either case, we have the equivalence
    \begin{equation}
    \begin{cases}
    \left(\frac{sq \sigma^{sq}}{r^{sq}-1 } \right)^{1/q} \sb{x}^{\p{\sig}}_{s,q}\le\norm{\cb{r^{sk} \mathscr{K}\p{\sig r^{-k},x} }_{k\in\N}}_{\ell^q\p{\N;\R}}\le  r^s \left(\frac{sq \sigma^{sq}}{r^{sq}-1 } \right)^{1/q} \sb{x}^{\p{\sig}}_{s,q}&1\le q<\infty\\
    \p{\f{\sig}{r}}^s \sb{x}_{s,\infty}^{\p{\sig}}\le\norm{\cb{r^{sk}\mathscr{K}\p{\sig r^{-k},x}}_{k\in\N}}_{\ell^q\p{\N;\R}}\le\sig^s\sb{x}^{\p{\sig}}_{s,\infty}&q=\infty.
    \end{cases}
    \end{equation}
    \item For $x\in\Sigma\p{X_0,X_1}$ we have that $x\in\p{X_0,X_1}_{s,q}$ if and only if $\cb{r^{sk}\mathscr{K}\p{r^{-k},x}}_{k\in\Z} \in \ell^q\p{\Z;\R}$.
    In either case, we have the equivalence
    \begin{equation}
        \begin{cases}
        \left(\frac{sq} {r^{sq}-1 } \right)^{1/q} \sb{x}_{s,q}\le\norm{\cb{r^{sk}\mathscr{K}\p{r^{-k},x}}_{k\in\Z}}_{\ell^q\p{\Z;\R}}
        \le r^s\left(\frac{sq }{r^{sq}-1 } \right)^{1/q} \sb{x}_{s,q}&1\le q<\infty\\
        r^{-s}\sb{x}_{s,\infty}\le\norm{\cb{r^{sk}\mathscr{K}\p{r^{-k},x}}_{k\in\Z}}_{\ell^\infty\p{\Z;\R}}
        \le\sb{x}_{s,\infty}&q=\infty.
        \end{cases}
    \end{equation}
\end{enumerate}
\end{prop}
\begin{proof}
We write 
\begin{equation}
    \int_{(0,\sigma)} (t^{-s} \mathscr{K}(t,x))^q \f{1}{t}\;\m{d}t = \sum_{k\in \N} \int_{(\sigma r^{-k-1}, \sigma r^{-k})} (t^{-s} \mathscr{K}(t,x))^q \f{1}{t}\;\m{d}t =: \sum_{k\in \N} I_k.
\end{equation}
We then use Proposition~\ref{basic properties of K for seminormed} estimate 
\begin{multline}
\frac{1}{sq \sigma^{sq}} [r^{sk} \mathscr{K}(\sigma r^{-k},x)]^q [r^{sq}-1]
 = \mathscr{K}(\sigma r^{-k},x) \int_{(\sigma r^{-k-1}, \sigma r^{-k})}  t^{-sq-1}\m{d}t \ge I_k \\
\ge \mathscr{K}(\sigma r^{-k-1},x) \int_{(\sigma r^{-k-1} , \sigma r^{-k})}  t^{-sq-1}\m{d}t = \frac{r^{-sq}}{sq \sigma^{sq}} [r^{s(k+1)} \mathscr{K}(\sigma r^{-k-1},x)]^q [r^{sq}-1].   
\end{multline}
Plugging this in above then proves the first item when $\sigma,q < \infty$.  The other cases follow similarly.
\end{proof}

\subsection{Integration into seminormed spaces}\label{seminorm integration}

To study the $J$-method for interpolation of seminormed spaces, we develop the following variant of the Bochner integral for functions valued in seminormed spaces.  Simple functions and their integrals are defined as usual.

\begin{defn}[Simple functions]\label{simple functions defn} Let $\p{Y,\mathfrak{M},\mu}$ be a measure space and $\p{X,\sb{\cdot}}$ a seminormed space. We say that a function $s:Y\to X$ is a simple function if:
\begin{enumerate}
\item $s$ is measurable, i.e. $s^{-1}\p{U}\in\mathfrak{M}$ for all $U\subseteq X$ open.
    \item $\card\p{s\p{Y}}$ is finite, and so there exist $n\in\N^+$, $\cb{a_j}_{j=1}^n\subseteq X$, and a pairwise disjoint collection $\cb{E_j}_{j=1}^n\subseteq\mathfrak{M}$ such that $s=\sum_{j=1}^na_j\chi_{E_j}$.
    \item $s$ has finite support, i.e. $\forall\;j\in\cb{1,\dots,n}$ with $a_j\in X\setminus\cb{0}$, $\mu\p{E_j}<\infty$.
\end{enumerate}
The collection of $X$-valued simple functions over $Y$ is denoted $\m{simp}\p{Y;X}$. We define the functional $\mathcal{I}:\m{simp}\p{Y;X}\to X$ via $
\mathcal{I}\p{s}=\sum_{j=1}^na_j\mu\p{E_j} \text{ for }s=\sum_{j=1}^na_j\chi_{E_j}$.
\end{defn}

With the functional $\mathcal{I}$ in hand, we can define the integral as a set-valued map.

\begin{defn}[Strongly measurable and $X$-integrable]\label{defn of strongly measurable and integrable} Let $\p{X,\sb{\cdot}}$ be a seminormed space and $\p{Y,\mathfrak{M},\nu}$ be a measure space. We say that a function $f:Y\to X$ is strongly measurable if:
\begin{enumerate}
    \item $f$ is measurable in the sense  that $f^{-1}\p{U}\in\mathfrak{M}$ for all $U\subseteq X$ open.
    \item There exists a sequence $\cb{s_n}_{n\in\N}\subseteq\m{simp}\p{Y;X}$ such that  $\sb{s_n-f}\to 0$ $\nu-$a.e. as $n\to\infty.$
    \end{enumerate}
We say that a strongly measurable function $f:Y\to X$ is $X$-integrable if the sequence of simple functions from item $(2)$ above satisfies, in addition, $\int_{Y}\sb{f-s_n}\m{d}\nu \to0$ as $ n\to\infty$. The collection of $X$-integrable functions over $Y$ is denoted $\mathfrak{L}^1\p{Y,\nu;X}$.
We define the set-valued mapping $\int_{Y}\p{\cdot}\;\m{d}\nu:\mathfrak{L}^1\p{Y,\nu;X}\to 2^X$ via
\begin{multline}
\int_Yf\;\m{d}\nu=\cb{\ell\in X\;:\;\exists\cb{s_n}_{n\in\N}\subset\m{simp}\p{Y;X},\;s_n\to f\;\text{a.e.},\;\int_{Y}\sb{f-s_n}\;\m{d}\nu \to 0,\;\sb{\mathcal{I}\p{s_n}-\ell}\to0}.
\end{multline}
\end{defn}

One of the benefits of defining the integral as a set-valued map is that it allows us to avoid invoking completeness to guarantee  $\{\mathcal{I}\p{s_n}\}_{n \in \N}$ converges.  The trade-off is that it can be the case that $\int_Y f \m{d}\nu =\varnothing$.  Note, though, that in the event that $X$ is a Banach space, the integral is the singleton containing the usual Bochner integral of $f$.

We next record some simple properties of the mapping $\int_Y\p{\cdot}\m{d}\nu$.
\begin{prop}\label{the integral is here}
Let $\p{Y,\mathfrak{M},\nu}$ be a measure space and $\p{X,\sb{\cdot}}$ be a seminormed space over $\mathbb{K}\in\cb{\R,\C}$. Then, the following hold for all $f,g\in\mathfrak{L}^1\p{Y,\nu;X}$ and $\al\in\mathbb{K}$:
\begin{enumerate}
    \item If $\p{X,\sb{\cdot}}$ is semi-Banach, then $\int_{Y}f\;\m{d}\nu\neq\es$.
    \item If $\ell_0,\ell_1\in\int_{Y}f\;\m{d}\nu$, then $\ell_0-\ell_1\in\mathfrak{A}\p{X}$ and $\sb{\ell_0}=\sb{\ell_1}$.
    \item $\sb{\int_{Y}f\;\m{d}\nu} =\{\sb{\ell} \;:\; \ell \in \int_Y f\m{d}\nu \} \subseteq [0, \int_{Y}\sb{f}\;\m{d}\nu ]$.
    \item If $\es\neq\int_{Y}f\;\m{d}\nu$ and $\es\neq\int_{Y}g\;\m{d}\nu$, then $\int_{Y}\p{f+\al g}\;\m{d}\nu=\int_{Y}f\;\m{d}\nu + \al\int_{Y}g\;\m{d}\nu$
\end{enumerate}
\end{prop}
\begin{proof}
 
To prove the first item note that if $s \in \m{simp}\p{Y;X}$, then $\sb{s} \in \m{simp}\p{Y;\R}$ and $\sb{\mathcal{I}\p{s} } \le \int_Y \sb{s}\m{d}\nu$.  Then for $\{s_n\}_{n \in \N} \subseteq \m{simp}\p{Y;X}$ such that $\sb{s_n -f} \to 0$ a.e. and $\int_Y \sb{s_n-f}\m{d}\nu \to 0$ as $n \to \infty$ we have that $\{\mathcal{I}\p{s_n}\}_{n \in \N}$ is Cauchy in $X$ and so $\int_Y f \m{d}\nu \neq \varnothing$ by Lemma \ref{characterizations of completeness in seminormed spaces}.  This proves the first item.  If $\ell_0,\ell_1 \in \int_Y f \m{d}\nu$, then there exist $\{s_n^i\}_{n \in \N} \subseteq \m{simp}\p{Y;X}$ for $i \in \{0,1\}$ such that $\sb{s_n^i -f} \to 0$ a.e.,  $\int_Y \sb{s_n^i-f}\m{d}\nu \to 0$, and $\sb{\ell_i - \mathcal{I}\p{s_n^i}} \to 0$ as $n \to \infty$.  Then 
\begin{multline}
    \sb{\ell_0 - \ell_1} \le \sb{\ell_0 - \mathcal{I}\p{s_n^0}} + \sb{\mathcal{I}\p{s^0_n}-\mathcal{I}\p{s^1_n}}+\sb{\ell_1 - \mathcal{I}\p{s_n^1}}\\\textstyle
    \le\sb{\ell_0 - \mathcal{I}\p{s_n^0}} +\int_{Y}\sb{f-s_n^0}+\int_{Y}\sb{f-s_n^1} +\sb{\ell_1 - \mathcal{I}\p{s_n^1}}\to 0
\end{multline} as $n \to \infty$, and hence $\ell_0-\ell_1 \in\mathfrak{A}\p{X}$, which proves the second item.  For the third item consider $\ell \in \int_Y f\m{d}\nu$ and pick the approximation sequence $\{s_n\}_{n \in \N}$ as above.  Then we may estimate 
\begin{equation}
\sb{\ell} \le \sb{\ell- \mathcal{I}\p{s_n}} + \sb{\mathcal{I}\p{s_n}} \le \sb{\ell- \mathcal{I}\p{s_n}} + \int_Y \sb{s_n} \m{d}\nu    
\le \sb{\ell- \mathcal{I}\p{s_n}} + \int_Y \sb{s_n-f} \m{d}\nu  + \int_Y \sb{f} \m{d}\nu.
\end{equation}
and send $n \to \infty$ to arrive at the bound $\sb{\ell} \le \int_Y \sb{f}\m{d}\nu$.  This proves the third item. The fourth item is immediate from the second item and continuity of vector operators in a seminorm space.

\end{proof}
\subsection{$J$-method and equivalence with $K$-method}

With a notion of seminorm integration in hand, we now turn our attention to the development of the $J-$method of interpolation for seminormed spaces.

\begin{defn}[$J$-functional]\label{J functional for seminorm}
Suppose that $\p{X_0,\sb{\cdot}_0}$ and $\p{X_1,\sb{\cdot}_1}$ are a pair of weakly compatible seminormed spaces. We define the following functional on their intersection: $\mathscr{J}:\R^+\times\Delta\p{X_0,X_1}\to\R$ via $\mathscr{J}\p{t,x}=\max\cb{\sb{x}_0,t\sb{x}_1}$.
\end{defn}

Some simple properties of the $J$-functional are recorded in the next proposition.

\begin{prop}\label{simple properties of J-functional}
Suppose that $\p{X_0,\sb{\cdot}_0}$ and $\p{X_1,\sb{\cdot}_1}$ are a pair of weakly compatible seminormed spaces.  The following hold:
\begin{enumerate}
    \item For each $x\in\Delta\p{X_0,X_1}$, the mapping $\R^+\ni t\mapsto\mathscr{J}\p{t,x}\in\R$ is convex.
    \item For any $t,s\in\R^+$ we have the bounds  $\min\cb{1,t/s}\mathscr{J}\p{s,\cdot}\le\mathscr{J}\p{t,\cdot}\le\max\cb{1,t/s}\mathscr{J}\p{s,\cdot}$.
    \item For any $t,s\in\R^+$ we have the inequality $\mathscr{K}\p{t,\cdot}\le\min\cb{1,t/s}\mathscr{J}\p{s,\cdot}$.
\end{enumerate}
\end{prop}
\begin{proof}
These are immediate from the definition $\mathscr{J}$.
\end{proof}

We now define the $J$-method of interpolation.

\begin{defn}[$J$-method of interpolation]\label{J-method of interpolation}
Suppose that $\p{X_0,\sb{\cdot}_0}$ and $\p{X_1,\sb{\cdot}_1}$ are a pair of weakly compatible seminormed spaces.  Recall that $\mu$, as defined in Section~\ref{notation stuff}, denotes Haar measure on $(0,\infty)$.  For $x\in\Sigma\p{X_0,X_1}$ we define the decomposition set of $x$ via
\begin{equation}\label{continuous decomposition set}
    \mathcal{D}\p{x}=\cb{u\in\mathfrak{L}^1\p{\R^+,\mu;\Sigma\p{X_0,X_1}}\;:\;u\p{t}\in\Delta\p{X_0,X_1}\;x\in\int_{\R^+}u\p{t}t^{-1}\m{d}t},
\end{equation}
where $\mathfrak{L}^1$ is as in Definition \ref{defn of strongly measurable and integrable}.  For $s\in\p{0,1}$ and $1\le q\le \infty$ we define $\sb{\cdot}_{s,q,\mathscr{J}}:\Sigma\p{X_0,X_1}\to\sb{0,\infty}$ via
\begin{equation}
  \sb{x}_{s,q,\mathscr{J}}= \begin{cases}
        \inf\cb{\p{\int_{\R^+}\p{t^{-s}\mathscr{J}\p{t,u\p{t}}}^qt^{-1}\;\m{d}t}^{1/q}\;:\;u\in\mathcal{D}\p{x}}&1\le q<\infty\\
        \inf\cb{\m{esssup}_{t\in\R^+}t^{-s}\mathscr{J}\p{t,u\p{t}}\;:\;u\in\mathcal{D}\p{x}}&q=\infty,
    \end{cases}
\end{equation}
with the usual understanding that $\inf\es=\infty$. The subspace of $\Sigma\p{X_0,X_1}$ on which $\sb{\cdot}_{s,q,\mathscr{J}}$ is finite is denoted $\p{X_0,X_1}_{s,q,\mathscr{J}}$, and we endow this space with the seminorm $\sb{\cdot}_{s,q,\mathscr{J}}$.
\end{defn}

We will now show that the $K$-method and the $J$-method give the same interpolation spaces. We need a seminormed space version of the so called `fundamental lemma of interpolation theory' (for the normed space version, see for instance Lemma 3.3.2 in~\cite{MR0482275}). The case for seminormed spaces is marginally more subtle, since $\sb{x}_{\Sigma}=0$ need not imply that there is a decomposition of $x=x_0+x_1$ where $\sb{x_0}_0=\sb{x_1}_1=0$. Our proof of the lemma is a slight generalization of the ideas in \cite{gustavsson}.

\begin{lem}[Fundamental lemma of interpolation theory]\label{fundamental lemma of interpolation theory}
Let $\p{X_0,\sb{\cdot}_0}$ and $\p{X_1,\sb{\cdot}_1}$ be a pair of weakly compatible seminormed spaces.  Suppose that $x\in\Sigma\p{X_0,X_1}$ satisfies
\begin{equation}\label{limit condition 1}
    \lim_{t\to0^+}\mathscr{K}\p{t,x}=0\quad\text{and}\quad\lim_{t\to\infty}t^{-1}\mathscr{K}\p{t,x}=0.
\end{equation}
Let $\ep\in\R^+$, $1 < r < \infty$, and suppose that $\varphi:\R^+\to\R^+$ is Lebesgue measurable and satisfies the following:
\begin{equation}\label{limit condition 2}
    \lim_{t\to0^+}\varphi\p{t}=0,\qquad\lim_{t\to\infty}t^{-1}\varphi\p{t}=0,\quad\text{and}\quad \inf\cb{\varphi\p{t}\;:\;r^{k-1}\le t\le r^{k+1}}=c_k\in\R^+ \text{ for }k \in \Z.
\end{equation}
Then there exists a strongly measurable $u:\R^+\to\Delta\p{X_0,X_1}$ with the following properties:
\begin{enumerate}
    \item $u\in\bigcap_{k\in\N^+}\mathfrak{L}^1\p{\p{r^{-k},r^k};\Sigma\p{X_0,X_1}}$, and for each $k \in \N^+$ we have that
    $
         \es\neq\int_{\p{r^{-k},r^k}}u\;\m{d}\mu\subseteq\Delta\p{X_0,X_1}+\mathfrak{A}\p{\Sigma\p{X_0,X_1}}.
    $
    \item For every sequence $\cb{\xi_k}_{k\in\Z}$ such that $\xi_k\in\int_{\p{r^{k-1},r^k}}u\;\m{d}\mu\neq\es$ for $k \in \Z$, we have that as $K\to\infty$ it holds
    $
    [-x+\sum_{k=-K}^K\xi_k]_\Sigma\to0.         
    $
    \item For a.e. $t\in\R^+$ it holds that $\mathscr{J}\p{t,u\p{t}}\le \p{\log\p{r}}^{-1}r(1+r)(\mathscr{K}\p{t,x}+\ep\varphi\p{t})$.
    \end{enumerate}
\end{lem}
\begin{proof}
Given $k\in\Z$, by the definition of the $\mathscr{K}$ functional we can find a decomposition $x=y_k+z_k$ with $\p{y_k,z_k}\in X_0\times X_1$ and
\begin{equation}\label{estimates}
    \sb{y_k}_0 + r^{k}\sb{z_k}_1\le\mathscr{K}(r^k,x) + \ep c_k.
\end{equation}
The assumptions \eqref{limit condition 1} and \eqref{limit condition 2} imply that
\begin{equation}\label{limit condition 3}
    \lim_{k\to-\infty}\sb{y_k}_0=0\quad\text{and}\lim_{k\to\infty}\sb{z_k}_1=0.
\end{equation}
Note that for each $k\in\Z$ we have that $\zeta_{k+1}=y_{k+1}-y_k=-z_{k+1}+z_k\in\Delta\p{X_0,X_1}$. This leads us to define $v:\R^+\to\Delta\p{X_0,X_1}$ via $v\p{t}= \sum_{k \in \Z} \zeta_{k} \chi_{[r^{k-1},r^k)}(t)$. It is clear that $v$ is strongly measurable, as it is a step function with countable image. For $k\in\N$ we have that $v$ restricted to $\p{r^{-k},r^k}$ is a simple function. Hence $v\in\bigcap_{k\in\N^+}\mathfrak{L}^1\p{\p{r^{-k},r^k},\mu;\Sigma\p{X_0,X_1}}$ and $\es\neq\int_{\p{r^{-k},r^k}}v\;\m{d}\mu\ni \mathcal{I}\big(v\vert_{\p{r^{-k},r^k}}\big)$. Moreover:
\begin{multline}\label{shows}
         \mathcal{I}\big(v\vert_{\p{r^{-k},r^k}}\big)=\textstyle{\sum_{j=-k+1}^k}\zeta_j\mu(\p{r^{j-1},r^j})\\=\log\p{r}\textstyle{\sum_{j=-k+1}^k}\p{y_{j}-y_{j-1}}=\log\p{r}\p{y_k-y_{-k}}=\log\p{r}\p{x+y_{-k}-z_k}.
\end{multline}
Notice that~\eqref{shows} paired with item $(2)$ of Proposition~\ref{the integral is here} reveal that \begin{equation}
    \int_{\p{r^{-k},r^k}}v\;\m{d}\mu=\log\p{r}\p{x+y_{-k}-z_k}+\mathfrak{A}\p{\Sigma\p{X_0,X_1}}\subseteq\Delta\p{X_0,X_1}+\mathfrak{A}\p{\Sigma\p{X_0,X_1}}.
\end{equation}
Now, for $k\in\Z$, we have that $v\vert_{\p{r^{k-1},r^k}}$ is a simple function. Hence, $\int_{\p{r^{k-1},r^k}}v\;\m{d}\mu=\log\p{r}\zeta_k+\mathfrak{A}\p{\Sigma\p{X_0,X_1}}$.
Pick any $\cb{\xi_k}_{k\in\Z}\subseteq\Sigma\p{X_0,X_1}$ with $\xi_k \in \int_{\p{r^{k-1},r^k}}v\;\m{d}\mu$. The previous fact paired with~\eqref{limit condition 3} yields
\begin{equation}
    \bsb{-\log\p{r}x+\textstyle{\sum_{j=-k}^k}\xi_j}_{\Sigma}=\log\p{r}\bsb{-x+\textstyle{\sum_{j=-k}^k}\zeta_j}_{\Sigma}=\sb{-y_{-k-1}-z_{k}}_{\Sigma}\le\sb{y_{-k-1}}_0+\sb{z_k}_1\to0\text{ as }k\to\infty.
\end{equation}
Thus $u=v/\log\p{r}$ satisfies items $(1)$ and $(2)$. To prove $(3)$, we take $t\in\R^+$ with $r^{k-1}\le t<r^k$, for some $k\in\Z$ and estimate (using again Proposition~\ref{basic properties of K for seminormed})
\begin{multline}
    \mathscr{J}\p{t,u\p{t}}\le\mathscr{J}\p{r^k,\zeta_k/\log\p{r}}=\log\p{r}^{-1}\max\cb{\sb{y_k-y_{k-1}}_0,r^k\sb{z_k-z_{k-1}}_1}
    \\\le\log\p{r}^{-1}\max\cb{\mathscr{K}(r^k,x)+\mathscr{K}(r^{k-1},x)+\ep c_k+\ep c_{k-1},\mathscr{K}(r^k,x)+r\mathscr{K}(r^{k-1},x)+\ep c_k+r\ep c_{k-1}}\\
    \le\log{r}^{-1}\max\cb{2\mathscr{K}\p{r^k,x}+2\ep\varphi\p{t},\p{1+r}\mathscr{K}\p{r^k,x}+\p{1+r}\ep\varphi\p{t}}\\
    \le\log\p{r}^{-1}r(1+r)\p{\mathscr{K}\p{t,x}+\ep\varphi\p{t}}.
\end{multline}
\end{proof}
We now give important sufficient conditions for the satisfaction of the hypotheses of Lemma~\ref{fundamental lemma of interpolation theory}.
\begin{lem}\label{real important suff right here}
Suppose that $\p{X_i,\sb{\cdot}_i}$, for $i\in\cb{0,1}$, are a weakly compatible pair of seminormed spaces, $s\in\p{0,1}$, and $1\le q\le\infty$. If $x\in\p{X_0,X_1}_{s,q}$, then $x$ satisfies equation~\eqref{limit condition 1}.
\end{lem}
\begin{proof}
We appeal to item $(5)$ of Proposition~\ref{inclusion relations of truncated interpolation spaces}; whence:
$\mathscr{K}\p{t,x}\le t^s\sb{x}_{s,\infty}\lesssim t^{s}\sb{x}_{s,q}\to0\text{ as }t\to0^+$
and
    $t^{-1}\mathscr{K}\p{t,x}\le t^{s-1}\sb{x}_{s,\infty}\lesssim t^{s-1}\sb{x}_{s,q}\to0\text{ as }t\to\infty$.
\end{proof}

With the lemmas in hand, we are now ready to prove the equivalence theorem for the non-truncated interpolation spaces.

\begin{thm}[Equivalence of $K$ method with $\sigma =\infty$ and the $J$ method]\label{equivalence theorem}
Let $\p{X_0,\sb{\cdot}_0}$ and $\p{X_1,\sb{\cdot}_1}$ be a pair of weakly compatible seminormed spaces.  Then for $s\in\p{0,1}$ and $1\le q\le \infty$, the spaces $\p{X_0,X_1}_{s,q}$ and $\p{X_0,X_1}_{s,q,\mathscr{J}}$ are identical as subspaces of $\Sigma\p{X_0,X_1}$, and the seminorms $\sb{\cdot}_{s,q}=\sb{\cdot}_{s,q}^{\p{\infty}}$ and $\sb{\cdot}_{s,q,\mathscr{J}}$ are equivalent.
\end{thm}
\begin{proof}
Suppose first that $x\in\p{X_0,X_1}_{s,q}$. Fix $\ep\in\R^+$. Then, since $\R^+\ni t\mapsto t^{-s}\mathscr{K}\p{t,x}\in\R$ belongs to $L^q\p{\R^+,\mu}$ (where again $\mu$ is defined in Section~\ref{notation stuff}), $x$ satisfies~\eqref{limit condition 1} from Lemma~\ref{fundamental lemma of interpolation theory}, thanks to Lemma~\ref{real important suff right here}. Define $\varphi : \R^+\to \R^+$ via $\varphi\p{t}=t^s\big(1 +\p{\log\p{t}}^2\big)^{-1}$.
and observe that $\varphi$ satisfies~\eqref{limit condition 2}. Thus, we can apply Lemma~\ref{fundamental lemma of interpolation theory} to obtain a strongly measurable function $u:\R^+\to\Delta\p{X_0,X_1}$  satisfying items $(1)$, $(2)$, and $(3)$ from the lemma.

We first show that $u\in\mathcal{D}\p{x}$, which amounts to proving that $u\in\mathfrak{L}^1\p{\R^+,\mu;X}$ and $x\in\int_{\R^+}u\;\m{d}\mu$. Since Lemma~\eqref{fundamental lemma of interpolation theory} tells us that we may take $u$ a step function, there is a natural choice of simple functions to attempt to satisfy Definition~\ref{defn of strongly measurable and integrable}. For $k\in\Z$ and $t\in[2^{k-1},2^k)$, there is some $\xi_k\in\Delta\p{X_0,X_1}$ such that $u\p{t}=\xi_k$.  Then for each  $n\in\N^+$ we define  $s_n=\sum_{k=-n}^n\xi_k\chi_{[2^{k-1},2^k)}\in\m{simp}\p{\R^+;\Sigma\p{X_0,X_1}}$. It is clear that $s_n\to u$ everywhere as $n\to\infty$.  Also, according to Proposition~\ref{simple properties of J-functional} and item $(3)$ from Lemma~\ref{fundamental lemma of interpolation theory}, we may bound
\begin{multline}\label{q}
    \int_{\R^+}\sb{u\p{t}-s_n\p{t}}_{\Sigma}\;\m{d}\mu\p{t}=\int_{\R^+\setminus\p{2^{-n-1},2^n}}\sb{u\p{t}}_{\Sigma}\f{1}{t}\;\m{d}t\le\int_{\R^+\setminus\p{2^{-n-1},2^n}}\min\cb{1,t^{-1}}\mathscr{J}\p{t,u\p{t}}t^{-1}\;\m{d}t\\
    \le6\p{\log\p{2}}^{-1}\ssb{\int_{\R^+\setminus\p{2^{-n-1},2^n}}\mathscr{K}\p{t,x}\min\cb{1,t^{-1}}t^{-1}\;\m{d}t+\ep\int_{\R^+\setminus\p{2^{-n-1},2^n}}\min\cb{1,t^{-1}} \varphi(t)\;t^{-1}\m{d}t}.
\end{multline}
To show that the right hand side of~\eqref{q} tends to zero as $n\to\infty$, it suffices to show that both integrands are integrable over $\R^+$. This is clear for the latter term involving $\varphi$. To handle the former, we use H\"{o}lder's inequality to bound
\begin{equation}
    \int_{\R^+}\mathscr{K}\p{t,x}\min\cb{1,t^{-1}}t^{-1}\;\m{d}t\le\sb{x}_{s,q}\begin{cases}
    \p{\int_{\R^+}\min\cb{t^{sp},t^{p\p{s-1}}}t^{-1}\;\m{d}t}^{1/p}&1<q\le\infty\\
    \sup\cb{\min\cb{t^s,t^{s-1}}\;:\;t\in\R^+}&q=1
    \end{cases}<\infty,
\end{equation}
where $p=q\p{q-1}^{-1}$.  Hence, $u$ is $\Sigma\p{X_0,X_1}$-integrable, with $x\in\int_{\R^+}u\;\m{d}\mu$; indeed, item $\p{2}$ in Lemma~\ref{fundamental lemma of interpolation theory} implies the limit: $\sb{x-\mathcal{I}\p{s_n}}_{\Sigma}\to0\quad\text{as}\quad n\to\infty$ holds.  Finally, we check that $x\in\p{X_0,X_1}_{s,q,\mathscr{J}}$. If $1\le q<\infty$, then again we use Lemma~\ref{fundamental lemma of interpolation theory} to see that
\begin{equation}
    \sb{x}_{s,q,\mathscr{J}}\le\bp{\int_{\R^+}\p{t^{-s}\mathscr{J}\p{t,u\p{t}}}^qt^{-1}\;\m{d}t}^{1/q}\le6\log\p{2}^{-1}\bsb{\sb{x}_{s,q}+\ep\bp{\int_{\R^+}\p{1+\log^2\p{t}}^{-q}t^{-1}\;\m{d}t}^{1/q}}.
\end{equation}
The integral on the right hand side is finite. As $\ep\in\R^+$ was chosen arbitrarily, we can let $\ep\to0^+$ and see that $\p{X_0,X_1}_{s,q}\emb\p{X_0,X_1}_{s,q,\mathscr{J}}$.  The case for $q=\infty$ is proved in the same way.

On the other hand, let $x\in\p{X_0,X_1}_{s,q,\mathscr{J}}$. Then, there is some $u\in\mathcal{D}\p{x}$ by hypothesis. Then for $t\in\R^+$ we use Proposition~\ref{simple properties of J-functional} to bound 
\begin{equation}\label{ww}
    \mathscr{K}\p{t,x}\le\int_{\R^+}\mathscr{K}\p{t,u\p{\tau}}\tau^{-1}\;\m{d}\tau\le\int_{\R^+}\min\cb{1,t\tau^{-1}}\mathscr{J}\p{\tau,u\p{\tau}}\tau^{-1}\;\m{d}\tau.
\end{equation}
In the case that $1\le q<\infty$, we use~\eqref{ww} and Hardy's inequalities (see Lemma~\ref{truncated hardy inequality}) to estimate
\begin{multline}
    \sb{x}_{s,q}\le\bp{\int_{\R^+}\bp{t^{-s}\int_{\R^+}\min\cb{1,t\tau^{-1}}\mathscr{J}\p{\tau,u\p{\tau}}\tau^{-1}\;\m{d}\tau}^qt^{-1}\;\m{d}t}^{1/q}\\
    \le\bp{\int_{\R^+}\bp{t^{-s}\int_{\p{0,t}}\mathscr{J}\p{\tau,u\p{\tau}}\tau^{-1}\;\m{d}\tau}^qt^{-1}\;\m{d}t}^{1/q}+\bp{\int_{\R^+}\bp{t^{-s}\int_{\p{t,\infty}}t\tau^{-1}\mathscr{J}\p{\tau,u\p{\tau}}\tau^{-1}\;\m{d}\tau}^q\;t^{-1}\m{d}t}^{1/q}
    \\
    \le\bp{s^{-1}+\p{1-s}^{-1}}\bp{\int_{\R^+}\p{t^{-s}\mathscr{J}\p{t,u\p{t}}}^qt^{-1}\;\m{d}t}^{1/q}.
\end{multline}
Taking the infimum over all $u\in\mathcal{D}\p{x}$ gives the case for $1\le q<\infty$. For the case $q=\infty$, we consider some $t\in\R^+$ and use~\eqref{ww}:
\begin{multline}
    t^{-s}\mathscr{K}\p{t,x}\le\int_{\R^+}\min\cb{t^{-s},t^{1-s}\tau^{-1}}\tau^{-s}\mathscr{J}\p{\tau,u\p{\tau}}\tau^{s-1}\;\m{d}\tau\\\le\m{esssup_{\tau\in\R^+}}\tau^{-s}\mathscr{J}\p{\tau,u\p{\tau}}\int_{\R^+}\min\cb{t^{-s}\tau^s,t^{1-s}\tau^{s-1}}\tau^{-1}\;\m{d}\tau.
\end{multline}
Notice first that $\int_{\R^+}\min\cb{t^{-s}\tau^s,t^{1-s}\tau^{s-1}}\tau^{-1}\;\m{d}\tau=\int_{\R^+}\min\cb{\tau^s,\tau^{s-1}}\tau^{-1}\;\m{d}\tau<\infty$. Taking the supremum over $t\in\R^+$, and then the infimum over all $u\in\mathcal{D}\p{x}$ show that $\p{X_0,X_1}_{s,\infty,\mathscr{J}}\emb\p{X_0,X_1}_{s,\infty}$.
\end{proof}

Next, we show that we have a discrete characterization of the seminorm on $\p{X_0,X_1}_{s,q,\mathscr{J}}$.

\begin{prop}[Discrete characterization of the $J$-method]\label{discrete seminorm on J-space} 
Let $\p{X_0,\sb{\cdot}_0}$ and $\p{X_1,\sb{\cdot}_1}$ be a pair of weakly compatible seminormed spaces.  For $x\in\Sigma\p{X_0,X_1}$ we define the discrete decomposition set of $x$ as
\begin{equation}\label{dis dec set}
    \tilde{\mathcal{D}}\p{x}=\cb{\cb{\xi_k}_{k\in\Z}\subseteq\Delta\p{X_0,X_1}\;:\;\textstyle{\sum_{k\in\Z}}\sb{\xi_k}_{\Sigma}<\infty,\;\lim_{K\to\infty}\big[-x+\textstyle{\sum_{k=-K}^K}\xi_k\big]_{\Sigma}=0}.
\end{equation}
Then for each $r\in (1,\infty)$
there is a constant $c\in\R^+$ such that for all $x\in\Sigma\p{X_0,X_1}$,
    \begin{equation}\label{discrete seminorm on J-space 0}
        c^{-1}\sb{x}_{s,q,\mathscr{J}}\le\inf\cb{\snorm{\big\{r^{sk}\mathscr{J}(r^{-k},\xi_k)\big\}_{k\in\Z}}_{\ell^q\p{\Z;\R}}\;:\;\cb{\xi_k}_{k\in\Z}\in\tilde{\mathcal{D}}\p{x}}\le c\sb{x}_{s,q,\mathscr{J}}.
    \end{equation}
\end{prop}
\begin{proof}
Let $x\in\p{X_0,X_1}_{s,q,\mathscr{J}}$ and $\ep\in\R^+$.  Again, we take $\varphi : \R^+ \to\R^+$ to be defined as in the proof of Theorem~\ref{equivalence theorem}. By Theorem~\ref{equivalence theorem} we have that $x\in\p{X_0,X_1}_{s,q}$ with $\sb{x}_{s,q}\le c\sb{x}_{s,q\mathscr{J}}$ for some $c$ depending only on $s$ and $q$. Thus, we can apply Lemma~\ref{fundamental lemma of interpolation theory} and then Theorem~\ref{equivalence theorem} again to find a step function $u:\R^+\to\Delta\p{X_0,X_1}$ with $u\in\mathcal{D}\p{x}$, obeying the bound
\begin{equation}\label{use me 1}
    \begin{cases}
    \p{\int_{\R^+}\p{t^{-s}\mathscr{J}\p{t,u\p{t}}}^qt^{-1}\;\m{d}t}^{1/q}&1\le q<\infty\\
    \m{esssup}_{t\in\R^+}\cb{t^{-s}\mathscr{J}\p{t,u\p{t}}\;:\;t\in\R^+}&q=\infty
    \end{cases}
    \le r(1+r)\log\p{r}^{-1}\ssb{\sb{x}_{s,q}+C\ep},
\end{equation}
where $C$ is a constant depending on $\varphi$, $s$, and $q$. Moreover, there is a sequence $\cb{\xi_k}_{k\in\Z}\subseteq\Delta\p{X_0,X_1}$ such that $u\p{t}=\xi_k$ for $t\in[r^{k-1},r^k)$, and 
\begin{equation}\label{use me 2}
    \int_{\R^+}\sb{u\p{t}}_{\Sigma}\;\m{d}\mu=\log\p{r}\sum_{k\in\Z}\sb{\xi_k}_{\Sigma}<\infty\quad\text{and}\quad x\in\int_{\R^+}u\;\m{d}\mu.
\end{equation}
Finally, item $(2)$ in Lemma~\ref{fundamental lemma of interpolation theory} implies that $\lim_{K\to\infty}[-x+\log\p{r}\sum_{k=-K}^K\xi_k]_{\Sigma}=0$. Thus $\cb{\log\p{r}\xi_k}_{k\in\Z}\in\tilde{\mathcal{D}}\p{x}$. Proposition~\ref{simple properties of J-functional} provides a constant $\tilde{c}$, depending on $s$, $q$, and $r$ such that
\begin{equation}\label{use me 3}
    \snorm{\big\{r^{sk}\mathscr{J}(r^{-k},\log\p{r}\xi_k)\big\}_{k\in\Z}}_{\ell^q\p{\Z}}\le\tilde{c}\begin{cases}
    \p{\int_{\R^+}\p{t^{-s}\mathscr{J}\p{t,u\p{t}}}^q\f{1}{t}\;\m{d}t}^{1/q}&1\le q<\infty\\
    \m{esssup}\cb{t^{-s}\mathscr{J}\p{t,u\p{t}}\;:\;t\in\R^+}&q=\infty.
    \end{cases}
\end{equation}
Together, \eqref{use me 1}, \eqref{use me 2}, and \eqref{use me 3} imply
\begin{equation}
    \inf\cb{\snorm{\big\{r^{sk}\mathscr{J}(r^{-k},\zeta_k)\big\}_{k\in\Z}}_{\ell^q\p{\Z;\R}}\;:\;\cb{\zeta_k}_{k\in\Z}\in\tilde{\mathcal{D}}\p{x}}\le c'\sb{x}_{s,q,\mathscr{J}}+C\ep
\end{equation}
for all $\ep\in\R^+$. Hence, the second inequality of \eqref{discrete seminorm on J-space 0} is proved.

Now suppose that $x\in\Sigma\p{X_0,X_1}$ is such that $\es\neq\tilde{\mathcal{D}}\p{x}$ and choose $\cb{\xi_k}_{k\in\Z}\in\tilde{\mathcal{D}}\p{x}$. We define $u:\R^+\to\Delta\p{X_0,X_1}$ via $
     u\p{t}=\log\p{r}^{-1}\sum_{k \in \Z} \xi_k \chi_{[r^{k-1},r^k) }(t)$.
For $n\in\N^+$ we take $s_n=u\vert_{\p{r^{-n},r^n}}$. Then $s_n\in\m{simp}\p{\R^+;\Sigma\p{X_0,X_1}}$ and $s_n\to u$ everywhere as $n\to\infty$, which shows $u$ to be strongly measurable. The condition $\sum_{k\in\Z}\sb{\xi_k}_{\Sigma}<\infty$ easily implies that $\int_{\R^+}\sb{s_n-u}_{\Sigma}\;\m{d}\mu\to0$ as $n\to\infty$. Then $u$ is $\Sigma\p{X_0,X_1}$-integrable.  Moreover, $x\in\int_{\R^+}u\;\m{d}\mu$, as the condition $\lim_{K\to\infty}[-x+\sum_{k=-K}^K\xi_k]_{\Sigma}=0$ implies that $\lim_{n\to\infty}\sb{-x+\mathcal{I}\p{s_n}}=0$.  Hence, $\mathcal{D}\p{x}\neq\es$. Using Proposition~\ref{simple properties of J-functional} once more, we see that
\begin{equation}
    \snorm{\big\{r^{sk}\mathscr{J}(r^{-k},\xi_k)\big\}_{k\in\Z}}_{\ell^q\p{\Z}}\ge c
    \begin{cases}
    \p{\int_{\R^+}\p{t^{-s}\mathscr{J}\p{t,u\p{t}}}^q\f{1}{t}\;\m{d}t}^{1/q}&1\le q<\infty\\
    \m{esssup}\cb{t^{-s}\mathscr{J}\p{t,u\p{t}}\;:\;t\in\R^+}&q=\infty
    \end{cases} 
    = c \sb{x}_{s,q,\mathscr{J}},
\end{equation} 
for some constant $c$ depending on $r,s,q$.  This implies the first inequality in \eqref{discrete seminorm on J-space 0}, and the proof is complete.
\end{proof}

As a corollary to the results in this subsection we will prove that the well-known reiteration theorem also holds in this seminormed setting under additional hypotheses. First, we require a brief quantitative lemma.

\begin{lem}\label{quantitative lemma}
Let $\p{X_0,\sb{\cdot}_0}$, $\p{X_1,\sb{\cdot}_1}$ be a pair of weakly compatible pair of seminormed spaces, $s\in\p{0,1}$, and $1\le q\le\infty$. If $t\in\R^+$ and $x\in\Delta(X_0,X_1)$ then we have the bound
\begin{equation}
    \sb{x}_{s,q}\le(1-s)^{1-1/q}(1/s+1/(1-s))t^{-s}\mathscr{J}\p{t,x}.
\end{equation}
\end{lem}
\begin{proof}
Using the bound from equation~~\eqref{hydrogen} and then item (3) of Proposition~\ref{simple properties of J-functional}, we simply compute:
\begin{equation}
    (1-s)^{-1+1/q}\sb{x}_{s,q}\le\sb{x}_{s,1}=\int_{\R^+}\tau^{-s}\mathscr{K}(\tau,x)\tau^{-1}\;\m{d}\tau\le\mathscr{J}(t,x)\int_{\R^+}\tau^{-s-1}\min\{1,\tau t^{-1}\}\;\m{d}\tau.
\end{equation}
\end{proof}
We now present the proof of the reiteration theorem. The justification is marginally more subtle than perhaps one would initially expect: we use crucially that the spaces are complete and are compatible in a sense stronger than the weak compatibility condition.
\begin{thm}[Reiteration]\label{reiteration theorem}
Let $\p{X_0,\sb{\cdot}_0}$ and $\p{X_1,\sb{\cdot}_1}$ be a pair of complete seminormed spaces such that either $\Delta\p{X_0,X_1}$ is complete or strong compatibility is satisfied. Set also $0\le s_0,s_1\le1$, $0<r<1$, $s=(1-r)s_0+rs_1$, and $1\le q,q_0,q_1\le\infty$. We then have the following reiteration formula:
\begin{equation}
    (X_{s_0,q_0},X_{s_1,q_1})_{r,q}=(X_0,X_1)_{s,q}\text{ for }X_{s_i,q_i}=\begin{cases}
    (X_0,X_1)_{s_i,q_i}&\text{if }s_i\not\in\cb{0,1}\\
    X_{s_i}&\text{if }s_i\in\cb{0,1}
    \end{cases},\;i\in\cb{0,1}.
\end{equation}
\end{thm}
\begin{proof}
 By the symmetry $(A_0,A_1)_{\vartheta,p}=(A_1,A_0)_{1-\vartheta,p}$ for $A_0$, $A_1$ weakly compatible seminormed spaces, $0<\vartheta<1$, $1\le p\le\infty$, it is sufficient to consider the case $0\le s_0<s_1<1$.

Let $\mathscr{K}$, $\mathscr{J}$ and $\Bar{\mathscr{K}}$, $\Bar{\mathscr{J}}$ denote the $K$ and $J$ functionals associated to the weakly compatible couples $X_0,X_1$ and $X_{s_0,q_0},X_{s_1,q_1}$, respectively. $\sb{\cdot}_{s,q}$ will denote the seminorm on $(X_0,X_1)_{s,q}$ and $\sb{\cdot}_{r,q}^{\--}$ will denote the seminorm on $(X_{s_0,q_0},X_{s_1,q_1})_{r,q}$.

Suppose that $x\in(X_{s_0,q_0},X_{s_1,q_1})_{r,q}$ and decompose $x=x_0+x_1$ for $x_i\in X_{s_i,q_i}$. In the event that $0<s_0$ we use the inclusion estimate after equation~\eqref{more lables are the coolest ever} to bound for $t\in\R^+$:
\begin{multline}
    \mathscr{K}\p{t,x}\le\mathscr{K}\p{t,x_0}+\mathscr{K}\p{t,x_1}\le t^{s_0}\sb{x_0}_{s_0,\infty}+t^{s_1}\sb{x_1}_{s_1,\infty}\\\textstyle\le\max_{i\in\cb{0,1}}\{\p{q_i(1-s_i)}^{1/q_i}\}t^{s_0}(\sb{x_0}_{s_0,q_0}+t^{s_1-s_0}\sb{x_1}_{s_1,q_1}),
\end{multline}
with the modification $(q_i(1-s_i))^{1/q_i}=1$ when $q_i=\infty$. On the other hand, if $s_0=0$ then we modify the above estimate with $\mathscr{K}\p{t,x_0}\le\sb{x_0}_0$. In either case we find there is a constant $c\in\R^+$ depending on $s_i$, $q_i$ for $i\in\cb{0,1}$ such that
\begin{equation}
    \sb{x}_{s,q}\le c\bp{\int_{\R^+}(t^{-r(s_1-s_0)}\Bar{\mathscr{K}}(t^{s_1-s_0},x))^qt^{-1}\;\m{d}t}^{1/q}=c(s_1-s_0)^{-1/q}\sb{x}^{\--}_{r,q},
\end{equation}
with the obvious modification for $q=\infty$. This argument justifies the inclusion $\p{X_{s_0,q_0},X_{s_1,q_1}}_{r,q}\emb(X_0,X_1)_{s,q}$.

Conversely, suppose that $x\in(X_0,X_1)_{s,q}$. By Proposition~\ref{discrete seminorm on J-space} there is at least one $\cb{\xi_k}_{k\in\Z}\in\tilde{\mathcal{D}}\p{x}$ (this latter set is defined in~\eqref{dis dec set}). For $t\in\R^+$ we claim that $\Bar{\mathscr{K}}(t,x)\le\sum_{k\in\Z}\Bar{\mathscr{K}}\p{t,\xi_k}$. It suffices to prove this claim under the assumption that the right hand side is finite. If this holds then we use the completeness of $\Sigma(X_{s_0,q_0},X_{s_1,q_1})$ (justified by Propositions~\ref{completeness of sum and intersection} and~\ref{completeness of truncated spaces}): there is $y\in\Sigma(X_{s_0,q_0},X_{s_1,q_1})$ such that $\Bar{\mathscr{K}}(t,y-\sum_{k=-K}^K\xi_k)\to0$ as $K\to\infty$. As $\Sigma(X_{s_0,q_0},X_{s_1,q_1})\emb\Sigma(X_0,X_1)$, we find $\sum_{k=-K}^K\xi_k\to x,y$ in $\Sigma(X_0,X_1)$, so $x-y\in\mathfrak{A}(\Sigma(X_0,X_1))$. By Propositions~\ref{kernel of the sum and intersection} and~\ref{kernels prop} we may equate
\begin{equation}
    \mathfrak{A}(\Sigma(X_0,X_1))=\Sigma(\mathfrak{A}(X_0),\mathfrak{A}(X_1))=\mathfrak{A}(\Sigma(X_{s_0,q_0},X_{s_1,q_1})).
\end{equation}
Hence $\Bar{\mathscr{K}}(t,x-y)=0$. For $K\in\N$ we then estimate:
\begin{equation}
    \textstyle\Bar{\mathscr{K}}(t,x)\le\Bar{\mathscr{K}}(t,x-y)+\Bar{\mathscr{K}}(t,y-\sum_{k=-K}^K\xi_k)+\sum_{k=-K}^K\Bar{\mathscr{K}}\p{t,\xi_k}.
\end{equation}
Sending $K\to\infty$ completes the proof of the claim.

With the claim in hand, we estimate the seminorm of $x$ in the space $(X_{s_0,q_0},X_{s_1,q_1})_{r,q}$ using first the discrete characterization of Proposition~\ref{discrete seminorm on truncated interpolation spaces}:
\begin{multline}\label{lithiummmm}
    \textstyle\sb{x}^{\--}_{r,q}\lesssim\snorm{\cb{2^{-rm}\Bar{\mathscr{K}}\p{2^m,x}}_{m\in\Z}}_{\ell^q}\le\snorm{\{\sum_{k\in\Z}2^{-rm}\Bar{\mathscr{K}}\p{2^m,\xi_k}\}_{m\in\Z}}_{\ell^q}\\\textstyle\le\snorm{\{\sum_{k\in\Z}\min\{2^{-rm},2^{(1-r)m-k}\}\Bar{\mathscr{J}}(2^k,\xi_k)\}_{m\in\Z}}_{\ell^q}.
\end{multline}
To each term in the innermost sum we next apply Lemma~\ref{quantitative lemma} with $t=2^{k/(s_1-s_0)}$. After a straightforward computation we arrive at the following inequality in the case $s_0>0$:
\begin{multline}
    \textstyle\Bar{\mathscr{J}}(2^k,\mathscr{\xi}_k)=\max\{\sb{\xi_k}_{s_0,q_0},2^k\sb{\xi_k}_{s_1,q_1}\}\\\textstyle\le\max_{i\in\cb{0,1}}\{(1-s_i)^{1-1/q_i}(1/s_i+1/(1-s_i))\}2^{-s_0k/(s_1-s_0)}\mathscr{J}(2^{k/(s_1-s_0)},\xi_k)\\\textstyle=c2^{-s_0k/(s_1-s_0)}\mathscr{J}(2^{k/(s_1-s_0)},\xi_k).
\end{multline}
In the case that $s_0=0$ the same argument gives the above inequality with $c=\max\{1,(1-s_1)^{1-1/q_1}(1/s_1+1/(1-s_1))\}$. In either case, we then incorporate this information into~\eqref{lithiummmm} and recall that $r=(s-s_0)/(s_1-s_0)$ in order to estimate
\begin{multline}
    \textstyle\snorm{\{\sum_{k\in\Z}\min\{2^{-rm},2^{(1-r)m-k}\}\Bar{\mathscr{J}}(2^k,\xi_k)\}_{m\in\Z}}_{\ell^q}\\\textstyle\le c\snorm{\{\sum_{k\in\Z}\min\{2^{-rm},2^{(1-r)m-k}\}2^{rk}2^{-sk/(s_1-s_0)}\mathscr{J}(2^{k/(s_1-s_0)},\xi_k)\}_{m\in\Z}}_{\ell^q}\\
    =\textstyle c\snorm{\{\sum_{j\in\Z}\min\{2^{-rj},2^{(1-r)j}\}2^{-s(m-j)/(s_1-s_0)}\mathscr{J}(2^{(m-j)/(s_1-s_0)},\xi_{m-j})\}_{m\in\Z}}_{\ell^q}\\\textstyle\le c(\sum_{j\in\Z}\min\{2^{-rj},2^{(1-r)j}\})\norm{\{2^{-sm/(s_1-s_0)}\mathscr{J}(2^{m/(s_1-s_0)})\}_{m\in\Z}}_{\ell^q}.
\end{multline}
Taking the infimum over all $\cb{\xi_k}_{k\in\Z}\in\tilde{\mathcal{D}}(x)$ and using finally Proposition~\ref{discrete seminorm on J-space} gives the remaining embedding.
\end{proof}

\subsection{Sum characterization of the truncated $K$-method}

The following theorem shows that the truncated spaces are the sum of the second factor and the $K$-method with $\sig=\infty$ space between the two factors. 

\begin{thm}[Sum characterization for truncated method]\label{label sum characterization of the truncated method}
Let $\p{X_0,\sb{\cdot}_0}$ and $\p{X_1,\sb{\cdot}_1}$ be a pair of weakly compatible seminormed spaces.  Suppose that $s\in\p{0,1}$, $\sig\in\R^+$, and $1\le q\le\infty$. Then we have the following equality of spaces and equivalence of seminorms: $\p{X_0,X_1}^{\p{\sig}}_{s,q}=\Sigma\big(\p{X_0,X_1}_{s,q},X_1\big)$.
\end{thm}
\begin{proof}
We begin by defining the functional $\tilde{\mathscr{K}}:\R^+\times\Sigma\big(\p{X_0,X_1}_{s,q},X_1\big)\to\R$ via
\begin{equation}
\tilde{\mathscr{K}}\p{t,x} =\inf\big\{\sb{\eta}_{s,q}+t\sb{\xi}_{1}\;:\;x=\eta+\xi,\;\p{\eta,\xi}\in\p{X_0,X_1}_{s,q}\times X_1\big\}.
\end{equation}
Note that for all $t\in\R^+$, the map $\tilde{\mathscr{K}}\p{t,\cdot}$ is an equivalent seminorm on $\Sigma\big(\p{X_0,X_1}_{s,q},X_1\big)$.

First suppose that $x\in\Sigma\big(\p{X_0,X_1}_{s,q},X_1\big) \subseteq \Sigma\p{X_0,X_1}$, where the latter inclusion follows from Proposition \ref{truncated interpolation spaces are intermediate}.  Pick  $y\in \p{X_0,X_1}_{s,q}$ and $z \in X_1$ such that $x=y+z$.  By Proposition~\ref{inclusion relations of truncated interpolation spaces} we have the bound
\begin{equation}
    \sb{x}^{\p{\sig}}_{s,q}\le\sb{y}^{\p{\sig}}_{s,q}+\sb{z}^{\p{\sig}}_{s,q}\le  \sb{y}_{s,q} + c_{s,q}\sig^{1-s}\sb{z}_1,\text{ for } 
    c_{s,q}=\begin{cases}
    q^{-1/q}\p{1-s}^{-1/q}&1\le q<\infty\\
    1&q=\infty.
    \end{cases}
\end{equation}
Thus, upon taking the infimum over all such decompositions of $x$, we arrive at the estimate
\begin{equation}
    \sb{x}^{\p{\sig}}_{s,q}\le\max\cb{1,c_{s,q}}\tilde{\mathscr{K}}\p{\sig^{1-s},x}.
\end{equation}
In particular, this implies that $\Sigma\big(\p{X_0,X_1}_{s,q},X_1\big) \subseteq \p{X_0,X_1}^{\p{\sig}}_{s,q}$.

On the other hand, suppose $x\in\p{X_0,X_1}_{s,q}^{\p{\sig}}$. Let $\ep\in\R^+$. For each $k\in\N$ we may then find $\p{a_k,b_k}\in X_0\times X_1$ with the following properties:
\begin{equation}\label{owo}
x=a_k+b_k,\text{ and }\sb{a_k}_0+\sig2^{-k}\sb{b_k}_1\le \mathscr{K}(\sig2^{-k},x)+\ep2^{-k}.\end{equation}
Set $\eta\in X_1$ via $\eta=b_0$.  Proposition~\ref{truncated interpolation spaces are intermediate} gives the bound 
\begin{equation}\label{bound 2}
    \sb{\eta}_{1}=\sb{b_0}_{1}\le\sig^{-1}\mathscr{K}\p{\sig,x}\le c\sb{x}^{\p{\sig}}_{s,q}
\end{equation}
for some $c$ depending on $s$, $q$, and $\sig$.   Note that~\eqref{owo} implies that for all $k\in\N$ we have $\xi_k=a_k-a_{k+1}=b_{k+1}-b_k\in\Delta\p{X_0,X_1}$. Hence, for $m\in\N$, we may use telescoping sums to compute
$\eta+\sum_{k=0}^{m}\xi_k = \eta + a_0 - a_{m+1} =x-a_{m+1}$.
Proposition~\ref{discrete seminorm on truncated interpolation spaces} provides a constant $c\in \R^+$ such that $\snorm{\big\{2^{sk}\mathscr{K}(\sig2^{-k},x)\big\}_{k\in\N}}_{\ell^q\p{\N}}\le c\sb{x}^{\p{\sig}}_{s,q}<\infty$,
which means that $\lim_{k\to\infty}\mathscr{K}\p{\sig2^{-k},x}=0$.  This and~\eqref{owo} imply that $\sb{a_{m+1}}_0\to 0$ as $m\to\infty$, and hence \begin{equation}\label{hereherherher}
    \lim_{m\to\infty}\big[x-\eta-\textstyle{\sum_{k=0}^m}\xi_k\big]_{\Sigma}\le\lim_{m\to\infty}\sb{a_{m+1}}_0=0.
\end{equation}
For $k\in\Z\setminus\N$ we set $\xi_k =0$.  Then \eqref{owo} implies that
\begin{equation}
    \sum_{k\in\Z}\sb{\xi_k}_{\Sigma} \le \sum_{k\in \N} \sb{a_k - a_{k+1} }_0 \le 2 \sum_{k \in \N} \sb{a_k}_0 \le
    2\sum_{k\in\N}\mathscr{K}(\sig2^{-k},x)+4\ep<\infty,
\end{equation}
where finiteness follows from the inclusion $x\in\p{X_0,X_1}_{s,q}$ thanks to Proposition \ref{discrete seminorm on truncated interpolation spaces} and H\"older's inequality (see the proof of Theorem~\ref{equivalence theorem}).  We deduce from this and~\eqref{hereherherher} that $\cb{\xi_k}_{k\in\Z}\in\tilde{\mathcal{D}}\p{x-\eta}$, where the latter set is defined in Proposition~\ref{discrete seminorm on J-space}.  

Next we again use \eqref{owo} and the fact that $\mathscr{K}(\cdot,x)$ is increasing to bound
\begin{equation}
\mathscr{J}(\sig2^{-k},\xi_k) =
    \max\big\{\sb{a_k-a_{k+1}},\sig2^{-k}\sb{b_{k+1}-b_{k}}\big\} 
\le 2\mathscr{K}(\sig2^{-k},x)+2^{-k+1}\ep  
\end{equation}
for $k \in \N$.  Since $\xi_k =0$ for $k \in \Z \backslash \N$ we have $\mathscr{J}\p{\sig2^{-k},\xi_k}=0$ in this case.  Combining these and using Proposition \ref{discrete seminorm on J-space}, we arrive at the bound
\begin{equation}\label{bound 3}
\sb{x-\eta}_{s,q}\le c \snorm{\big\{2^{sk}\mathscr{J}(\sig2^{-k},\xi_k)\big\}_{k\in\Z}}_{\ell^q\p{\Z}}
\le c\snorm{\big\{2^{sk}\mathscr{K}(\sig2^{-k},x)\big\}_{k\in\N}}_{\ell^q\p{\N}}+2c\ep
\le c\p{{\sb{x}^{\p{\sig}}_{s,q}+2\ep}}
\end{equation}
for a constant $c\in \R^+$ depending on $s,q$ and $\sigma$, and possibly increasing from line to line.  Combining~\eqref{bound 2} and~\eqref{bound 3} then yields the estimate $\tilde{\mathscr{K}}\p{1,x}\le\sb{x-\eta}_{s,q}+\sb{\eta}_{1}\le C\big(\sb{x}^{\p{\sig}}_{s,q}+2\ep\big)$
for every $\ep\in\R^+$ and some $C\in\R^+$ depending only on $s$, $q$, and $\sig$. Letting $\ep\to0^+$, we find that $ \p{X_0,X_1}^{\p{\sig}}_{s,q} \emb \Sigma\big(\p{X_0,X_1}_{s,q},X_1\big)$.
\end{proof}

As a corollary, we have the following useful density result.

\begin{prop}[Dense subspaces]\label{density in interpolation spaces}
Let $\p{X_0,\sb{\cdot}_0}$ and $\p{X_1,\sb{\cdot}_1}$ be a pair of weakly compatible seminormed spaces.  The following hold for $s\in\p{0,1}$, $\sig\in\R^+$, and $1\le q<\infty$:
\begin{enumerate}
    \item $\Delta\p{X_0,X_1}$ is dense in $\p{X_0,X_1}_{s,q}$.
    \item $X_1$ is dense in $\p{X_0,X_1}_{s,q}^{\p{\sig}}$.
\end{enumerate}
\end{prop}
\begin{proof}
For the first item we use the equivalence of the $K$ and $J$ methods from Theorem~\ref{equivalence theorem} and then the discrete characterization from Proposition~\ref{discrete seminorm on J-space}.  Indeed, for $x\in\p{X_0,X_1}_{s,q}$,  the discrete decomposition set $\tilde{\mathcal{D}}\p{x}$ is nonempty, and we may find $\cb{\xi_k}_{k\in\Z}\subseteq\Delta\p{X_0,X_1}$ such that
\begin{equation}\label{the stuff holds}
    \snorm{\big\{2^{sk}\mathscr{J}(2^{-k},\xi_k)\big\}_{k\in\Z}}_{\ell^q\p{\Z}}<\infty\text{ and }\lim_{n\to\infty}\big[-x+\textstyle{\sum_{k=-n}^n}\xi_k\big]_{\Sigma}=0.
\end{equation}
Applying the discrete $J$ method to $-x + \sum_{k=-n}^n \xi_k$, we find that
\begin{equation}
\big[-x+\textstyle{\sum_{k=-n}^n\xi_k}\big]_{s,q}\le C\snorm{\big\{2^{sk}\mathscr{J}(2^{-k},\xi_k)\big\}_{k\in\Z\setminus\cb{-n,\dots,n}}}_{\ell^q\p{\Z\setminus\cb{-n,\dots,n}}}
\end{equation}
for some $C \in \R^+$ some constant depending on $s,q$.  Since $q < \infty$, the right side of this inequality vanishes as $n \to \infty$. As it is the case $\sum_{k=-n}^n \xi_k \in \Delta\p{X_0,X_1},$ the first item is shown. 

To prove the second item we recall from Theorem~\ref{label sum characterization of the truncated method} that $\p{X_0,X_1}_{s,q}^{\p{\sig}}=\Sigma\big(\p{X_0,X_1}_{s,q},X_1\big)$.  Since  $X_1$ is dense in $X_1$ and $\Delta\p{X_0,X_1}$ is dense in $\p{X_0,X_1}_{s,q}$ by the first item, the density of $X_1$ in $\p{X_0,X_1}_{s,q}^{\p{\sig}}$ follows, and the second item is proved.
\end{proof}
\subsection{Examples of seminorm interpolation spaces}
Here we record a few examples of spaces obtained via seminorm interpolation.

\begin{exa}[Nesting of factors] 
Suppose that $\p{X_0,\sb{\cdot}_0}$ and $\p{X_1,\sb{\cdot}_1}$ are a pair of weakly compatible seminormed spaces, and let $\sigma \in \R^+$, $s \in (0,1)$, and $1 \le q \le \infty$.
\begin{enumerate}
    \item If $X_0\emb X_1$, then $\p{X_0,X_1}_{s,q}^{\p{\sigma}}=X_1$ (equivalent seminorms). Indeed, by item $(4)$ in Proposition~\ref{inclusion relations of truncated interpolation spaces}:
    \begin{equation}
        X_1\emb\p{X_0,X_1}_{s,q}^{\p{\sigma}}\emb\Sigma\p{X_0,X_1}\emb X_1.
    \end{equation}
     \item On the other hand, if $X_1\emb X_0$, then $\p{X_0,X_1}_{s,q}^{\p{\sigma}}=\p{X_0,X_1}_{s,q}$ (equivalent seminorm). Indeed, by Theorem~\ref{label sum characterization of the truncated method}:
     \begin{equation}
         \p{X_0,X_1}_{s,q}\emb\p{X_0,X_1}^{\p{\sigma}}_{s,q}=\Sigma\big(\p{X_0,X_1}_{s,q},X_1\big)\emb\Sigma\big(\p{X_0,X_1}_{s,q},\Delta\p{X_0,X_1}\big)\emb\p{X_0,X_1}_{s,q}.
     \end{equation}
\end{enumerate}
\end{exa}

\begin{exa}[Lebesgue spaces]\label{lorentz}
Let $\p{Y,\mathfrak{M},\mu}$ be a measure space, $\p{X,\norm{\cdot}}$ be a Banach space, and take $1\le p,q,r\le\infty$ and  $\sigma \in \R^+$, $s \in (0,1)$. Then
\begin{equation}
    \p{L^p\p{Y;X};L^r\p{Y;X}}_{s,q}^{\p{\sigma}}=\Sigma\p{L^{t,q}\p{Y;X},L^r\p{Y;X}},
\end{equation}
where the left factor on the right-hand-side sum is a Lorentz space and $1\le t\le\infty$ satisfies $\f{1}{t}=\f{1-s}{p}+\f{s}{r}$. This follows from the sum characterization in Theorem~\ref{label sum characterization of the truncated method} and the well-known characterization of Lorentz spaces as interpolation spaces (see, for instance, Theorem 1.18.6.1 in~\cite{MR503903}).
\end{exa}

Our next example, which is a slight modification of the previous one, introduces seminormed versions of Lebesgue and Lorentz spaces.  These spaces essentially consist of the classical spaces plus constants.  Our motivation for introducing this somewhat odd variant is that they appear naturally in several places later in the paper.  

\begin{exa}[Seminormed Lorentz spaces]\label{seminormed Lebesgue spaces}
Let $\p{X,\norm{\cdot}}$ be a Banach space and $\p{Y,\mathfrak{M},\mu}$ be a measure space such that $\mu(Y) = \infty$.  Let $1\le p<\infty$ and $1\le q\le\infty$ be such that if $p=1$ then $q=1$.  In this range, the Lorentz spaces $L^{p,q}(Y;X)$ are Banach spaces (when $p=1$, $1 < q \le \infty$ they only have quasinorms), and contain only the trivial constant function.  We define the seminormed Lorentz space
\begin{equation}
    \dot{L}^{p,q}(Y;X)=\{f\in L^1_{\m{loc}}\p{Y;X}\;:\;\exists\;c\in X,\;f-c\in L^{p,q}(Y;X)\}
\end{equation}
with seminorm $\sb{f}_{\dot{L}^{p,q}}=\inf\{\norm{f-c}_{L^{p,q}}\;:\;c\in X\}$.  Note that for each $f \in \dot{L}^{p,q}(Y;X)$ the constant $c \in X$ such that $f-c \in L^{p,q}(Y;X)$ is uniquely determined since the only constant in $L^{p,q}(Y;X)$ is $0$, and as such we have that $\sb{f}_{\dot{L}^{p,q}}= \norm{f-c}_{L^{p,q}}$.  If $p=q$ we write $\dot{L}^{p}\p{Y;X}$ in place of $\dot{L}^{p,p}\p{Y;X}$ for the seminormed Lebesgue spaces. 

Suppose now that  $1\le p_0,p_1<\infty$ and $1\le q_0,q_1\le\infty$ are such that $q_i = 1$ if $p_i =1$.  We claim that for $s\in(0,1)$, $\sigma \in \R^+$, $1/p=(1-s)/p_0+s/p_1$, and $1 \le q \le \infty$ we have the seminormed interpolation identities:
\begin{equation}
\begin{split}
    (\dot{L}^{p_0,q_0}(Y;X),\dot{L}^{p_1,q_1}(Y;X))_{s,q} & =\dot{L}^{p,q}(Y;X)  \text{ with equality of seminorms, and }   \\
    (\dot{L}^{p_0,q_0}(Y;X),\dot{L}^{p_1,q_1}(Y;X))_{s,q}^{\p{\sigma}} & = 
    \Sigma(\dot{L}^{p,q}(Y;X), \dot{L}^{p_1,q_1}(Y;X)).    
\end{split}
\end{equation}
The latter formula follows from the former and the sum characterization in Theorem~\ref{label sum characterization of the truncated method}, so we will only prove the former.  The proof of the former for standard Lorentz spaces can be found, for instance, in Theorems 1.18.6.1/2 of~\cite{MR503903}.  In proving this we let $\mathscr{K}$, $\dot{\mathscr{K}}$ denote the $K$-functionals associated to the couples $L^{p_0,q_0}(Y;X)$, $L^{p_1,q_1}(Y;X)$ and $\dot{L}^{p_0,q_0}(Y;X)$, $\dot{L}^{p_1,q_1}(Y;X)$, respectively. The seminorm on $(L^{p_0,q_0}(Y;X),L^{p_1,q_1}(Y;X))_{s,q}$ is $\sb{\cdot}_{s,q}$ while the seminorm on $(\dot{L}^{p_0,q_0}(Y;X),\dot{L}^{p_1,q_1}(Y;X))_{s,q}$ will be denoted $\sb{\cdot}_{s,q}^{\cdot}$.

Let $f\in\dot{L}^{p,q}(Y;X)$, $t\in \R^+$, and choose the unique $c\in X$ such that $f-c\in L^{p,q}(Y;X)$. If $f -c = g_0 + g_1$ for $g_i \in L^{p_i,q_i}(Y;X)$, then $\dot{g}_i \in \dot{L}^{p_i,q_i}(Y;X)$ and $\sb{g_i}_{\dot{L}^{p_i,q_i}} = \norm{g_i}_{L^{p_i,q_i}}$, and hence 
\begin{equation}
\dot{\mathscr{K}}(t,f) \le \sb{g_0 + c}_{\dot{L}^{p_0,q_0}} + t   \sb{g_1}_{\dot{L}^{p_1,q_1}} = \norm{g_0 }_{L^{p_0,q_0}} + t   \norm{g_1}_{L^{p_1,q_1}}  \Rightarrow \dot{\mathscr{K}}(t,f) \le \mathscr{K}(t,f-c).
\end{equation}
Similarly, if $f = g_0 + g_1$ for $g_i \in \dot{L}^{p_i,q_i}(Y;X)$, then there exist unique $c_i \in X$ such that $g_i - c_i \in L^{p_i,q_i}(Y;X)$ and $c_0 + c_1 = c$, which means that $f-c = (g_0 - c_0) + (g_1 - c_1)$ and
\begin{equation}
\mathscr{K}(t,f-c) \le \norm{g_0 - c_0}_{L^{p_0,q_0}} + t \norm{g_1 - c_1}_{L^{p_1,q_1}} = \sb{g_0 }_{\dot{L}^{p_0,q_0}} + t   \sb{g_1}_{\dot{L}^{p_1,q_1}} \Rightarrow 
\mathscr{K}(t,f-c) \le \dot{\mathscr{K}}(t,f).
\end{equation}
Thus, for $t \in \R^+$ we have that $\mathscr{K}(t,f-c)=\dot{\mathscr{K}}(t,f)$, and we deduce from this and the usual interpolation properties of Lebesgue and Lorentz spaces that $\sb{f}_{s,q}^{\cdot}=\sb{f-c}_{s,q}= \norm{f-c}_{L^{p,q}} = \sb{f}_{\dot{L}^{p,q}}$.  A similar argument proves the same identity for each $f \in (\dot{L}^{p_0,q_0}(Y;X),\dot{L}^{p_1,q_1}(Y;X))_{s,q},$ from which the claim follows.
\end{exa}

In our last example of this subsection we quantify a sense in which the space of functions of bounded mean oscillation is a substitute for the space of essentially bounded functions.

\begin{exa}[BMO and Lebesgue spaces]\label{BMO}
Recall that the space $\m{BMO}\p{\R^n;\mathbb{K}}$ consists of $f \in  L^1_{\loc}(\R^n;\mathbb{K})$ such that 
\begin{equation}
    \sb{f}_{\m{BMO}} = \sup_{Q}\f{1}{\Le^n\p{Q}}\int_{Q}\Big|f-\f{1}{\Le^n(Q)}\int_{Q}f\Big| <\infty,
\end{equation}
where the supremum is taken over cubes of the form $Q = \prod_{j=1}^n [a_j,a_j + \ell]$.  This only defines a seminormed space, as it is readily verified that the annihilator consists of all constant functions.

Let $1\le p<\infty$  and $\mathbb{K}\in\cb{\R,\C}$.  We claim that for all $s\in\p{0,1}$, $\sigma \in \R^+$,  and $1\le q\le\infty$ we have the formulae:
\begin{equation}\label{monkey}
    \p{L^p\p{\R^n;\mathbb{K}};\m{BMO}\p{\R^n;\mathbb{K}}}_{s,q}=\dot{L}^{r,q}\p{\R^n;\mathbb{K}}
\end{equation}
and 
\begin{equation}\label{monkey2}
    \p{L^p\p{\R^n;\mathbb{K}},\m{BMO}\p{\R^n;\mathbb{K}}}_{s,q}^{\p{\sigma}}=\Sigma(\dot{L}^{r,q}\p{\R^n;\mathbb{K}};\m{BMO}\p{\R^n;\mathbb{K}})=\Sigma\p{L^{r,q}\p{\R^n;\mathbb{K}},\m{BMO}\p{\R^n;\mathbb{K}}},
\end{equation}
for $r=p/(1-s)\in(1,\infty)$ and the dotted spaces as in Example~\ref{seminormed Lebesgue spaces}. 

We first remark how~\eqref{monkey2} will follow from~\eqref{monkey}. Given formula~\eqref{monkey} we may apply the sum characterization (Theorem~\ref{label sum characterization of the truncated method}) to obtain the first equality of equation~\eqref{monkey2}. The second equality then follows from the following constant shifting argument. For each $f\in\dot{L}^{r,q}\p{\R^n;\mathbb{K}}$ there is a unique $c\in\mathbb{K}$ such that $f-c\in L^{r,q}\p{\R^n;\mathbb{K}}$ and $\sb{f}_{\dot{L}^{r,q}}=\norm{f-c}_{L^{r,q}}$; if, in addition, $g\in\m{BMO}(\R^n;\mathbb{K})$ then $\sb{g}_{\m{BMO}}=\sb{g+c}_{\m{BMO}}$. Thus $f+g=(f-c)+(g+c)$ belongs to the right most space in~\eqref{monkey2}, and its seminorm is no more than $\sb{f}_{\dot{L}^{r,q}}+\sb{g}_{\m{BMO}}$. This argument shows the embedding of the middle space within the rightmost. The opposite embedding is clear, as $L^{r,q}(\R^n;\mathbb{K})\emb\dot{L}^{r,q}(\R^n;\mathbb{K})$.

The space on the left side of~\eqref{monkey} was asserted to be $L^{r,q}(\R^n;\mathbb{K})$ in the paper~\cite{MR448052}.  However, there is a subtle error in the proof of this,  Theorem 1 in~\cite{MR448052}, caused by failing to recognize that $\p{L^p\p{\R^n;\mathbb{K}};\m{BMO}\p{\R^n;\mathbb{K}}}_{s,q}$ is not Hausdorff due to a nontrivial annihilator (see Proposition~\ref{kernels prop}), and so limits in the interpolation space are not unique, nor is the standard quasi-Banach reiteration theorem available for use.  Here we will give a variant of the argument used in~\cite{MR448052} to correctly identify the missing constants now present in the right side of~\eqref{monkey}.  Note, though, that~\cite{MR448052} also seeks to identify the left side of~\eqref{monkey} with $0 < p < 1$, but we are unable to address this question without further generalizing our work to spaces defined with semiquasinorms.

We first prove~\eqref{monkey} in the special case $q=r$.  Since $L^\infty\p{\R^n;\mathbb{K}}\emb \m{BMO}(\R^n;\mathbb{K})$ it follows immediately that $L^r(\R^n;\mathbb{K})=(L^p(\R^n;\mathbb{K}),L^\infty(\R^n;\mathbb{K}))_{s,p}\emb(L^p(\R^n;\mathbb{K}),\m{BMO}(\R^n;\mathbb{K}))_{s,r}$. The right hand space has an annihilator consisting of exactly the constant functions (Proposition~\ref{kernels prop} applies since the intersection of the factors is Banach). Thus, the same embedding holds for $\dot{L}^r\p{\R^n;\mathbb{K}}$. 

To prove the reverse embedding  we need two tools from harmonic analysis.  The first is the decreasing rearrangement of a measurable function $g: \R^n \to \mathbb{K}$, which we denote by $g^\star : \R^+ \to [0,\infty]$.  We refer, for instance, to Chapter 1.4 of~\cite{MR3243734} for a thorough discussion of rearrangements and their relation to Lorentz spaces.  The main features we will need here are the estimates $(g_0 + g_1)^\star(t) \le g_0^\star(t/2) + g_1^\star(t/2)$ for $g_0,g_1 : \R^n \to \mathbb{K}$ measurable,  and 
\begin{equation}
\qnorm{g}_{L^{p,\infty}} =  \sup_{t \in \R^+} t^{1/p} g^\star(t)  \le \left(\int_{\R^+} (g^\star(t))^p dt \right)^{1/p} = \norm{g}_{L^p}      
\end{equation}
for $1 \le p < \infty$ and $g \in L^p(\R^n;\mathbb{K})$, where on the left $\qnorm{\cdot}_{L^{p,\infty}}$ is the quasinorm on $L^{p,\infty}(\R^n;\mathbb{K})$ that is equivalent to the interpolation norm when $p >1$.  The second tool is the `sharp' function: given $f\in L^1_{\loc}(\R^n;\mathbb{K})$ we define $f^\sharp : \R^n \to [0,\infty]$ via 
\begin{equation}
    f^\sharp\p{x}=\sup_{Q\ni x}\f{1}{\Le^n\p{Q}}\int_{Q}\Big|f-\f{1}{\Le^n(Q)}\int_{Q}f\Big|,
\end{equation}
where the supremum is taken over all cubes of the form $Q = \prod_{j=1}^n [a_j,a_j + \ell] \subset \R^n$ containing $x$. We will employ two essential facts about the sharp function.  First, if $1<\rho_0, \rho<\infty$ then there is $c_\rho\in\R^+$ such that we have the control: 
\begin{equation}\label{fancy equivalence}
    \textstyle{c_\rho}^{-1}\norm{f}_{L^\rho}\le\norm{f^\sharp}_{L^\rho},\;\text{ for all }f\in L^{\rho_0}(\R^n;\mathbb{K}).
\end{equation}
In other words, provided that $f$ belongs to some $L^{\rho_0}$, the above inequality holds in any $L^\rho$. A proof can be found in Chapter IV of~\cite{MR1232192}. Second, $(\cdot)^{\sharp}$ has the same boundedness properties as the Hardy-Littlewood maximal functions (see, for instance, Chapter I of ~\cite{MR1232192}). That is, the sharp map is weak type $(1,1)$ and strong type $(p,p)$.  This follows since the sublinear operators are related pointwise via $(\cdot)^{\sharp}\le 2M(\cdot)$, where $M$ is the cubic maximal operator.

With these tools in hand, we can prove prove the reverse inclusion.  Suppose initially that
\begin{equation}
    g\in\Delta(L^p(\R^n;\mathbb{K});\m{BMO}(\R^n;\mathbb{K}))\subset(L^p(\R^n;\mathbb{K}),\m{BMO}(\R^n;\mathbb{K}))_{s,r}
\end{equation} and decompose $g=g_0+g_1$ for $g_0\in L^p(\R^n;\mathbb{K})$ and $g_1\in\m{BMO}(\R^n;\mathbb{K})$. Using the subadditivity of $(\cdot)^\sharp$ with the weak-type $(p,p)$ boundedness of $(\cdot)^{\sharp}$ and the definition of $\sb{\cdot}_{\m{BMO}}$, we may estimate for $t\in\R^+$:
\begin{multline}\label{universe}
    \textstyle t^{1/p}(g^\sharp)^\star(t)\le t^{1/p}(g_0^\sharp + g_1^\sharp)^\star(t)  \le t^{1/p}({g_0}^\sharp)^\star(t/2)+t^{1/p}({g_1}^\sharp)^\star(t/2) \le 2^{1/p}\qnorm{{g_0}^\sharp}_{L^{p,\infty}}+t^{1/p}\snorm{({g_1}^\sharp)^\star}_{L^\infty} \\
    \textstyle \le c\big(\qnorm{{g_0}^\sharp}_{L^{p,\infty}}+t^{1/p}\snorm{{g_1}^\sharp}_{L^\infty}\big)
    \le c\big(\norm{g_0}_{L^p}+t^{1/p}\sb{g_1}_{\m{BMO}}\big),
\end{multline}
where $c\in\R^+$ is a constant independent of $g$. Taking the infimum over all decompositions of $g$ shows that $t^{1/p} (g^\sharp)^\star(t) \le  c\mathscr{K}(t^{1/p},g)$ for $t \in \R^+$.  This, the identity  $r=\f{p}{1-s}$, and~\eqref{fancy equivalence} then allow us to estimate 
\begin{multline}\label{fancy fancy}
    {c_{r}}^{-1}\norm{g}_{L^r}\le\snorm{g^\sharp}_{L^r}
    =\Big(\int_{\R^+}t\big((g^\sharp)^\star(t)\big)^rt^{-1}\;\m{d}t\Big)^{1/r} 
    = \Big(\int_{\R^+}\big(t^{(1-s)/p} (g^\sharp)^\star(t)\big)^rt^{-1}\;\m{d}t\Big)^{1/r} \\ 
    \le c \Big(\int_{\R^+}\big(t^{-s/p} \mathscr{K}(t^{1/p},g)\big)^rt^{-1}\;\m{d}t\Big)^{1/r}
    \le cp^{1/r}\Big(\int_{\R^+}(\tau^{-s}\mathscr{K}(\tau,g))^r\tau^{-1}\;\m{d}\tau\Big)^{1/r}=c\sb{g}_{s,r}.
\end{multline}
We now use~\eqref{fancy fancy} to deduce the general case via a limiting argument. Given $f\in(L^p(\R^n;\mathbb{K});\m{BMO}(\R^n;\mathbb{K}))_{s,r}$, Proposition~\ref{density in interpolation spaces} asserts that there is a sequence $\cb{f_k}_{k\in\N}\subset\Delta(L^p(\R^n;\mathbb{K});\m{BMO}(\R^n;\mathbb{K}))$ for which $f_k\to f$ in $(L^p(\R^n;\mathbb{K});\m{BMO}(\R^n;\mathbb{K}))_{s,r}$ as $k\to\infty$. For $m,k\in\N$ taking $g=f_k-f_m$ in~\eqref{fancy fancy} shows $\cb{f_k}_{k\in\N}$ is a Cauchy sequence in $L^r\p{\R^n;\mathbb{K}}$. Let $\tilde{f}$ denote its $L^r$-limit. By Proposition~\ref{kernels prop} the annihilator of $(L^p(\R^n;\mathbb{K});\m{BMO}(\R^n;\mathbb{K}))_{s,r}$ is the subspace of constant functions. As $L^r(\R^n;\mathbb{K})$ is embedded within this former space we conclude that $f-\tilde{f}$ is a constant function (note that it's precisely at this point where the error appears in~\cite{MR448052}). Hence $f\in\dot{L}^r\p{\R^n;\mathbb{K}}$, and we can estimate:
\begin{equation}
    \textstyle\sb{f}_{\dot{L}^r}\le\snorm{\tilde{f}}_{L^r}\le\snorm{\tilde{f}-f_k}_{L^r}+cc_r\sb{f_k}_{s,r}\to cc_r\sb{f}_{s,r}\text{ as }k\to\infty.
\end{equation}
This completes the proof of~\eqref{monkey} in the special case $q=r$.

To prove~\eqref{monkey} in the general case we will use reiteration, and for this we need the fact that
the intersection of  $L^p\p{\R^n;\mathbb{K}}$ and  $\m{BMO}\p{\R^n;\mathbb{K}}$ is complete, which follows easily from the fact that convergence in $L^p$ implies convergence in $L^1$ of every cube.  Let $u,v\in\R$ satisfy $1\le p<u<r<v<\infty$ and set $\vartheta=(1/r-1/u)/(1/v-1/u)\in(0,1)$. The above special case shows that
\begin{equation}
    ((L^p(\R^n;\mathbb{K}),\m{BMO}(\R^n;\mathbb{K}))_{1-p/u,u},(L^p(\R^n;\mathbb{K}),\m{BMO}(\R^n;\mathbb{K}))_{1-p/v,v})_{\vartheta,q}=(\dot{L}^u(\R^n;\mathbb{K}),\dot{L}^v(\R^n;\mathbb{K}))_{\vartheta,q}.
\end{equation}
Example~\eqref{lorentz} informs us that the right hand side is the Lorentz-like space $\dot{L}^{r,q}\p{\R^n;\mathbb{K}}$, while Theorem~\ref{reiteration theorem} tells us the left hand side is equal to $(L^p(\R^n;\mathbb{K}),\m{BMO}(\R^n;\mathbb{K}))_{\sigma,q}$ for $\sigma=(1-\vartheta)(1-p/u)+\vartheta(1-p/v)$. Using that $p=(1-s)((1-\vartheta)/u+\vartheta/v)^{-1}$ we compute that $\sig=s$. Thus~\eqref{monkey} is shown in all cases.

\end{exa}

\section{Homogeneous Sobolev and homogeneous Besov spaces}\label{Homogeneous Sobolev and homogeneous Besov spaces}
We now use the interpolation theory developed in the previous section to realize the homogeneous Besov spaces as intermediate interpolation spaces with respect to  members of the scale of homogeneous Sobolev spaces.  Along the way we will also develop frequency space characterizations used later in the paper.  Many of the results we present in this section are essentially already known in the literature, and we have attempted to omit as many proofs as possible.  The proofs we have included are meant to highlight the direct use of seminorms rather techniques employing spaces of distributions modulo polynomials.  The precise statements of the results in our notation will also be essential in the following section, where we develop the theory of screened Sobolev and screened Besov spaces. The reader already fluent in analysis of homogeneous function spaces could skip to Section~\ref{Screened Sobolev and screened Besov spaces}.

\subsection{Dyadic localization}
Here, for convenience of the reader, we recall the essentials of dyadic localization and Littlewood-Paley theory. We refer the reader to Appendix~\ref{harmonic analysis} for the relevant notions of real valued tempered distributions and multipliers.

\begin{lem}[Dyadic Partition of Unity]\label{dyadic partition of unity}
There exists a radial $\psi\in C^\infty_c\p{\R^n;\R}$ with $\supp\psi=\bar{B\p{0,2}}\setminus B\p{0,2^{-1}}$, $\psi\p{\xi}\in\R^+$ for $\xi\in B\p{0,2}\setminus\Bar{B\p{0,2^{-1}}}$, and $\sum_{k\in\Z}\del_{2^k}\psi=1$ on $\R^n\setminus\cb{0}$. Note that the $\del_{2^k}$ are the isotropic dilation operators, as in Lemma~\ref{scaling invariance of fourier multipliers}.
\end{lem}
\begin{proof}
See, for instance, Proposition 2.10 in \cite{MR2768550}.
%
\end{proof}

This dyadic partition of unity leads to the creation of `projection-like' operators that localize a given distribution at a certain dyadic annulus of frequencies.

\begin{defn}[Dyadic localization]\label{dyadic localization} 
Let  $\psi$ be the special function from Lemma~\ref{dyadic partition of unity}. To each $j\in\Z$ we define the operator $\uppi_j:   \mathscr{S}^\ast\p{\R^n;\mathbb{K}}\to C^\omega\p{\R^n;\mathbb{K}} \cap \mathscr{S}^\ast\p{\R^n;\mathbb{K}}$ via $\uppi_j f = [\p{\del_{2^j}\psi} \hat{f}]^{\vee}$.  This is well-defined by the Paley-Wiener-Schwartz theorem (see Chapter 6, Section 4 in~\cite{MR1336382}) and  Lemma~\ref{mult lemma}.
\end{defn}

The following lemmas record some basic properties of these operators.  

\begin{lem}\label{convergence}
The following hold:
\begin{enumerate}
    \item Suppose that $\varphi\in\mathscr{S}\p{\R^n;\mathbb{K}}$ is such that $0\not\in\supp\hat{\varphi}$. Then $\sum_{j=-m}^m\uppi_j\varphi\to\varphi$  in $\mathscr{S}\p{\R^n;\mathbb{K}}$  as $m\to\infty.$
\item If $f\in\mathscr{S}^\ast\p{\R^n;\mathbb{K}}$, then for each $\ell\in\Z$ the sequence $\big\{\sum_{j=0}^m\uppi_{j+\ell}f\big\}_{m\in\N}$ converges in $\mathscr{S}^\ast\p{\R^n;\C}$ to $g \in \mathscr{S}^\ast\p{\R^n;\mathbb{K}}$ with the property that for all $\varphi\in\mathscr{S}\p{\R^n;\C}$ with $0\not\in\supp\hat{\varphi}$
\begin{equation}
    \br{f-g,\varphi}=\br{\textstyle{\sum_{j=-\infty}^{\ell-1}}\uppi_jf,\varphi}
\end{equation}
where the right-hand-side is well defined since $\br{\uppi_jf,\varphi}=0$ for all but finitely many $j\in\Z$, $j<\ell$.

\item Suppose that $f,g\in\mathscr{S}^\ast\p{\R^n,\mathbb{K}}$ satisfy $\uppi_jf=\uppi_jg$ for all $j\in\Z$. Then there exists a polynomial $Q: \R^n \to \mathbb{K}$ such that $f+Q=g$. 
\end{enumerate}
\end{lem}
\begin{proof}
The first item follows from standard properties of the Schwartz class, and the second item follows from the first.  We now prove the third item.  If $\varphi\in\mathscr{S}\p{\R^n;\mathbb{K}}$ is such that $0\not\in\supp\hat{\varphi}$, by the first item $\varphi=\sum_{j\in\Z}\uppi_j\varphi$, with convergence in $\mathscr{S}\p{\R^n;\mathbb{K}}$. Consequently:
\begin{equation}
    \br{g-f,\varphi}=\textstyle{\sum_{j\in\Z}}\br{g-f,\uppi_j\varphi}=\textstyle{\sum_{j\in\Z}}\br{\uppi_j\p{g-f},\varphi}=0.
\end{equation}
Then $\hat{g}-\hat{f}$ is a tempered distribution supported at the origin, and hence $g-f$ is a $\mathbb{K}$-valued polynomial by, for instance, Corollary 2.4.2 in \cite{MR3243734}.
\end{proof}

The next lemma shows that the operators are almost idempotent and almost orthogonal.

\begin{lem}[Almost idempotence and almost orthogonality of dyadic localization]\label{almost idempotent} 
The operators $\cb{\uppi_j}_{j\in\Z}$ from Definition~\ref{dyadic localization} are `almost idempotent': if $j\in\Z$ and $f\in\mathscr{S}^\ast\p{\R^n;\mathbb{K}}$, then for all $m,k\in\N^+$ we have that
\begin{equation}
    \uppi_jf=\p{\textstyle{\sum_{\ell=-m}^k}\uppi_{j+\ell}}\uppi_jf=\uppi_j\p{\textstyle{\sum_{\ell=-m}^k}\uppi_{j+\ell}}f.
\end{equation}
They are also  `almost orthogonal': if $j,k \in \Z$ and $\abs{j-k}> 1$, then $\uppi_j\uppi_kf=0$.
\end{lem}
\begin{proof}
These follow immediately from the properties of $\psi$ from Lemma \ref{dyadic partition of unity}.
\end{proof}

Next we recall the Littlewood-Paley characterizations of $L^p$.

\begin{thm}[Littlewood-Paley inequalities in $L^p$]\label{littlewood-paley characterization of Lp}
Let $1<p<\infty$. The following hold:
\begin{enumerate}
    \item\emph{Frequency characterization of $L^p\p{\R^n;\mathbb{K}}$:} For $f\in\mathscr{S}^\ast\p{\R^n;\mathbb{K}}$ write
\begin{equation}
    \sb{f}_{L^p_{\thicksim}}=\bnorm{\p{\textstyle{\sum_{j\in\Z}}\abs{\uppi_jf}^2}^{1/2}}_{L^p}\in\sb{0,\infty}.
\end{equation}There exists a constant $c\in\R^+$, depending only on $n$ and $p$, such that the following hold:
\begin{enumerate}
    \item If $f\in L^p\p{\R^n;\mathbb{K}}$, then
     $c^{-1}\sb{f}_{L^p_{\thicksim}}\le\norm{f}_{L^p}$.
    \item If $f\in\mathscr{S}^\ast\p{\R^n;\mathbb{K}}$ is such that $\sb{f}_{L^p_{\thicksim}}< \infty$, then there exists a unique polynomial $Q: \R^n \to \mathbb{K}$ such that $f-Q$ can be identified with an $L^p\p{\R^n;\mathbb{K}}$ function, and $\norm{f-Q}_{L^p}\le c\sb{f}_{L^p_{\thicksim}}$.
\end{enumerate}
\item\emph{Vector-valued inequality:} Let $\phi\in L^1\p{\R^n;\C}\cap C^1\p{\R^n;\C}$ satisfy 
\begin{equation}
    0=\int_{\R^n}\phi\text{ and }\sup_{x\in\R^n}\p{1+\abs{x}}^{n+1}\p{\abs{\phi\p{x}}+\abs{\grad\phi\p{x}}}<\infty.
\end{equation}
For $f\in L^p\p{\R^n;\C}$ and $j\in\Z$ we write $\uppi^\phi_jf=\p{\del_{2^j}\phi}^{\vee}\ast f$. Let $1<r<\infty$. There is a constant $c\in\R^+$, depending only on $n$, $p$, $r$, and $\phi$, such that for any sequence $\cb{f_k}_{k\in\Z}\subset L^p\p{\R^n;\C}$ we have the bound
\begin{equation}
    \bnorm{\bp{\textstyle{\sum_{j\in\Z}}\bp{\textstyle{\sum_{k\in\Z}}\abs{\uppi_k^\phi f_j}^2}^{r/2}}^{1/r}}_{L^p}\le c\bnorm{\bp{\textstyle{\sum_{j\in\Z}}\abs{f_j}^r}^{1/r}}_{L^p}.
\end{equation}
\end{enumerate}
\end{thm}
\begin{proof}
See Theorem 6.1.2 and Proposition 6.1.4  in~\cite{MR3243734}.
\end{proof}

\subsection{Homogeneous Sobolev spaces}
Our primary goal in this subsection is to develop frequency-space characterizations of the homogeneous Sobolev spaces.

\begin{defn}[Homogeneous Sobolev spaces]\label{homogeneous Sobolev integer order}
Let $1\le p\le\infty$ and define the homogeneous Sobolev space
\begin{equation}
    \dot{W}^{1,p}\p{\R^n;\mathbb{K}}=\cb{f\in L^1_{\loc}\p{\R^n;\mathbb{K}}\;:\;\forall\;j\in\cb{1,\dots,n},\;\pd_jf\in L^p\p{\R^n;\mathbb{K}}}.
\end{equation}
This vector space is endowed with the seminorm $\sb{\cdot}_{\dot{W}^{1,p}}:\dot{W}^{1,p}\p{\R^n;\mathbb{K}}\to\R$ given by $\sb{f}_{\dot{W}^{1,p}}=\sum_{j=1}^n \norm{\partial_j f}_{L^p}$.
\end{defn}

Next we recall some useful facts about homogeneous Sobolev spaces. The first fact is a density result.

\begin{lem}[Density of compactly supported smooth functions in the homogeneous Sobolev spaces]\label{density of ccinfty in homog}
For $1\le p<\infty$ the following are equivalent:
\begin{enumerate}
    \item $C_c^\infty\p{\R^n;\mathbb{K}} \subset \dot{W}^{1,p}\p{\R^n;\mathbb{K}}$ is dense: for every $u\in\dot{W}^{1,p}\p{\R^n;\mathbb{K}}$ and $\ep\in\R^+$ there exists $w\in C_c^\infty\p{\R^n;\mathbb{K}}$ such that $\sb{u-w}_{\dot{W}^{1,p}}<\ep$.
    \item $1<p$ or $n\ge 2$.
\end{enumerate}
\end{lem}
\begin{proof}
See Theorem 4 in~\cite{MR1315521}.
\end{proof}

The second shows that functions in $\dot{W}^{1,p}$ define tempered distributions.

\begin{lem}[Members of homogeneous Sobolev spaces are tempered]\label{homog are tempered}
Let $1\le p\le\infty$.  Then the inclusion $\dot{W}^{1,p}\p{\R^n;\mathbb{K}}\subset\mathscr{S}^\ast\p{\R^n;\mathbb{K}}$ holds. More precisely, if $f\in\dot{W}^{1,p}\p{\R^n;\mathbb{K}}$, then the mapping
\begin{equation}
    \mathscr{S}\p{\R^n;\C}\ni\varphi\mapsto\int_{\R^n}f\varphi\in\C
\end{equation}
is well defined, continuous on $\mathscr{S}\p{\R^n;\C}$, and defines a $\mathbb{K}$-valued distribution.
\end{lem}
\begin{proof}
If $1\le p<n$, then the Gagliardo-Nirenberg-Sobolev embedding (see, for instance, Theorem 12.9 in~\cite{MR3726909}) implies that each member of $\dot{W}^{1,p}\p{\R^n;\mathbb{K}}$ is the sum of a constant function and an $L^q$-integrable function with $q=\f{np}{n-p}$, and thus defines a tempered distribution.  If $p=n$, then $\dot{W}^{1,n}\p{\R^n;\mathbb{K}} \emb \m{BMO}\p{\R^n;\mathbb{K}}$ by, for instance, Theorem 12.31 in~\cite{MR3726909}. The fact that functions of bounded mean oscillation are tempered is a consequence of item $\p{ii}$ in Proposition 3.1.5 in~\cite{MR3243741}.  Next if $n<p<\infty$, then $\dot{W}^{1,p}\p{\R^n;\mathbb{K}}\emb\dot{C}^{0,1-n/p}\p{\R^n;\mathbb{K}}$ (the latter space is the homogeneous H\"older space defined in Section~\ref{notation stuff}) thanks to Morrey's embedding (see Theorem 12.48 and Remark 12.49 in~\cite{MR3726909}). The H\"older space is tempered since its members grow at most linearly. Finally $\dot{W}^{1,\infty}(\R^n;\mathbb{K})$ is tempered since its elements may be modified on a null set to obtain a Lipschitz map - and hence tempered distribution (see the proof of Lemma~\ref{K functionals and modulus of continuity} below).
\end{proof}

The third result concerns the completeness of this space.

\begin{lem}[Completeness and annihilators of homogeneous Sobolev spaces]\label{completeness of homogeneous Sobolev spaces} 
Suppose that $1\le p\le\infty$. Then, the space $\dot{W}^{1,p}\p{\R^n;\mathbb{K}}$ is semi-Banach.  Moreover, $\mathfrak{A}\big(\dot{W}^{1,p}\big)=\cb{\text{constant functions}}$.
\end{lem}
\begin{proof}
This follows from the completeness of the Lebesgue spaces paired with Poincar\'e inequalities on cubes.
\end{proof}

We now prove a strong compatibility result.

\begin{lem}[Strong compatibility]\label{compatibility of homog and lebesgue}
 For $1\le p\le\infty$, the seminormed spaces $L^p\p{\R^n;\mathbb{K}}$ and $\dot{W}^{1,p}\p{\R^n;\mathbb{K}}$ are strongly compatible in the sense of definition~\ref{admissable semi normed spaces}.
\end{lem}
\begin{proof}
This result is an easy consequence of Proposition~\ref{kernel of the sum and intersection}. We view $L^p\p{\R^n;\mathbb{K}}$ and $\dot{W}^{1,p}\p{\R^n;\mathbb{K}}$ as simultaneously belonging to $L^1_\loc\p{\R^n;\mathbb{K}}$. Let $X$ denote the vector subspace consisting of their sum. Notice that $\Delta\big(L^p\p{\R^n;\mathbb{K}},\dot{W}^{1,p}\p{\R^n;\mathbb{K}}\big)=W^{1,p}\p{\R^n;\mathbb{K}}$ is a Banach space. Hence the annihilator of $X$, $\mathfrak{A}\p{X}$, is the sum of the annihilators of each factor. This is exactly the collection of constant functions. Therefore $L^p\p{\R^n;\mathbb{K}},\dot{W}^{1,p}\p{\R^n;\mathbb{K}}\emb X$, and $\mathfrak{A}\p{X}=\mathfrak{A}\p{L^p\p{\R^n;\mathbb{K}}}\cup\mathfrak{A}\big(\dot{W}^{1,p}\p{\R^n;\mathbb{K}}\big)$. This shows that the pair $L^p\p{\R^n;\mathbb{K}}$ and $\dot{W}^{1,p}\p{\R^n;\mathbb{K}}$ are strongly compatible.
\end{proof}

Now we explore the precise relation between the scales of homogeneous Sobolev spaces and the Riesz potential spaces. This yields a Fourier characterization of the former.

\begin{defn}[Riesz potentials and spaces]\label{riesz potential spaces}
Let $s\in\R$. If $f\in\mathscr{S}^\ast\p{\R^n;\mathbb{K}}$ is such that $0\not\in\supp\hat{f}$, then we define $\Lm^sf\in\mathscr{S}^\ast\p{\R^n;\mathbb{K}}$ via: $\br{\Lm^s f,\varphi} = \big\langle\hat{f},\varrho \abs{\cdot}^s\check{\varphi}\big\rangle \in \C$, where $\varrho \in C^\infty\p{\R^n}$ is any radial function satisfying $\varrho=1$ on $\supp\hat{f}$, and $\varrho=0$ on $B\p{0,\kappa}$, $\kappa=\min\big\{1,\m{dist}(\supp\hat{f},0)\big\}\in\R^+$. The purpose of the cutoff function $\varrho$ is to guarantee that $\varrho \abs{\cdot}^s\hat{\varphi}\in\mathscr{S}\p{\R^n;\C}$. It's straightforward to verify that this definition of $\Lm^s$ is independent of $\varrho$; hence, $\Lm^s$ defines a linear map on its domain that preserves the property of being $\mathbb{K}$-valued.  For $1<p<\infty$ we define the Riesz potential space 
\begin{equation}
    \dot{H}^{s,p}\p{\R^n;\mathbb{K}}=\Big\{f\in\mathscr{S}^\ast\p{\R^n;\mathbb{K}}\;:\;\big\{\textstyle{\sum_{k=-j}^j}\Lm^s\uppi_kf\big\}_{j\in\N}\subset L^p\p{\R^n;\mathbb{K}}\text{ is convergent }\Big\}.
\end{equation}
We equip this space with the seminorm $\sb{\cdot}_{\dot{H}^{s,p}} \to [0,\infty)$ defined by
\begin{equation}
    \sb{f}_{\dot{H}^{s,p}} =  \lim_{j\to\infty}\bnorm{\textstyle{\sum_{k=-j}^j}\Lm^s\uppi_kf}_{L^p}
    = \bnorm{\lim_{j\to\infty} \textstyle{\sum_{k=-j}^j}\Lm^s\uppi_kf}_{L^p}.
\end{equation}
\end{defn}

We first present a Littlewood-Paley characterization of $\dot{H}^{s,p}$ that gives a more useful seminorm to work with.  The proof is similar to that of Theorem 1.3.8 in \cite{MR3243741}, but here we work directly with the seminorms and avoid the technique of quotienting by polynomials.

\begin{thm}[Littlewood-Paley characterization of the Riesz potential spaces]\label{littlewood paley characterization of Riesz potential spaces}
Let $1<p<\infty$ and $s\in\R$.  Define the extended seminorm $\sb{\cdot}_{\dot{H}^{s,p}}^{\thicksim} : \mathscr{S}^\ast\p{\R^n;\mathbb{K}} \to \sb{0,\infty}$ via
\begin{equation}
 \sb{f}_{\dot{H}^{s,p}}^{\thicksim} = \bnorm{\bp{\textstyle{\sum_{j\in\Z}}\p{2^{sj}\abs{\uppi_jf}}^2}^{1/2}}_{L^p}.
    \end{equation}
    Then there exists $c\in\R^+$, depending on $s,n,p$, such that following hold:
    \begin{enumerate}
        \item If $f\in\dot{H}^{s,p}\p{\R^n;\mathbb{K}}$, then $\sb{f}_{\dot{H}^{s,p}}^{\thicksim}\le c\sb{f}_{\dot{H}^{s,p}}$.
        \item If $f\in\mathscr{S}^\ast\p{\R^n;\mathbb{K}}$ satisfies $\sb{f}_{\dot{H}^{s,p}}^{\thicksim}<\infty$, then  $f\in\dot{H}^{s,p}\p{\R^n;\mathbb{K}}$ and $\f{1}{c}\sb{f}_{\dot{H}^{s,p}}\le\sb{f}_{\dot{H}^{s,p}}^{\thicksim}$.
    \end{enumerate}
\end{thm}
\begin{proof}
Suppose first that $f\in\dot{H}^{s,p}\p{\R^n;\mathbb{K}}$. By hypothesis, there exists $f_s\in L^p\p{\R^n;\mathbb{K}}$ such that $\sum_{j=-m}^m\Lm^s\uppi_jf\to f_s$ as $m\to\infty$ in $L^p\p{\R^n;\mathbb{K}}$.
Consider $\phi\in C^\infty_c\p{\R^n;\R}\subset\mathscr{S}\p{\R^n;\C}$ defined via $\phi\p{\xi}=\abs{\xi}^{-s}\psi\p{\xi}$ (recall that $\psi$ is the special function from Lemma~\ref{dyadic partition of unity}). Observe this function is radial and that for $j\in\Z$ it holds that
\begin{multline}\label{that guy}
    2^{sj}\uppi_jf=2^{sj}\p{\del_{2^j}\sb{\p{\abs{\cdot}^{-s}\psi\abs{\cdot}^s}}\hat{f}}^{\vee}=2^{sj}\p{\del_{2^j}\sb{\p{\abs{\cdot}^{-s}\psi\abs{\cdot}^s}}\p{\textstyle{\sum_{\ell=-1}^1}\del_{2^{j+\ell}}\psi}\hat{f}}^{\vee}\\=\p{\del_{2^j}\phi\abs{\cdot}^s\p{\textstyle{\sum_{\ell=-1}^1}\del_{2^{j+\ell}}\psi}\hat{f}}^{\vee}={\uppi^{\phi}_j\Lm^s\p{\textstyle{\sum_{\ell=-1}^1}\uppi_{j+\ell}}f}={\uppi_j^{\phi}f_s}.
\end{multline}
Hence by item $(2)$ from Theorem~\ref{littlewood-paley characterization of Lp} we obtain the  bound
\begin{equation}
    \sb{f}^{\thicksim}_{\dot{H}^{s,p}}=\bnorm{\bp{\textstyle{\sum_{j\in\Z}}\abs{\uppi_j^\phi f_s}^2}^{1/2}}_{L^p}\lesssim\norm{f_s}_{L^p}=\sb{f}_{\dot{H}^{s,p}}.
\end{equation}

On the other hand, suppose that $f\in\mathscr{S}^\ast\p{\R^n;\mathbb{K}}$ satisfies $\sb{f}^\thicksim_{\dot{H}^{s,p}}<\infty$. We again use Theorem~\ref{littlewood-paley characterization of Lp} to show that the sequence $\{\sum_{j=-m}^m\Lm^s\uppi_jf\}_{j\in\N}$ is $L^p\p{\R^n;\mathbb{K}}$-Cauchy.  To begin, we claim that the sequence is actually contained within $L^p\p{\R^n;\mathbb{K}}$.  Indeed, the bound $\sb{f}^{\thicksim}_{\dot{H}^{s,p}}<\infty$ implies that $\uppi_jf\in L^p\p{\R^n;\mathbb{K}}$ for all $j\in\Z$. Then the annular frequency support of $\uppi_jf$ implies that the multiplier defining $\Lm^s$ can be taken to be smooth and compactly supported, and thus satisfying the hypotheses of Theorem~\ref{mihlin multiplier theorem}. The theorem then guarantees that $\Lm^s\uppi_jf\in L^p\p{\R^n;\C}$.  To complete the proof of the claim note that this sequence is $\mathbb{K}$-valued by the results in Appendix~\ref{harmonic analysis}. Let $m,k\in\N$  with $m<k$. Using Lemma~\ref{almost idempotent} shows that for $j \in \Z$ we have
\begin{equation}\label{pineapple naught}
    \uppi_j\p{\sum_{\ell=-k}^k\Lm^s\uppi_\ell f-\sum_{\ell=-m}^m\Lm^s\uppi_\ell f}=\begin{cases}
    \Lm^s\uppi_jf&m+1<\abs{j}\le k-1\\
    0&\abs{j}\ge k+2,\;\abs{j}<m\\
    \uppi_{k+1}\Lm^s\uppi_k f&j=k+1\\
    \uppi_{-k-1}\Lm^s\uppi_{-k}f&j=-k-1\\
    \uppi_{m}\Lm^s\uppi_{m+1}f&j=m\\
    \uppi_{-m}\Lm^s\uppi_{-m-1}f&j=-m\\
    \p{\uppi_{k-1}+\uppi_{k}}\Lm^s\uppi_k f&j=k\\
    \p{\uppi_{-k+1}+\uppi_{-k}}\Lm^s\uppi_{-k} f&j=-k\\
    \p{\uppi_{m+1}+\uppi_{m+2}}\Lm^s\uppi_{m+1}f&j=m+1\\
    \p{\uppi_{-m-1}+\uppi_{-m-2}}\Lm^s\uppi_{-m-1}f&j=-m-1.
    \end{cases}
\end{equation}
Hence, by item $(1)$ from Theorem~\ref{littlewood-paley characterization of Lp} (due to $L^p$ inclusion, there does not appear a polynomial) we may bound
\begin{multline}\label{this guy}
    \bnorm{\textstyle{\sum_{\ell=-k}^k}\Lm^s\uppi_\ell f-\textstyle{\sum_{\ell=-m}^m}\Lm^s\uppi_\ell f}_{L^p}\lesssim\bnorm{\p{\sum_{m+1<\abs{j}\le k-1}\abs{\Lm^s\uppi_jf}^2}^{1/2}}_{L^p}+\norm{\uppi_{k+1}\Lm^s\uppi_kf}_{L^p}\\+\norm{\uppi_{-k-1}\Lm^s\uppi_{-k}f}_{L^p}+\norm{\uppi_m\Lm^s\uppi_{m+1}f}_{L^p}
    +\norm{\uppi_{-m}\Lm^s\uppi_{-m-1}}_{L^p}+\norm{\p{\uppi_{k-1}+\uppi_{k}}\Lm^s\uppi_k f}_{L^p}\\+\norm{\p{\uppi_{-k+1}+\uppi_{-k}}\Lm^s\uppi_{-k} f}_{L^p}+\norm{\p{\uppi_{m+1}+\uppi_{m+2}}\Lm^s\uppi_{m+1}f}_{L^p}+\norm{\p{\uppi_{-m-1}+\uppi_{-m-2}}\Lm^s\uppi_{-m-1}f}_{L^p}.
\end{multline}
By Theorem~\ref{mihlin multiplier theorem} applied to $\psi$ and then Lemma~\ref{scaling invariance of fourier multipliers}, there is a constant $c\in\R^+$, depending only on $\psi$, $n$, and $p$, such that for all $j\in\Z$ it holds that $\norm{\del_{2^j}\psi}_{\mathscr{M}_p}=c$. Therefore the bound \eqref{this guy} implies that
\begin{equation}\label{cat}
    \bnorm{\textstyle{\sum_{\ell=-k}^k}\Lm^s\uppi_\ell f-\textstyle{\sum_{\ell=-m}^m}\Lm^s\uppi_\ell f}_{L^p}\lesssim\bnorm{\bp{\textstyle{\sum_{m<\abs{j}\le k}}\abs{\Lm^s\uppi_jf}^2}^{1/2}}_{L^p}.
\end{equation}
Now let $\nu\in C^\infty_c\p{\R^n;\R}\subset\mathscr{S}\p{\R^n;\C}$ be the radial function defined via $\nu\p{\xi}=\abs{\xi}^s\psi\p{\xi}$; the properties of $\psi$ guarantee that $\int_{\R^n} \nu =0$. Arguing as in~\eqref{that guy} shows that $\Lm^s\uppi_jf=2^{sj}\uppi_j^\nu f$; moreover, Lemma~\ref{almost idempotent} implies that $\uppi^\nu_j=\uppi^\nu_j\sum_{\ell=-1}^1\uppi_{j+\ell}$ for each $j\in\Z$. Hence, \eqref{cat} paired with item $(2)$ of Theorem~\ref{littlewood-paley characterization of Lp} yield the bounds
\begin{multline}\label{kool kat}
    \bnorm{\textstyle{\sum_{\ell=-k}^k}\Lm^s\uppi_\ell f-\textstyle{\sum_{\ell=-m}^m}\Lm^s\uppi_\ell f}_{L^p}\lesssim\bnorm{\bp{\textstyle{\sum_{m<\abs{j}\le k}}\p{2^{sj}\abs{\uppi_j^{\nu}f}}^2}}^{1/2}_{L^p}\\\le\textstyle{\sum_{\ell=-1}^1}\bnorm{\bp{\textstyle{\sum_{m<\abs{j}\le k}}\big(2^{sj}|\uppi^\nu_j\uppi_{j+\ell}f|\big)^2}^{1/2}}_{L^p}=\sum_{\ell=-1}^12^{-s\ell}\bnorm{\bp{\sum_{m<\abs{j}\le k}\abs{\uppi_{j}^{\nu}\uppi_{j+\ell}{2^{s\p{j+\ell}}f}}^2}^{1/2}}_{L^p}\\
    \le\textstyle{\sum_{\ell=-1}^1}2^{-s\ell}\bnorm{\bp{\sum_{m-1<\abs{r}\le k+1}\sum_{m<\abs{j}\le k}\abs{\uppi_j^\nu\uppi_r{2^{sr}f}}^2}^{1/2}}_{L^p}\lesssim\bnorm{\bp{\sum_{m-1<\abs{r}\le k+1}\p{2^{sr}\abs{\uppi_rf}}^2}^{1/2}}_{L^p}.
\end{multline}
Since $\sb{f}_{\dot{H}^{s,p}}^\thicksim<\infty$, we can now show that as $m\to\infty$ the final expression in~\eqref{kool kat} tends to zero. Indeed, for a.e. $x\in\R^n$ the sum $\sum_{r\in\Z}\p{2^{sr}\abs{\uppi_rf\p{x}}}^2$ is finite. For such $x$ we have that
\begin{equation}
\bp{\textstyle{\sum_{m-1<\abs{r}\le k+1}}\p{2^{sr}\abs{\uppi_rf\p{x}}}^2}^{1/2}\to0\text{ as } m<k\to\infty.
\end{equation}
The limit in $L^p\p{\R^n;\mathbb{K}}$ follows now from the dominated convergence theorem.  We deduce then that the sequence $\big\{\sum_{j=-m}^m\Lm^s\uppi_j f\big\}_{j\in\N}$ is  Cauchy in $L^p\p{\R^n;\mathbb{K}}$, and hence $f\in\dot{H}^{s,p}\p{\R^n;\mathbb{K}}$. We can now argue exactly as above to deduce that for each $m\in\N$ it holds that
\begin{equation}
    \bnorm{\textstyle{\sum_{j=-m}^m}\Lm^s\uppi_jf}_{L^p}\lesssim_{s,n,p,\psi}\bnorm{\bp{\textstyle{\sum_{j\in\Z}}\p{2^{sj}\abs{\uppi_j f}}^2}^{1/2}}_{L^p}.
\end{equation}
\end{proof}
 

Using the Littlewood-Paley characterization of the Riesz potential spaces, we are now able to see that the spaces $\dot{H}^{1,p}$ and $\dot{W}^{1,p}$ essentially coincide.  The proof of the following result is technical refinement of Theorem 6.3.1 in \cite{MR0482275} in the sense that we do not require the Fourier transform of $f$ to vanish near the origin.

\begin{thm}[Frequency space characterization of $\dot{W}^{1,p}$]\label{frequency space characterization of homogeneous sobolev spaces}
Let $1<p<\infty$. There exists a constant $c\in\R^+$, depending only on $n$, $p$, and $\psi$, such that the following hold:
\begin{enumerate}
    \item If $f\in \dot{W}^{1,p}\p{\R^n;\mathbb{K}}$, then $c^{-1}\sb{f}_{\dot{H}^{1,p}}\le\sb{f}_{\dot{W}^{1,p}}$
    \item If $f\in\dot{H}^{1,p}\p{\R^n;\mathbb{K}}$, then there exists a $\mathbb{K}$-valued polynomial $Q$ with the property that $f-Q$ can be identified with an $\dot{W}^{1,p}$-function and $\sb{f-Q}_{\dot{W}^{1,p}}\le c\sb{f}_{\dot{H}^{1,p}}$. Moreover, the coefficients of terms of degree $1$ and higher of $Q$ are uniquely determined.
\end{enumerate}
\end{thm}
\begin{proof}
Suppose first that $f\in\dot{H}^{1,p}\p{\R^n;\mathbb{K}}$. Let us first show that the sequence $\big\{\sum_{j=-m}^m\uppi_jf\big\}_{m\in\N}$ is Cauchy in $\dot{W}^{1,p}\p{\R^n;\mathbb{K}}$. This sequence of tempered distributions is identified with a sequence of locally integrable functions since each member has compactly supported Fourier transform (see the Paley-Wiener-Schwartz theorem in, for instance, Chapter 6, Section 4 of \cite{MR1336382}).  If $k\in\cb{1,\dots,n}$ the mapping $\R^n\ni\xi\mapsto\ii\xi_k\abs{\xi}^{-1}$ (a scalar multiple of the usual Riesz transform) belongs to $\mathscr{M}_p\p{\R^n;\mathbb{K}}$ by Theorem~\ref{mihlin multiplier theorem} and Lemma~\ref{mult lemma}; therefore, since $\sum_{j=-m}^m\Lm^1\uppi_jf\in L^p\p{\R^n;\mathbb{K}}$ it then holds that
\begin{equation}
    \bnorm{\textstyle{\sum_{j=-m}^m}\pd_k\uppi_jf}_{L^p}\lesssim_{n,p}\bnorm{\textstyle{\sum_{j=-m}^m}\Lm^1\uppi_jf}_{L^p}<\infty \text{ for } m\in\N\imp\cb{\sum_{j=-m}^m\uppi_jf}_{m\in\N}\subset\dot{W}^{1,p}\p{\R^n;\mathbb{K}}.
\end{equation}
The above argument, supplemented with ideas from the latter half of the proof of Theorem~\ref{littlewood paley characterization of Riesz potential spaces}, yields for $m,k\in\N$ with $m<k$:
\begin{multline}
    \textstyle{\sum_{\ell=1}^n}\bnorm{\sum_{j=-k}^k\pd_\ell\uppi_j f-\sum_{j=-m}^m\pd_\ell\uppi_j f}_{L^p}\lesssim_{n,p}\bnorm{\sum_{j=-k}^k\Lm^1\uppi_jf-\sum_{j=-m}^m\Lm^1\uppi_jf}_{L^p}\\\lesssim_{n,p,\psi}\bnorm{\p{\textstyle{\sum_{m-1<\abs{r}\le k+1}}\p{2^{sr}\abs{\uppi_r f}}^2}^{1/2}}_{L^p}.
\end{multline}
This estimate paired with Theorem~\ref{littlewood paley characterization of Riesz potential spaces} shows that the sequence in question is indeed Cauchy in the space $\dot{W}^{1,p}$. As this seminormed space is semi-Banach thanks to Lemma~\ref{completeness of homogeneous Sobolev spaces}, we are assured of the existence of $g\in\dot{W}^{1,p}\p{\R^n;\mathbb{K}}$ with the property that for each $\ell \in \{1,\dotsc,n\}$
\begin{equation}
\pd_\ell\textstyle{\sum_{j=-m}^m}\uppi_jf\to\pd_\ell g\text{ in }L^p\p{\R^n;\mathbb{K}}\text{ as }m\to\infty.
\end{equation}
Theorem~\ref{mihlin multiplier theorem} assures us that for all $j\in\Z$,  $\uppi_j\in\mathcal{L}\p{L^p\p{\R^n;\mathbb{K}};L^p\p{\R^n;\mathbb{K}}}$. This fact, paired with Lemma~\ref{almost idempotent}, shows that
\begin{equation}
    \uppi_j\pd_\ell g=\uppi_j\pd_\ell f\text{ for all }j\in\Z\text{ and for all }\ell\in\cb{1,\dots,n}.
\end{equation}
Hence Lemma~\ref{convergence} implies that $\grad f=\grad g+P$ for a $\mathbb{K}^n$-valued polynomial $P$. By Poincar\'{e}'s lemma there is $\mathbb{K}-$valued polynomial $Q$ such that $\grad Q= P$. We are free to adjust the constant term of $Q$ so that $f=g+Q$. If $\tilde{Q}$ were another polynomial with the property that  $f-\tilde{Q}\in\dot{W}^{1,p}\p{\R^n;\mathbb{K}}$. Then $\tilde{Q}-Q$ would also belong to the space $\dot{W}^{1,p}\p{\R^n;\mathbb{K}}$. Hence $\grad(\tilde{Q}-Q)$ is necessarily zero.

We now estimate  $\sb{f-Q}_{\dot{W}^{1,p}}$  using again the fact that $\xi\mapsto\ii\xi_\ell\abs{\xi}^{-1}$ belongs to $\mathscr{M}_p\p{\R^n;\mathbb{K}}$ for all $\ell\in\cb{1,\dots,n}$:
\begin{equation}
    \sb{f-Q}_{\dot{W}^{1,p}}\lesssim_n\lim_{m\to\infty}\bnorm{\textstyle{\sum_{j=-m}^m}\grad\uppi_jf}_{L^p}\lesssim_{n,p}\limsup_{m\to\infty}\bnorm{\textstyle{\sum_{j=-m}^m}\Lm^1\uppi_jf}_{L^p}.
\end{equation}
The proof of Theorem~\ref{littlewood paley characterization of Riesz potential spaces} shows that for each $m\in\N$ we may bound
\begin{equation}
    \bnorm{\textstyle{\sum_{j=-m}^m}\Lm^1\uppi_jf}_{L^p}\lesssim_{n,p,\psi}\bnorm{\bp{\textstyle{\sum_{j\in\Z}}\p{2^j\abs{\uppi_jf}}^2}^{1/2}}_{L^p}.
\end{equation}
Thus, the proof of the second item is now complete.

On the other hand, suppose that $f\in\dot{W}^{1,p}\p{\R^n;\mathbb{K}}$. Let $\rho\in C^\infty\p{\R;\R}$ be an even function that vanishes in the interval $\sb{-\f12n^{-1/2},\f12n^{-1/2}}$ and is identically $1$ outside of $\p{-n^{-1/2},n^{1/2}}$. Consider the multipliers $\bf{m}_0,\bf{m}_1:\R^n\to\R$ defined via
\begin{equation}\label{coffee two}
    \bf{m}_0(\xi) = \rho\big(\abs{\xi}n^{-1/2}\big)\abs{\xi}\bp{\textstyle{\sum_{\ell=1}^n}\rho\p{\xi_\ell}\abs{\xi_\ell}}^{-1}\text{ and  }\bf{m}_1(\xi)=\textstyle{\sum_{\ell=1}^n}\rho\p{\xi_\ell}\abs{\xi_\ell}.
\end{equation}
Observe that $\bf{m}_0$ is smooth, vanishes in $B\p{0,1/2}\subset\R^n$, and agrees with $\xi\mapsto\abs{\xi}\p{\sum_{\ell=1}^n\rho\p{\xi_\ell}\abs{\xi_\ell}}^{-1}$ for $\abs{\xi}\ge 1$. One can verify that $\bf{m_0}\in\mathscr{M}_p\p{\R^n;\mathbb{K}}$. Now if $m\in\N$, then $B\p{0,\kappa_m}\subset\R^n\setminus\supp\mathscr{F}\big(\sum_{j=-m}^m\uppi_jf\big)$ where $\kappa_m=2^{-m-2}$. This tells us that
\begin{equation}\label{coffee too}
    \Lm^1\textstyle{\sum_{j=-m}^m}\uppi_jf=\kappa_m\p{\del_{\kappa_m}\bf{m}_0\del_{\kappa_m}\bf{m}_1}^{\vee}\ast\textstyle{\sum_{j=-m}^m}\uppi_jf.
\end{equation}
In turn, by Lemma~\ref{scaling invariance of fourier multipliers} we can bound
\begin{equation}\label{coffee zero}
    \bnorm{\Lm^1\textstyle{\sum_{j=-m}^m}\uppi_jf}_{L^p}\le\kappa_m\norm{\del_{\kappa_m}\bf{m}_0}_{\mathscr{M}_p}\bnorm{\del_{\kappa_m}\bf{m}_1\sum_{j=-m}^m\uppi_jf}_{L^p}=\kappa_m\norm{\bf{m}_0}_{\mathscr{M}_p}\bnorm{\p{\del_{\kappa_m}\bf{m}_1}^{\vee}\ast\sum_{j=-m}^m\uppi_jf}_{L^p}.
\end{equation}
Finally, the fact that for each $\ell\in\cb{1,\dots,n}$ the map (and its dilates) $\xi\mapsto\ii\f{\abs{\xi_\ell}}{\xi_\ell}\rho\p{\xi_\ell}$ belong to $\mathscr{M}_p\p{\R^n;\mathbb{K}}$ yields the bound
\begin{equation}\label{coffee one}
    \bnorm{\p{\del_{\kappa_m}\bf{m}_1}^\vee\ast\textstyle{\sum_{j=-m}^m}\uppi_jf}_{L^p}\lesssim_{n,p}{\kappa_m}^{-1}\textstyle{\sum_{\ell=1}^n}\bnorm{\pd_\ell\textstyle{\sum_{j=-m}^m}\uppi_jf}_{L^p}={\kappa_m}^{-1}\textstyle{\sum_{\ell=1}^n}\bnorm{\sum_{j=-m}^m\uppi_j\pd_\ell f}_{L^p}.
\end{equation}
Since $\uppi_j\in\mathcal{L}\p{L^p\p{\R^n;\mathbb{K}};L^p\p{\R^n;\mathbb{K}}}$ for each $j\in\Z$, the estimates \eqref{coffee zero} and \eqref{coffee one} imply the inclusion $\big\{\sum_{j=-m}^m\Lm^1\uppi_jf\big\}_{m\in\N} \subset L^p\p{\R^n;\mathbb{K}}$.

To see that this sequence is also Cauchy, we apply the argument in \eqref{coffee too}, \eqref{coffee zero}, and~\eqref{coffee one} to the function $\sum_{j=-k}^k\Lm^1\uppi_jf-\sum_{j=-m}^m\Lm^1\uppi_jf$ for $k>m$, $m,k\in\N$. This shows that
\begin{equation}\label{math is math is math}
    \bnorm{\textstyle{\sum_{j=-k}^k}\Lm^1\uppi_jf-\textstyle{\sum_{j=-m}^m}\Lm^1\uppi_jf}_{L^p}\lesssim_{n,p}\textstyle{\sum_{\ell=1}^n}\bnorm{\textstyle{\sum_{m<\abs{j}\le k}}\uppi_j\pd_\ell f}_{L^p}.
\end{equation}
The term above on the right can be universally estimated using item $(1)$ of Theorem~\ref{littlewood-paley characterization of Lp} (due to $L^p$ inclusion,  there does not appear a polynomial) and Theorem~\ref{mihlin multiplier theorem}:
\begin{equation}\label{math is math}
    \textstyle{\sum_{\ell=1}^n}\bnorm{\textstyle{\sum_{m<\abs{j}\le k}}\uppi_j\pd_\ell f}_{L^p}\lesssim_{n,p,\psi}\textstyle{\sum_{\ell=1}^n}\bnorm{\bp{\sum_{m<\abs{j}\le k}\abs{\uppi_j\pd_\ell f}^2}^{1/2}}_{L^p}.
\end{equation}
As a consequence of item $(1)$ in Theorem~\ref{littlewood-paley characterization of Lp}, we have the equivalence for each $\ell\in\cb{1,\dots,n}$:
\begin{equation}
    \bnorm{\bp{\textstyle{\sum_{j\in\Z}}\abs{\uppi_j\pd_\ell f}^2}^{1/2}}_{L^p}\asymp_{n,p,\psi}\bnorm{\pd_\ell f}_{L^p}<\infty.
\end{equation}
Therefore equations~\eqref{math is math is math} and~\eqref{math is math} show the sequence $\big\{\sum_{j=-m}^m\Lm^1\uppi_jf\big\}_{m\in\N}$ to be $L^p\p{\R^n;\mathbb{K}}$-Cauchy.

Finally, \eqref{coffee zero}, \eqref{coffee one}, and Theorem~\ref{littlewood-paley characterization of Lp} imply that
\begin{equation}
    \lim_{m\to\infty}\bnorm{\textstyle{\sum_{j=-m}^m}\Lm^1\uppi_jf}_{L^p}\lesssim_{n,p,\psi}\limsup\limits_{m\to\infty}\textstyle{\sum_{\ell=1}^n}\bnorm{\bp{\textstyle{\sum_{j=-m}^m}\abs{\uppi_j\pd_\ell f}^2}^{1/2}}_{L^p}\asymp_{n,p,\psi}\textstyle{\sum_{\ell=1}^n}\norm{\pd_\ell f}_{L^p}.
\end{equation}
\end{proof}

The generalizations of Theorem~\ref{frequency space characterization of homogeneous sobolev spaces} for the pairs of spaces $\dot{H}^{m,p}(\R^n;\mathbb{K})$ and $\dot{W}^{m,p}(\R^n;\mathbb{K})$ hold for $m\in\N\setminus\cb{0,1}$ and $1<p<\infty$ and are proven in essentially the same way as above.

\subsection{Homogeneous Besov spaces}
We now turn our attention to the scale of homogeneous Besov spaces.
\begin{defn}[Translations, difference quotients, moduli of continuity]\label{difference quotients and moduli of continuity}
Let $1\le p\le\infty$. 
\begin{enumerate}
    \item For $h\in\R^n$ we define the $h$-translation operator, $\tau_h$, and the $h-$forward difference operator, $\Delta_h$, as follows.  Given $f: \R^n \to \mathbb{K}$, we let $\tau_h f, \Delta_h f : \R^n \to \mathbb{K}$ be given by 
    $
    \tau_hf\p{x}=f\p{x-h}$ and  $\Delta_hf\p{x}=\p{\tau_{-h}-1}f\p{x}
    $.
    \item We define the $L^p$-modulus of continuity as the functional $\omega_{p}:\R^+\times L^1_\loc\p{\R^n;\mathbb{K}}\to\sb{0,\infty}$ with action $\omega_{p}\p{t,f} =\sup\big\{\norm{\Delta_hf}_{L^p}\;:\;h\in\Bar{B\p{0,t}}\subset\R^n\big\}$.
\end{enumerate}
\end{defn}

\begin{defn}[Homogeneous Besov spaces]\label{defintion of besov spaces}Let $s\in\p{0,1}$ and $1\le p,q\le\infty$. We define the homogeneous Besov space $\dot{B}^{s,p}_q\p{\R^n;\mathbb{K}}=\big\{f\in L^1_\loc\p{\R^n;\mathbb{K}}\;:\;\sb{f}_{\dot{B}^{s,p}_q}<\infty\big\}$,
where $\sb{\cdot}_{\dot{B}^{s,p}_q}:L^1_\loc\p{\R^n;\mathbb{K}}\to\sb{0,\infty}$ is defined by
\begin{equation}
    \sb{f}_{\dot{B}^{s,p}_q}=
    \begin{cases}
    \p{\int_{\R^+}\p{t^{-s}\omega_{p}\p{t,f}}^qt^{-1}\;\m{d}t}^{1/q}&1\le q<\infty\\
    \sup\cb{t^{-s}\omega_{p}\p{t,f}\;:\;t\in\R^+}&q=\infty.
    \end{cases}
\end{equation}
\end{defn}

The following equivalent seminorm is occasionally useful.

\begin{defn}\label{difference quotient norm on besov 2}
Let $n\in\N^+$, $s\in\p{0,1}$, $1\le p, q\le\infty$. We define $\sb{\cdot}_{\dot{B}^{s,p}_q}^{\thicksim}:L^1_\loc\p{\R^n;\mathbb{K}}\to\sb{0,\infty}$ via 
\begin{equation}
\sb{f}_{\dot{B}^{s,p}_q}^{\thicksim} =
    \begin{cases}
    \p{\int_{\R^n}\p{\abs{h}^{-s}\norm{\Delta_hf}_{L^p}}^q\abs{h}^{-n}\;\m{d}h}^{1/q}&1\le q<\infty\\
    \m{esssup}\cb{\abs{h}^{-s}\norm{\Delta_hf}_{L^p}\;:\;h\in\R^n}&q=\infty
    \end{cases}.
\end{equation}
Proposition 17.21 in~\cite{MR3726909} shows that $\sb{\cdot}^{\thicksim}_{\dot{B}^{s,p}_q}$ is equivalent to $\sb{\cdot}_{\dot{B}^{s,p}_q}$.
\end{defn}

The proof of the following lemma is a slightly modified excerpt from the proof of Theorem 17.24 in~\cite{MR3726909}. We include it to emphasize the connection between the $K$-functional on the sum of $L^p$ and $\dot{W}^{1,p}$ and the $L^p$ modulus of continuity on $L^1_\loc$.

\begin{lem}[Relation between $K$-functional and moduli of continuity]\label{K functionals and modulus of continuity}
Fix $1\le p\le\infty$. Let $\mathscr{K}$ denote the $K$-functional, from Definition~\ref{K functional for semi normed spaces}, corresponding to the space $L^p\p{\R^n;\mathbb{K}}$ and $\dot{W}^{1,p}\p{\R^n;\mathbb{K}}$. Then for all $\p{t,u}\in\R^+\times L^1_{\loc}\p{\R^n;\mathbb{K}}$ we have the equivalence:
\begin{equation}\label{want to show equiv}
    2^{-1}\omega_p\p{t,u}\le\mathscr{K}\p{t,u}\le (1+n^{3/2})\omega_p\p{t,u}.
\end{equation}
\end{lem}
\begin{proof}
First, we prove the left inequality in~\eqref{want to show equiv}. We may reduce to proving this under the extra assumption that $u\in\Sigma\big(L^p\p{\R^n;\mathbb{K}};\dot{W}^{1,p}\p{\R^n;\mathbb{K}}\big)$ since otherwise the right-hand-side is infinite, and there is nothing to prove.  Assume this and let $t\in\R^+$.  Suppose that $\p{v,w}\in L^p\p{\R^n;\mathbb{K}}\times\dot{W}^{1,p}\p{\R^n;\mathbb{K}}$ are a decomposition of $u$, that is: $u=v+w$. For any $h\in\bar{B\p{0,t}}$ the estimate $\norm{\Delta_hv}_{L^p}\le 2\norm{v}_{L^p}$ is clear by the triangle inequality and invariance of the $L^p$-norm under translations.  On the other hand, we have that $\norm{\Delta_hw}_{L^p} \le \abs{h} \sb{ w}_{\dot{W}^{1,p}}$. In the case that $1\le p<\infty$ we let $\cb{\varphi_\ep}_{\ep\in\p{0,1}}\subset C^\infty_c\p{B\p{0,1}}$ be a standard mollifier. Then for $\ep\in\p{0,1}$ we can use the fundamental theorem of calculus and Minkowski's integral inequality to bound:
\begin{multline}
    \bp{\int_{\R^n}\abs{w\ast\varphi_\ep\p{x+h}-w\ast\varphi_\ep\p{x}}^p\;\m{d}x}^{1/p}=\bp{\int_{\R^n}\Big|\int_{\p{0,1}}\grad\p{w\ast\varphi_\ep}\p{x+\sig h}\cdot h\;\m{d}\sig\Big|^p\;\m{d}x}^{1/p}\\\le\abs{h}\int_{\p{0,1}}\bp{\int_{\R^n}\abs{\grad\p{w\ast\varphi_\ep}\p{x+\sig h}}^p\;\m{d}x}^{1/p}\;\m{d}\sig=\abs{h}\sb{w\ast\varphi_\ep}_{\dot{W}^{1,p}}\le\abs{h}\sb{w}_{\dot{W}^{1,p}}.
\end{multline}
Letting $\ep\to0^+$ and using Fatou's lemma gives the claim. On the other hand if $p=\infty$ we repeat the same argument and use lower semicontinuity of weak-$\ast$ convergence in $L^\infty=(L^1)^\ast$ of the mollified functions in place of Fatou's lemma. Note that since the sequence of mollified functions converge pointwise a.e. the argument also shows $\dot{W}^{1,\infty}(\R^n;\mathbb{K})\emb\dot{C}^{0,1}(\R^n;\mathbb{K})$.

Now we take the supremum over $h\in\Bar{B\p{0,t}}$: $\norm{\Delta_hu}_{L^p}\le2\p{\norm{v}_{L^p}+t\sb{ w}_{\dot{W}^{1,p}}}\imp\omega_p\p{t,u}\le2\p{\norm{v}_{L^p}+t\norm{\grad w}_{L^p}}$.
We then take the infimum over all such decompositions of $u$ to see that $\omega_p\p{t,u}\le2\mathscr{K}\p{t,u}$.

Next, we prove the first inequality in~\eqref{want to show equiv} in the case that $1\le p<\infty$. Suppose that $t\in\R^+$, and $u\in L^1_\loc\p{\R^n;\mathbb{K}}$ is such that $\omega_p\p{t,u}<\infty$ (if this is infinite, then there is, again, nothing to prove). Let $Q\big(0,n^{-\f12}t\big)$ denote the cube centered at the origin with sides of length $n^{-\f12}t$ which are parallel to the coordinate axes.  Consider $v,w:\R^n\to\mathbb{K}$ given by
\begin{equation}\label{decomp}
    v\p{x}=-\f{n^{\f{n}{2}}}{t^n}\int_{Q\p{0,n^{-\f12}t}}\Delta_yu\p{x}\;\m{d}y \text{ and }
    w\p{x}=\f{n^{\f{n}{2}}}{t^n}\int_{Q\p{0,n^{-\f12}t}}\tau_{-y}u\p{x}\;\m{d}y.
\end{equation}
By construction we have that $u=v+w$. We estimate $v$ with the Minkowski integral inequality:
\begin{equation}\label{wow}
    \norm{v}_{L^p}=\bp{\int_{\R^n}\Big| \f{n^{\f{n}{2}}}{t^n}\int_{Q\p{0,n^{-\f12}t}}\Delta_yu\p{x}\;\m{d}y\Big|^{p}\;\m{d}x}^{\f{1}{p}}\le\f{n^{\f{n}{2}}}{t^n}\int_{Q\p{0,n^{-\f12}t}}\norm{\Delta_yu}_{L^p}\;\m{d}y\le\omega_p\p{t,u},
\end{equation}
where in the last inequality we used that $Q\big(0,n^{-\f12}t\big)\subseteq\Bar{B\p{0,t}}$. Next we estimate $w$ in $\dot{W}^{1,p}\p{\R^n;\mathbb{K}}$ (see also Lemma~\ref{regularity promotion via averaging}). Let $j\in\cb{1,\dots,n}$. For $z\in\R^n$ we adopt the following notation: the canonical basis of $\R^n$ is the set $\cb{e_1,\dots,e_n}$ and $z=\p{z_j',z_j}$, $z_j\in\R$, $z_j'\in\R^{n-1}$, and $z_j'=\p{z_1,\dots,z_{j-1},z_{j+1},\dots,z_n}$. By a change of coordinates we have that
\begin{equation}
    \int_{\p{-\f12n^{-\f12}t,\f12n^{-\f12}t}}u\p{x_j'+y_j',x_j+y_j}\;\m{d}y_j=\int_{\p{-\f12n^{-\f12}t+x_j,\f12n^{-\f12}t+x_j}}u\p{x_j'+y_j',\tau}\;\m{d}\tau.
\end{equation}
Then for a.e. $x\in\R^{n}$ we have that
\begin{equation}
    \pd_j\int_{\p{-n^{-\f12}t,n^{-\f12}t}}u\p{x_j'+y_j',x_j+y_j}\;\m{d}y_j=\Delta_{n^{-\f12}te_j}u\p{x_j'+y_j',x_j-n^{-\f12}t/2}.
\end{equation}
Thus, upon differentiating under the integral and applying Fubini's theorem, we find that for a.e. $x\in\R^n$
\begin{equation}\label{pufferfish}
    \pd_jw\p{x}=\f{n^{\f{n}{2}}}{t^n}\int_{\tilde{Q}\p{0,n^{-\f12}t}}\Delta_{n^{-\f12}te_j}u\big(x_j'+y_j',x_j-\f12n^{-\f12}t\big)\;\m{d}y_j'.
\end{equation}
Where $\tilde{Q}\big(0,n^{-\f12}t\big)\subset\R^{n-1}$ is the cube centered at zero with sides parallel to the coordinate axes of length $n^{-\f12}t$. Again Minkowski's integral inequality and the fact that $\omega_p\p{\cdot,u}$ is increasing show that
\begin{equation}\label{wowo}
    \norm{\pd_jw}_{L^p}\le\f{n^{\f{n}{2}}}{t^n}\int_{\tilde{Q}\p{0,n^{-\f12}t}}\bnorm{\Delta_{n^{-\f12}te_j}u\bp{\cdot_j'+y_j',\cdot_j-\f12n^{-\f12}t}}_{L^p}\;\m{d}y_j'\le n^{1/2}t^{-1}\omega_p\p{t,u}.
\end{equation}
Synthesizing~\eqref{wow} and~\eqref{wowo}, we deduce that 
\begin{equation}\label{bobo}
    \mathscr{K}\p{t,u}\le\norm{v}_{L^p}+t\textstyle{\sum_{j=1}^n}\norm{\pd_jw}_{L^p}\le (1+n^{3/2})\omega_p\p{t,u}.
\end{equation}
With the first bound and the estimate~\eqref{bobo} in hand, the proof is complete when $p< \infty$.

On the other hand, if $p=\infty$ we again decompose $u=v+w$ as in equation~\eqref{decomp}. In this case it is straightforward to see that $\norm{v}_{L^\infty}\le\omega_\infty\p{t,u}$ and for all $j\in\cb{1,\dots,n}$, $\norm{\pd_jw}_{L^\infty}\le n^{1/2}t^{-1}\omega_\infty\p{t,u}$.
\end{proof}

From this equivalence we can characterize the homogeneous Besov spaces as seminorm interpolation spaces.

\begin{coro}[Interpolation characterization of homogeneous Besov spaces]\label{interpolation characterization of Besov spaces}
For all $s\in\p{0,1}$ and $1\le p,q\le\infty$ we have the equality of seminormed spaces with equivalence of seminorms:
\begin{equation}\label{bear}
    (L^p\p{\R^n;\mathbb{K}};\dot{W}^{1,p}\p{\R^n;\mathbb{K}})_{s,q}=\dot{B}^{s,p}_q\p{\R^n;\mathbb{K}}.
\end{equation}
Consequently, the following hold: $\dot{B}^{s,p}_q\p{\R^n;\mathbb{K}}$ is semi-Banach and $\mathfrak{A}\big(\dot{B}^{s,p}_q\p{\R^n;\mathbb{K}}\big)=\cb{\text{constants}}$; if $p,q<\infty$, then $C^\infty_c\p{\R^n;\mathbb{K}}$ is dense in $\dot{B}^{s,p}_q\p{\R^n;\mathbb{K}}$, we have the inclusion $\dot{B}^{s,p}_q\p{\R^n;\mathbb{K}}\subset\mathscr{S}^\ast\p{\R^n;\mathbb{K}}$, and finally the reiteration formulae
\begin{equation}
    (L^p(\R^n;\mathbb{K}),\dot{B}^{t,p}_r(\R^n;\mathbb{K}))_{s,q}=\dot{B}^{u_0,p}_q(\R^n;\mathbb{K})\text{ and }(\dot{B}^{t,p}_r(\R^n;\mathbb{K}),\dot{W}^{1,p}(\R^n;\mathbb{K}))_{s,q}=\dot{B}^{u_1,p}_q(\R^n;\mathbb{K})
\end{equation}
and
\begin{equation}
    (\dot{B}^{t_0,p}_{r_0}(\R^n;\mathbb{K}),\dot{B}^{t_1,p}_{r_1}(\R^n;\mathbb{K}))_{s,q}=\dot{B}^{u_2,p}_q(\R^n;\mathbb{K})
\end{equation}
hold for $0<t,t_0,t_1<1$, $1\le r,r_0,r_1\le\infty$, $u_0=(1-s)t$, $u_1=(1-s)t+s$, and $u_2=(1-s)t_0+st_1$.
\end{coro}
\begin{proof}
Equation~\ref{bear} is an immediate consequence of Lemma~\ref{K functionals and modulus of continuity} and the definition of the seminorm on $\dot{B}^{s,p}_q$. The several consequences follow from the results on interpolation of seminormed spaces from Section~\ref{interpolation of seminormed spaces}.
\end{proof}

We now explore frequency space characterizations of the homogeneous Besov spaces.

\begin{defn}[Homogeneous Besov-Lipschitz spaces]\label{homogeneous besov-lipschitz spaces} 
Let $s\in\R$, $1<p<\infty$, and $1\le q\le\infty$. Then we define the space ${_\wedge}\dot{B}^{s,p}_q\p{\R^n;\mathbb{K}}=\big\{f\in\mathscr{S}^\ast\p{\R^n;\mathbb{K}}\;:\;\sb{f}_{{_\wedge}\dot{B}^{s,p}_q}<\infty\big\}$, where $\sb{\cdot}_{{_\wedge}\dot{B}^{s,p}_q}:\mathscr{S}^\ast\p{\R^n;\C}\to\sb{0,\infty}$ is given by
\begin{equation}
\sb{f}_{{_\wedge}\dot{B}^{s,p}_q} =\norm{\cb{2^{sj}\norm{\uppi_jf}_{L^p}}_{j\in\Z}}_{\ell^q\p{\Z}}.
\end{equation}
\end{defn}

The following theorem shows that the theory of seminormed space interpolation applied to pairs of Riesz potential spaces yields the homogeneous Besov-Lipschitz spaces. We note that the following result appears as Theorem 6.3.1 in~\cite{MR0482275}, where the proof is abbreviated.

\begin{thm}\label{fourier characterization of besov spaces part 1}
Let  $1<p<\infty$ and $s_0,s_1\in\R$ with $s_0<s_1$. Then for $\alpha \in\p{0,1}$ and $1\le q\le\infty$ we have the equality of seminormed spaces with equivalence of seminorms:
\begin{equation}
    (\dot{H}^{s_0,p}\p{\R^n;\mathbb{K}},\dot{H}^{s_1,p}\p{\R^n;\mathbb{K}})_{\al,q}={_\wedge}\dot{B}^{s,p}_q\p{\R^n;\mathbb{K}},\text{ where } s=\p{1-\al}s_0+\al s_1.
\end{equation}
\end{thm}
\begin{proof}
The pair of seminormed spaces $\dot{H}^{s_0,p}\p{\R^n;\mathbb{K}}$ and $\dot{H}^{s_1,p}\p{\R^n;\mathbb{K}}$ are weakly compatible as witnessed by the space of tempered distributions. Suppose that $f\in\big(\dot{H}^{s_0,p}\p{\R^n;\mathbb{K}},\dot{H}^{s_1,p}\p{\R^n;\mathbb{K}}\big)_{\al,q}$. Let $f=f_0+f_1$ be a decomposition with $f_\ell\in\dot{H}^{s_\ell,p}\p{\R^n;\mathbb{K}}$ for $\ell\in\cb{0,1}$. We first claim that we have the universal bound
\begin{equation}\label{equation good bound}
    \norm{\uppi_jf_\ell}_{L^p}\lesssim_{n,p,\psi}2^{-s_\ell j}\sb{f_\ell}_{\dot{H}^{s_\ell,p}},\text{ for }j\in\Z\text{ and }\ell\in\cb{0,1}.
\end{equation}
This follows from the Littlewood-Paley characterization of the Riesz potential spaces from Theorem~\ref{littlewood paley characterization of Riesz potential spaces}. Thus,
\begin{equation}
    \norm{\uppi_jf}_{L^p}\lesssim_{n,p,\psi}\textstyle{\sum_{\ell=0}^1}2^{-s_\ell j}\sb{f_\ell}_{\dot{H}^{s_\ell,p}}\imp 2^{sj}\norm{\uppi_jf}_{L^p}\lesssim_{n,p,\psi}2^{\al\p{s_1-s_0}j}\mathscr{K}\big(2^{-\p{s_1-s_0}j},f\big).
\end{equation}
Therefore by Proposition~\ref{basic properties of K for seminormed} and Proposition~\ref{discrete seminorm on truncated interpolation spaces}, 
\begin{equation}
    \sb{f}_{{_\wedge}\dot{B}^{s,p}_q}\lesssim_{n,p,\psi}\snorm{\big\{2^{\al\p{s_1-s_0}j}\mathscr{K}\big(2^{-\p{s_1-s_0}j},f\big)\big\}_{j\in\Z}}_{\ell^q\p{\Z}}\lesssim_{n,p,\psi,s_0,s_1}\sb{f}_{\al,q}.
\end{equation}

On the other hand, suppose that $f\in{_\wedge}\dot{B}^{s,p}_q\p{\R^n;\mathbb{K}}$. Then, we will see that for all $j\in\Z$ it holds that $\uppi_jf\in\Delta\big(\dot{H}^{s_0,p}\p{\R^n;\mathbb{K}},\dot{H}^{s_1,p}\p{\R^n;\mathbb{K}}\big)$. In fact, we claim that the sequence $\cb{\uppi_j f}_{j\in\Z}$ belongs to the discrete decomposition set of $f$, $\tilde{\mathcal{D}}\p{f}$. For $j\in\Z$ and $k\in\cb{0,1}$ we estimate via Lemma~\ref{almost idempotent} and Theorem~\ref{littlewood paley characterization of Riesz potential spaces}:
\begin{equation}
    \norm{\uppi_j f}_{\dot{H}^{s_k,p}}\lesssim_{n,p,\psi}\textstyle{\sum_{\ell=-1}^1}2^{s_k\p{j+\ell}}\norm{\uppi_{j+\ell}\uppi_j f}_{L^p}\lesssim_{n,p,\psi,s_k}2^{s_kj}\norm{\uppi_jf}_{L^p}.
\end{equation}
In turn, we have that
\begin{equation}\label{estimates of J}
    \mathscr{J}\big(2^{-j\p{s_1-s_0}},\uppi_jf\big)=\max\big\{\norm{\uppi_jf}_{\dot{H}^{s_0,p}},2^{-j\p{s_1-s_0}}\norm{\uppi_j f}_{\dot{H}^{s_1,p}}\big\}\lesssim_{n,p,\psi,s_0,s_1}2^{js_0}\norm{\uppi_jf}_{L^p}.
\end{equation}
Then \eqref{estimates of J} and Proposition~\ref{simple properties of J-functional}  imply that the following series converges absolutely:
\begin{multline}
    \textstyle{\sum_{j\in\Z}}\sb{\uppi_jf}_{\Sigma}\le\textstyle{\sum_{j\in\Z}}\min\big\{1,2^{j\p{s_1-s_0}}\big\}\mathscr{J}\big(2^{-j\p{s_1-s_0}},f\big)\\\lesssim_{n,p,\psi,s_0,s_1}\textstyle{\sum_{j\in\Z}}\min\big\{2^{j\p{s_0-s}},2^{j\p{s_1-s}}\big\}2^{sj}\norm{\uppi_jf}_{L^p}\\\le\snorm{\{\min\{2^{j\p{s_0-s}},2^{j\p{s_1-s}}\}\}_{j\in\Z}}_{\ell^{q'}\p{\N}}\snorm{\{2^{sj}\norm{\uppi_jf}_{L^p}\}_{j\in\Z}}_{\ell^q\p{\N}}.
\end{multline}

Next we show that $\lim_{m\to\infty}\big[\sum_{j=-m}^m\uppi_jf-f\big]_{\Sigma}=0$. We decompose $f=f^-+f^+$ where $f^+=\sum_{j\in\N}\uppi_jf$ (and $f^-=f-f^+$) with the series converging in $\mathscr{S}^\ast\p{\R^n;\mathbb{C}}$ by virtue of Lemma~\ref{convergence}.  Both factors in this decomposition are $\mathbb{K}$-valued, thanks to Lemma~\ref{real valued distributions}. Then for $m\in\N\setminus\cb{0,1}$ we have the bound
\begin{equation}
    \ssb{\textstyle{\sum_{j=-m}^m}\uppi_jf-f}_{\Sigma}\le\ssb{\textstyle{\sum_{j=-m}^{-1}}\uppi_jf-f^-}_{\dot{H}^{s_1,p}}+\ssb{\textstyle{\sum_{j=0}^m}\uppi_jf-f^+}_{\dot{H}^{s_0,p}} =: \bf{I}_m+\bf{II}_m.
\end{equation}
We prove that $\lim_{m\to\infty}\bf{I}_m=0$. The argument that $\bf{II}_m\to0$ as $m\to\infty$ follows similarly. With the aide of Lemma~\ref{convergence} and Lemma~\ref{almost idempotent}, we compute the action of the family $\cb{\uppi_k}_{k\in\Z}$ on the expression appearing in $\bf{I}_m$:
\begin{equation}
    \uppi_k\bp{\sum_{j=-m}^{-1}\uppi_jf-f^-}=
    \uppi_k\bp{\sum_{j=-m}^{\infty}\uppi_jf-f} =
    \begin{cases}
    0&k>-m\\
    -\uppi_{-m}\uppi_{-m-1}f&k=-m\\
    -\p{\uppi_{-m-2}+\uppi_{-m-1}}\uppi_{-m-1}f&k=-m-1\\
    -\uppi_kf&k<-m-1
    \end{cases}.
\end{equation}
Thus by Theorem~\ref{littlewood paley characterization of Riesz potential spaces}, Theorem~\ref{mihlin multiplier theorem}, and Lemma~\ref{scaling invariance of fourier multipliers}
\begin{multline}\label{shiba}
    \bf{I}_m\lesssim_{n,p,\psi}\ssb{\textstyle{\sum_{j=-m}^{-1}}\uppi_jf-f^-}_{\dot{H}^{s_1,p}}^{\thicksim}\le 2^{-s_1m}\norm{\uppi_{-m}\uppi_{-m-1}f}_{L^p}+2^{-s_1\p{m+1}}\norm{\p{\uppi_{-m-2}+\uppi_{-m-1}}\uppi_{-m-1}f}_{L^p}\\
    +\bnorm{\bp{\textstyle{\sum_{j=-\infty}^{-m-2}}\p{2^{s_1j}\abs{\uppi_jf}}^2}^{1/2}}_{L^p}\lesssim_{n,p,\psi}\bnorm{\bp{\textstyle{\sum_{j=-\infty}^{-m-1}}\p{2^{s_1j}\abs{\uppi_j f}}^2}^{1/2}}_{L^p}.
\end{multline}
To obtain good bounds on the last expression in~\eqref{shiba} we break to cases on the size of $p$. Suppose first that $1<p\le 2$. In this case the mapping $\R^+\cup\cb{0}\ni\eta\mapsto\eta^{\f{p}{2}}\in\R^+\cup\cb{0}$ is subadditive. Therefore since $s_0<s<s_1$ we may use the inclusion $\ell^{q} \hookrightarrow \ell^{p}$ for $q \le p$ and H\"{o}lder's inequality otherwise to deduce that
\begin{multline}\label{blue}
    \bnorm{\bp{\textstyle{\sum_{j=-\infty}^{-m-1}}\p{2^{s_1j}\abs{\uppi_jf}}^2}^{1/2}}_{L^p}\le\bp{\int_{\R^n}\textstyle{\sum_{j=-\infty}^{-m-1}}\p{2^{s_1j}\abs{\uppi_jf}}^p}^{1/p}\\\le\begin{cases}
    2^{(s_1-s)(m-1)}\p{\sum_{j=-\infty}^{-m-1}\p{2^{sj}\norm{\uppi_jf}_{L^p}}^q}^{1/q}&q\le p\\
    \norm{\cb{2^{j\p{s_1-s}}}_{j=-\infty}^{-m-1}}_{\ell^r}\norm{\cb{2^{sj}\norm{\uppi_j f}_{L^p}}_{j=-\infty}^{-m-1}}_{\ell^q}&p<q,\;\f{1}{p}=\f{1}{q}+\f{1}{r}
    \end{cases}<\infty.
\end{multline}

The finiteness follows from the hypothesis that $f\in{_\wedge}\dot{B}^{s,p}_q\p{\R^n;\mathbb{K}}$. Notice also that the final expression in~\eqref{blue} tends to zero as $m\to\infty$. Thus $\bf{I}_m\to0$ as $m\to\infty$ in the case $1< p\le 2$.

On the other hand, in the case that $2<p<\infty$, we bound via Minkowski's integral inequality:
\begin{multline}
    \bnorm{\bp{\textstyle{\sum_{j=-\infty}^{-m-1}}\p{2^{s_1j}\abs{\uppi_jf}}^2}^{1/2}}_{L^p}\le\bp{\textstyle{\sum_{j=-\infty}^{-m-1}}\p{2^{s_1j}\norm{\uppi_j f}_{L^p}}^2}^{1/2}\\\le\textstyle{\sum_{j=-\infty}^{-m-1}}2^{s_1j}\norm{\uppi_jf}_{L^p}\le\snorm{\big\{2^{sj}\norm{\uppi_jf}_{L^p}\big\}_{j=-\infty}^{-m-1}}_{\ell^q}\snorm{\big\{2^{j\p{s_1-s}}\big\}_{j=-\infty}^{-m-1}}_{\ell^{q'}}.
\end{multline}
This bound again implies that $\bf{I}_m\to0$ as $m\to\infty$.

Thus, we learn that $\cb{\uppi_jf}_{j\in\Z}\in\tilde{\mathcal{D}}\p{f}$. Using the discrete characterization of the $J$-method in Proposition~\ref{discrete seminorm on J-space} and equation~\eqref{estimates of J} we obtain the estimate that completes the proof:
\begin{equation}
    \sb{f}_{\al,q}\lesssim_{\al,q}\snorm{\big\{2^{\al\p{s_1-s_0}j}\mathscr{J}(2^{-j\p{s_1-s_0}},\uppi_jf)\big\}_{j\in\Z}}_{\ell^q\p{\Z}}\lesssim_{n,p,s_0,s_1,\psi}\snorm{\big\{2^{sj}\norm{\uppi_jf}_{L^p}\big\}_{j\in\Z}}_{\ell^q\p{\Z}}=\sb{f}_{{_\wedge}\dot{B}^{s,p}_q}.
\end{equation}
\end{proof}

We can now relate $\dot{B}^{s,p}_q\p{\R^n;\mathbb{K}}$ and ${_\wedge}\dot{B}^{s,p}_q\p{\R^n;\mathbb{K}}$.

\begin{coro}\label{besov fourier char}
For each $s\in\p{0,1}$, $1<p<\infty$, $1\le q\le\infty$, there exists $c\in\R^+$ with the following properties:
\begin{enumerate}
    \item If $f\in\dot{B}^{s,p}_q\p{\R^n;\mathbb{K}}$, then $f\in{_\wedge}\dot{B}^{s,p}_q\p{\R^n;\mathbb{K}}$, and $\sb{f}_{{_\wedge}\dot{B}^{s,p}_q}\le c\sb{f}_{\dot{B}^{s,p}_q}$.
    \item On the other hand if $f\in{_\wedge}\dot{B}^{s,p}_q\p{\R^n;\mathbb{K}}$, then there exists a $\mathbb{K}$-valued polynomial $Q$ such that $f-Q$ is identifiable with a member of $\dot{B}^{s,p}_q\p{\R^n;\mathbb{K}}$ and $c^{-1}\sb{f-Q}_{\dot{B}^{s,p}_q}\le\sb{f}_{{_\wedge}\dot{B}^{s,p}_q}$. Moreover, the coefficients of $Q$, aside from the constant term, are uniquely determined.
\end{enumerate}
\end{coro}
\begin{proof}
The first item follows at once from the embeddings $L^p\p{\R^n;\mathbb{K}}\emb\dot{H}^{0,p}\p{\R^n;\mathbb{K}}$ and $\dot{W}^{1,p}\p{\R^n;\mathbb{K}}\emb\dot{H}^{1,p}\p{\R^n;\mathbb{K}}$ (see Theorems~\ref{littlewood-paley characterization of Lp} and~\ref{frequency space characterization of homogeneous sobolev spaces}) and the interpolation characterizations of $\dot{B}^{s,p}_q\p{\R^n;\mathbb{K}}$ (Corollary~\ref{interpolation characterization of Besov spaces}) and ${_\wedge}\dot{B}^{s,p}_q\p{\R^n;\mathbb{K}}$ (Theorem~\ref{fourier characterization of besov spaces part 1}).

For the second item, we let $f\in{_\wedge}\dot{B}^{s,p}_q\p{\R^n;\mathbb{K}}$. The finiteness of $\sb{f}_{{_\wedge}\dot{B}^{s,p}_q}$ implies that for each $j\in\Z$ we have $\uppi_j f\in L^p(\R^n;\mathbb{K})$, and so then Theorem~\ref{mihlin multiplier theorem} implies that $\uppi_j f\in\dot{W}^{1,p}\p{\R^n;\mathbb{K}}$. Using the Littlewood-Paley characterization of $\dot{W}^{1,p}\p{\R^n;\mathbb{K}}$ (Theorem~\ref{frequency space characterization of homogeneous sobolev spaces}) and the almost orthogonality of $\cb{\uppi_j}_{j\in\Z}$ (see Lemma~\ref{almost idempotent}) we learn that  
\begin{equation}\label{avacado}
    \sb{\uppi_jf}_{\dot{W}^{1,p}}\lesssim_{n,p,\psi}\bnorm{\bp{\textstyle{\sum_{k\in\Z}}\p{2^{k}\abs{\uppi_k\uppi_jf}}^2}^{1/2}}_{L^p}\lesssim_{n,p,\psi}2^{j}\norm{\uppi_jf}_{L^p},
\end{equation}
where here we can neglect any polynomial terms on the left when using Theorem~\ref{frequency space characterization of homogeneous sobolev spaces} since  $\uppi_j f\in\dot{W}^{1,p}\p{\R^n;\mathbb{K}}$.  Consequently we have the absolute convergence:
\begin{multline}
    \textstyle{\sum_{j\in\Z}}\sb{\uppi_jf}_{\Sigma\p{L^p,\dot{W}^{1,p}}}\le\textstyle{\sum_{j\in\Z}}\min\cb{1,2^{j}}\max\cb{\norm{\uppi_jf}_{L^p},2^{-j}\sb{\uppi_jf}_{\dot{W}^{1,p}}}\\\lesssim_{n,p,\psi}\textstyle{\sum_{j\in\Z}}\min\cb{1,2^j}\norm{\uppi_jf}_{L^p}\le\snorm{\big\{\min\{2^{-sj},2^{\p{1-s}j}\}\big\}_{j\in\Z}}_{\ell^{q'}\p{\Z}}\sb{f}_{{_\wedge}\dot{B}^{s,p}_q}<\infty.
\end{multline}
Proposition~\ref{completeness of sum and intersection} ensures  that $\Sigma\big(L^p\p{\R^n;\mathbb{K}},\dot{W}^{1,p}\p{\R^n;\mathbb{K}}\big)$ is semi-Banach. Hence, there exists $\tilde{f}$ belonging to this sum such that $\ssb{\sum_{j=-m}^m\uppi_jf-\tilde{f}}_{\Sigma(L^p,\dot{W}^{1,p})}\to0$ as $m\to\infty$.  Moreover, $\cb{\uppi_jf}_{j\in\Z}$ belongs to the discrete decomposition set of $\tilde{f}$, so we are free to estimate the seminorm of $\tilde{f}$ in the interpolation space $\dot{B}^{s,p}_q\p{\R^n;\mathbb{K}}=\big(L^p\p{\R^n;\mathbb{K}},\dot{W}^{1,p}\p{\R^n;\mathbb{K}}\big)_{s,q}$ via the discrete characterization of the $J$-method (see Proposition~\ref{discrete seminorm on J-space}) and \eqref{avacado}:
\begin{equation}
    \ssb{\tilde{f}}_{\dot{B}^{s,p}_q}\lesssim_{s,q}\snorm{\big\{2^{sj}\mathscr{J}\p{2^{-j},\uppi_jf}\big\}_{j\in\Z}}_{\ell^q\p{\N}}\lesssim_{n,p,\psi}\snorm{\cb{2^{sj}\norm{\uppi_j f}_{L^p}}_{j\in\Z}}_{\ell^q\p{\N;\R}}=\sb{f}_{{_\wedge}\dot{B}^{s,p}_q}.
\end{equation}
It remains to show that $f$ and $\tilde{f}$ differ by a polynomial. As the family of operators $\cb{\uppi_j}_{j\in\Z}$ are continuous on both $L^p(\R^n;\mathbb{K})$ and $\dot{W}^{1,p}(\R^n;\mathbb{K})$ (and hence their sum) we find that $\ssb{\uppi_j\sum_{k=-m}^m\uppi_jf-\uppi_j\tilde{f}}_{\Sigma(L^p,\dot{W}^{1,p})}\to0$ as $m\to\infty$ for each $j\in\Z$. But if $m>\abs{j}$ then Lemma~\ref{almost idempotent} tells us that $\uppi_j\sum_{k=-m}^m\uppi_kf=\uppi_jf$. Therefore \begin{equation}
    \uppi_j\big(f-\tilde{f}\big)\in\mathfrak{A}\bp{\Sigma\p{L^p\p{\R^n;\mathbb{K}},\dot{W}^{1,p}\p{\R^n;\mathbb{K}}}}=\cb{\text{constant functions}}\text{ for all }j\in\Z.
\end{equation}
Since $\supp\mathscr{F}\uppi_j\p{f-\tilde{f}}\subset\R^n\setminus\Bar{B\p{0,2^{j-2}}}$ and constant functions are supported at the origin on the Fourier side, we must have $\uppi_jf=\uppi_j\tilde{f}$ for all $j\in\Z$. Thus Lemma~\ref{convergence} provides us a $\mathbb{K}$-valued polynomial $Q$ such that $f-Q=\tilde{f}\in\dot{B}^{s,p}_q\p{\R^n;\mathbb{K}}$.

If $P,Q$ are two polynomials such that $f-Q,f-P\in\dot{B}^{s,p}_q\p{\R^n;\mathbb{K}}$, then $P-Q\in\dot{B}^{s,p}_q\p{\R^n;\mathbb{K}}$ is a polynomial which implies that $P-Q\in\dot{B}^{s,p}_q\p{\R^n;\mathbb{K}}$ is a constant.
\end{proof}

\section{Screened Sobolev and screened Besov spaces}\label{Screened Sobolev and screened Besov spaces}

Recall from the introduction that \cite{leoni2019traces} defines the screened Sobolev space $\tilde{W}^{s,p}_{\p{\sig}}\p{U}$ as the collection of locally integrable functions $f: U \to \R$ for which \eqref{screed space defn} holds.  In this section we introduce a generalized scale of spaces, the screened Besov spaces, and use our previous seminorm interpolation theory to study their properties.

\subsection{Motivation, definitions, and basic properties}

In an effort to better understand the screened Sobolev spaces, we introduce the following scale of screened Besov spaces with constant screening function.

\begin{defn}[Screened Besov spaces]\label{screened besov spaces defn}
Let $\mathbb{K}\in\cb{\R,\C}$,  $1\le p,q\le\infty$, $s\in\p{0,1}$, $\sig\in\R^+$, and let $\es\neq U\subseteq\R^n$ be open.  For $h\in\R^n$ write $U_h=U\cap\tau_{-h}U$.  We define the extended seminorm  $\sb{\cdot}^{\p{\sig}}_{\tilde{B}^{s,p}_q}:L^1_\loc\p{\R^n;\mathbb{K}}\to\sb{0,\infty}$ via \begin{equation}
   \sb{f}^{\p{\sig}}_{\tilde{B}^{s,p}_q}=
   \begin{cases}
    \p{\int_{B\p{0,\sig}}\p{\abs{h}^{-s}\norm{\Delta_h f}_{L^p\p{U_h;\mathbb{K}}}}^q\abs{h}^{-n}\;\m{d}h}^{1/q}&1\le q<\infty\\
    \m{esssup}\cb{\abs{h}^{-s}\norm{\Delta_hf}_{L^p\p{U_h;\mathbb{K}}}\;:\;h\in B\p{0,\sig}}&q<\infty
    \end{cases}.
\end{equation}
The screened Besov space, ${_{\p{\sig}}}\tilde{B}^{s,p}_q\p{U;\mathbb{K}}$, is the subspace of $L^1_\loc\p{U;\mathbb{K}}$ on which the above seminorm is finite. When $\sig=1$ we write $\tilde{B}^{s,p}_q\p{U;\mathbb{K}}$ and  $\sb{\cdot}_{\tilde{B}^{s,p}_q}$ in place of ${_{\p{1}}}\tilde{B}^{s,p}_q\p{U;\mathbb{K}}$ and $\sb{\cdot}^{\p{1}}_{\tilde{B}^{s,p}_q}$.
\end{defn}

We begin by providing an equivalent seminorm that utilizes the $L^p-$modulus of continuity.

\begin{prop}\label{constant screening on rn equiv seminorm}
Let   $1\le p,q\le \infty$ and $\omega_p$ be the $L^p$-modulus of continuity from Definition~\ref{difference quotients and moduli of continuity}.  Then for all $f\in L^1_\m{loc}\p{\R^n;\mathbb{K}}$ and all $s\in\p{0,1}$, $\sig\in\R^+$ we have the equivalence
\begin{equation}\label{equiv norm eqn 1}
    \sb{f}^{\p{\sig}}_{\tilde{B}^{s,p}_{q}\p{\R^n;\mathbb{K}}}\asymp_{n,s}\begin{cases}\p{\int_{\p{0,\sig}}\p{t^{-s}\omega_p\p{t,f}}^qt^{-1}\;\m{d}t}^{1/q}&1\le q<\infty\\
    \sup\cb{t^{-s}\omega_p\p{t,f}\;:\;t\in\sb{0,\sig}}&q=\infty
    \end{cases}.
\end{equation}
\end{prop}
\begin{proof}
The result is trivial when $q=\infty$, so we only prove the case $1\le q<\infty$.

Using spherical coordinates we write:
\begin{equation}\label{spherical}
    \sb{f}^{\p{\sig}}_{\tilde{B}^{s,p}_{q}}=\bp{\int_{B\p{0,\sig}}\p{\abs{h}^{-s}\norm{\Delta_h f}_{L^p}}^q\abs{h}^{-n}\;\m{d}h}^{1/q}
    =\bp{\int_{\p{0,\sig}}\int_{\pd B\p{0,1}}\norm{\Delta_{tz} f}_{L^p}^q\;\m{d}\mathcal{H}^{n-1}(z)t^{-1-sq}\;\m{d}t}^{\f1q}.
\end{equation}
The `$\lesssim$' inequality in \eqref{equiv norm eqn 1} then follows from~\eqref{spherical} and the simple bound
\begin{equation}
    \int_{\p{0,\sig}}\int_{\pd B\p{0,1}}\norm{\Delta_{tz }f}_{L^p}^q\;\m{d}\mathcal{H}^{n-1}(z) t^{-1-sq}\;\m{d}t
    \le\mathcal{H}^{n-1}\p{\pd B\p{0,1}}\int_{\p{0,\sig}}\p{t^{-s}\omega_p\p{t,f}}^qt^{-1}\;\m{d}t.
\end{equation}

For the `$\gtrsim$' inequality in \eqref{equiv norm eqn 1}, we pick $t\in\p{0,\sig}$ and let $h\in B\p{0,t}\setminus\cb{0}$ and $\xi\in B\p{h/2,\abs{h}/2}$. Observe that $\Delta_h f=\Delta_\xi f + \tau_{-\xi}\Delta_{h-\xi} f$, and hence $\norm{\Delta_h f}_{L^p}\le\norm{\Delta_\xi f}_{L^p}+\norm{\Delta_{h-\xi}f}_{L^p}$.
 We then average over $\xi\in B\p{h/2,\abs{h}/2}$ and use a change of variables to arrive at the bounds
\begin{multline}
    \norm{\Delta_h f}_{L^p}\le2^n\Le^n\p{B\p{0,1}}\abs{h}^{-n}\int_{B\p{h/2,\abs{h}/2}}\p{\norm{\Delta_\xi f}_{L^p}+\norm{\Delta_{h-\xi}f}_{L^p}}\;\m{d}\xi\\\le2^{n+1}\Le^n\p{B\p{0,1}}\abs{h}^{-n}\int_{B\p{0,\abs{h}}}\norm{\Delta_\xi f}_{L^p}\;\m{d}\xi\le 2^{n+1}\Le^n\p{B\p{0,1}}\int_{B\p{0,\abs{h}}}\norm{\Delta_\xi f}_{L^p}\abs{\xi}^{-n}\;\m{d}\xi\\
    \le2^{n+1}\Le^n\p{B\p{0,1}}\int_{B\p{0,t}}\norm{\Delta_\xi f}_{L^p}\abs{\xi}^{-n}\;\m{d}\xi.
\end{multline}
In this expression we take the supremum over $h\in B\p{0,t}$, raise the result to the $q^{\m{th}}$ power, then multiply by $t^{-1-sq}$, and finally integrate over $\p{0,\sig}$; this results in the following chain of inequalities, in which we also employ Lemma~\ref{truncated hardy inequality}, H\"older's inequality, and  \eqref{spherical}:
\begin{multline}
    \int_{\p{0,\sig}}\p{t^{-s}\omega_p\p{t,f}}^qt^{-1}\;\m{d}t\le\p{2^{n+1}\Le^n\p{B\p{0,1}}}^q\int_{\p{0,\sig}}\bp{\int_{B\p{0,t}}\norm{\Delta_\xi f}_{L^p} \abs{\xi}^{-n}\;\m{d}\xi}^qt^{-1-sq}\;\m{d}t\\
    =\p{2^{n+1}\Le^n\p{B\p{0,1}}}^q\int_{\p{0,\sig}}\bp{\int_{\p{0,t}}\int_{\pd B\p{0,1}}\norm{\Delta_{\rho z} f}_{L^p}\; \m{d}\mathcal{H}^{n-1}(z) \; \rho^{-1}\m{d}\rho}^{q}t^{-1-sq}\;\m{d}t\\
    \le\p{s^{-1}2^{n+1}\Le^n\p{B\p{0,1}}}^q\int_{\p{0,\sig}}\bp{\int_{\pd B\p{0,1}}\norm{\Delta_{t z }f}_{L^p}\;\m{d}\mathcal{H}^{n-1}(z)}^qt^{-1-sq}\;\m{d}t\\\le\p{s^{-1}2^{n+1}\Le^n\p{B\p{0,1}}}^q\p{\mathcal{H}^{n-1}\p{\pd B\p{0,1}}}^{q-1}\big(\sb{f}^{\p{\sig}}_{\tilde{B}^{s,p}_{q}}\big)^q.
\end{multline}
\end{proof}

Proposition~\ref{constant screening on rn equiv seminorm} leads us to define the following equivalent extended seminorm.

\begin{defn}\label{equivalent seminorm in terms on moduli of continuity} 
Let $1\le p,q\le\infty$, $\sig\in\R^+$, and $s\in\p{0,1}$.  We define ${^{\circ}}\sb{\cdot}_{\tilde{B}^{s,p}_{q}}^{\p{\sig}}: L^1_\loc\p{\R^n;\mathbb{K}} \to \sb{0,\infty}$  via
\begin{equation}
\quad
 {^{\circ}}\sb{f}_{\tilde{B}^{s,p}_{q}}^{\p{\sig}} =
 \begin{cases} 
 \p{\int_{\p{0,\sig}}\p{t^{-s}\omega_p\p{t,f}}^qt^{-1}\;\m{d}t}^{1/q}&1\le q<\infty\\
    \sup\cb{t^{-s}\omega_p\p{t,f}\;:\;t\in\sb{0,\sig}}&q=\infty
    \end{cases}.
\end{equation}
Note that Proposition~\ref{constant screening on rn equiv seminorm} ensures that this seminorm is equivalent to the one from Definition~\ref{screened besov spaces defn}. 
\end{defn}

\subsection{Interpolation characterization of screened Besov spaces}

Using the equivalent seminorm on the space ${_{\p{\sig}}}\tilde{B}^{s,p}_{q}(\R^n;\mathbb{K})$ from Definition~\ref{equivalent seminorm in terms on moduli of continuity}, we can realize that the $s,q,\sig$-truncated interpolation space between $L^p(\R^n;\mathbb{K})$ and $\dot{W}^{1,p}(\R^n;\mathbb{K})$ is equal to the screened space ${_{\p{\sig}}}\tilde{B}^{s,p}_{q}(\R^n;\mathbb{K})$. Precisely, we have the following theorem.

\begin{thm}[Interpolation characterization of screened spaces]\label{screened sobolev spaces are truncated interpolation spaces}
Let $1\le p,q\le\infty$, $s\in\p{0,1}$, and $\sig\in\R^+$. Then we have the equality of vector spaces with equivalent seminorms:
\begin{equation}\label{this is an equation}
    {_{\p{\sig}}}\tilde{B}^{s,p}_{q}\p{\R^n;\mathbb{K}}=\big(L^p\p{\R^n;\mathbb{K}},\dot{W}^{1,p}\p{\R^n;\mathbb{K}}\big)^{\p{\sig}}_{s,q}.
\end{equation}
In fact, for all $f\in\Sigma\big(L^p\p{\R^n;\mathbb{K}},\dot{W}^{1,p}\p{\R^n;\mathbb{K}}\big)$ we have the equivalence
\begin{equation}\label{this is an equivalence}
    2^{-1}{^{\circ}}\sb{f}^{\p{\sig}}_{\tilde{B}^{s,p}_{q}}\le\sb{f}_{s,q}^{\p{\sig}}\le \big(1+n^{3/2}\big) {^{\circ}}\sb{f}^{\p{\sig}}_{\tilde{B}^{s,p}_{q}}.
\end{equation}
\end{thm}
\begin{proof}
Let $\mathscr{K}$ denote the $K$-functional on the sum of $L^p\p{\R^n;\mathbb{K}}$ and $\dot{W}^{1,p}\p{\R^n;\mathbb{K}}$ and note that the strong compatibility of these spaces is shown in Lemma~\ref{compatibility of homog and lebesgue}. It is sufficient to observe that for all $t\in\p{0,\sig}$ and all $f\in\Sigma\big(L^p\p{\R^n;\mathbb{K}},\dot{W}^{1,p}\p{\R^n;\mathbb{K}}\big)$ we have the equivalence
\begin{equation}\label{this is also an equivalence}
    2^{-1}\omega_p\p{t,f}\le\mathscr{K}\p{t,f}\le \big(1+n^{3/2}\big)\omega_p\p{t,f}.
\end{equation}
This is a consequence of Lemma~\ref{K functionals and modulus of continuity}.
\end{proof}

This interpolation characterization has numerous important and useful corollaries that we can read off from the abstract theory of seminorm interpolation presented previously.  The first is that we can now can build an explicit bridge to well-studied function spaces.

\begin{coro}[Sum characterization of screened spaces]\label{sum characterization of screened sobolev spaces}
Let  $s\in\p{0,1}$ and $1\le p,q\le\infty$. The following hold:
\begin{enumerate}
    \item If $\sig,\tau\in\R^+$, then we have the equality of vector spaces with equivalence of seminorms:
    \begin{equation}
        {_{\p{\sig}}}\tilde{B}^{s,p}_q\p{\R^n;\mathbb{K}}=\Sigma\p{\dot{B}^{s,p}_q\p{\R^n;\mathbb{K}};\dot{W}^{1,p}\p{\R^n;\mathbb{K}}}={_{\p{\tau}}}\tilde{B}^{s,p}_q\p{\R^n;\mathbb{K}}.
    \end{equation}
    \item If $p<\infty$ and $\sig:\R^n\to\R^+$ is a screening function with $\log\sig$ a bounded function,  then we have the equality of spaces with equivalence of seminorms:
\begin{equation}
    \tilde{W}^{s,p}_{\p{\sig}}\p{\R^n;\R}=\Sigma\p{\dot{B}^{s,p}_p\p{\R^n;\R},\dot{W}^{1,p}\p{\R^n;\R}}.
\end{equation}
\end{enumerate}
\end{coro}
\begin{proof}
Given Corollary~\ref{screened sobolev spaces are truncated interpolation spaces}, the first item is immediate from Theorem~\ref{label sum characterization of the truncated method}. For the second item we set $\sig_+=\sup\sig$ and $\sig_-=\inf\sig$.  By hypothesis, these are both positive.  It is a simple matter to observe that:
\begin{equation}
    {_{\p{\sig_+}}}\tilde{B}^{s,p}_p\p{\R^n;\R}=\tilde{W}^{s,p}_{\p{\sig_{+}}}\p{\R^n;\R} \emb \tilde{W}^{s,p}_{\p{\sig}}\p{\R^n;\R} \emb \tilde{W}^{s,p}_{\p{\sig_-}}\p{\R^n;\R}={_{\p{\sig_-}}}\tilde{B}^{s,p}_p\p{\R^n;\R}.
\end{equation} Thus the second item follows from the first.
\end{proof}

The next corollary shows us when we have density of smooth and compactly supported functions in the screened spaces. This result is, in fact, sharp, as we will see in the next section.

\begin{coro}\label{density of com}
Let  $s\in\p{0,1}$ and $1\le p,q<\infty$.  Then $C^\infty\p{\R^n;\mathbb{K}}\cap\tilde{B}^{s,p}_{q}\p{\R^n;\mathbb{K}}$ is dense in $\tilde{B}^{s,p}_{q}\p{\R^n;\mathbb{K}}$.  Moreover, if $n\ge 2$ or $1<p$, then $C^\infty_c\p{\R^n;\mathbb{K}}$ is dense in $\tilde{B}^{s,p}_{q}\p{\R^n;\mathbb{K}}$.
\end{coro}
\begin{proof}
This is a consequence of Corollary~\ref{sum characterization of screened sobolev spaces}, Lemma~\ref{density of ccinfty in homog}, and Corollary~\ref{density in interpolation spaces}.
\end{proof}

We also learn that the screened spaces are semi-Banach and their annihilator is nothing more than the space of constant functions.

\begin{coro}\label{semi banach and kernels of screened spaces}
Let $s\in\p{0,1}$ and $1\le p,q\le\infty$. Then $\tilde{B}^{s,p}_{q}\p{\R^n;\mathbb{K}}$ is semi-Banach with annihilator $\mathfrak{A}\big(\tilde{B}^{s,p}_q\p{\R^n;\mathbb{K}}\big)=\cb{\text{constant functions}}.$
\end{coro}
\begin{proof}
    This follows from Theorem~\ref{screened sobolev spaces are truncated interpolation spaces}, Propositions~\ref{completeness of truncated spaces} and~\ref{kernels prop}, and finally Lemma~\ref{compatibility of homog and lebesgue}.
\end{proof}

We note that Corollary~\ref{semi banach and kernels of screened spaces} appears in~\cite{leoni2019traces} for the scale of screened Sobolev spaces with general screening functions.

\subsection{A concrete decomposition}

The previous subsection shows that the screened Besov spaces coincide with the sum of a homogeneous Sobolev and a homogeneous Besov space. In either case the seminorms are, at best, tedious to work with. The purpose of this subsection is to show that we achieve a nearly optimal decomposition into the summands in a simple way. We then use this decomposition to show that compactly supported smooth functions are not dense in the space $\tilde{B}^{s,1}_{q}\p{\R;\mathbb{K}}$ for any $s\in\p{0,1}$, $1\le q\le\infty$.

\begin{defn}\label{average and singular parts}
Let  $Q=\p{-1/2,1/2}^n\subset\R^n$ and define the operators $\mathds{H},\mathds{L}:L^1_\loc\p{\R^n;\mathbb{K}}\to L^1_\loc\p{\R^n;\mathbb{K}}$ via 
\begin{equation}
    \mathds{H}f\p{x}=\int_{Q}\p{f\p{x}-f\p{x+y}}\;\m{d}y \text{ and }  \mathds{L}f\p{x}=\int_{Q}f\p{x+y}\;\m{d}y.
\end{equation}
Notice that the sum of $\mathds{H}$ and $\mathds{L}$ is the identity.
\end{defn}

The following theorem utilizes these to arrive at another equivalent seminorm.

\begin{thm}[Fundamental decomposition of screened Besov spaces]\label{fundamental decomposition of screened sobolev spaces}
Let  $1\le p,q\le\infty$ and $s\in\p{0,1}$. There exists a constant $c\in\R^+$ such that for all $f\in L^1_\loc\p{\R^n;\mathbb{K}}$ 
\begin{equation}\label{really silly}
    c^{-1}\sb{f}_{\tilde{B}^{s,p}_{q}}\le\norm{\mathds{H}f}_{B^{s,p}_q}+\sb{\mathds{L}f}_{\dot{W}^{1,p}}\le c\sb{f}_{\tilde{B}^{s,p}_{q}}.
\end{equation}
In particular, we have the equality of seminormed spaces $\tilde{B}^{s,p}_q\p{\R^n;\mathbb{K}}=\Sigma(B^{s,p}_q\p{\R^n;\mathbb{K}};\dot{W}^{1,p}\p{\R^n;\mathbb{K}})$ with equivalence of seminorms.
\end{thm}
\begin{proof}
By the sum characterization of Corollary~\ref{sum characterization of screened sobolev spaces} and the embedding $B^{s,p}_q\p{\R^n;\mathbb{K}} \emb \dot{B}^{s,p}_q\p{\R^n;\mathbb{K}}$, it is sufficient to prove the second inequality in~\eqref{really silly}.  Suppose that $f\in\tilde{B}^{s,p}_{q}\p{\R^n;\mathbb{K}}$. By Lemma~\ref{regularity promotion via averaging}, we have that $\mathds{L}f\in W^{1,1}_\loc\p{\R^n;\mathbb{K}}$ and for $j\in\cb{1,\dots,n}$ and a.e. $x\in\R^n$ it holds that
\begin{equation}
    \pd_j\mathds{L}f\p{x}=\int_{\Pi_jQ\p{x}}\Delta_{e_j}f\p{y_j',x_j-1/2}\;\m{d}y_i',
\end{equation}
where $Q\p{x}=\prod_{k=1}^n\p{-1/2+x_k,1/2+x_k}$ and $\Pi_j=\p{I-e_j\otimes e_j}$.  Then when $1\le p<\infty$ an application of Minkowski's integral inequality, Proposition~\ref{constant screening on rn equiv seminorm}, and   Proposition~\ref{truncated interpolation spaces are intermediate} show that
\begin{equation}
    \norm{\pd_j\mathds{L}f}_{L^p}\le\bp{\int_{\R^n}\abs{\Delta_{e_j}f\p{x}}^p\;\m{d}x}^{1/p}\le\omega_p\p{1,f}\le2\mathscr{K}\p{1,f}\le2q^{-1/q}\p{1-s}^{-1/q}\sb{f}_{s,q}^{\p{1}}.
\end{equation}
When $p=\infty$ it is similarly clear that $\norm{\pd_j\mathds{L}f}_{L^\infty}\le 2\sb{f}^{\p{1}}_{s,\infty}$.
Note that $\mathscr{K}$ is the $K$-functional associated to the sum $\Sigma\big(L^p\p{\R^n;\mathbb{K}},\dot{W}^{1,p}\p{\R^n;\mathbb{K}}\big)$ and that $\sb{\cdot}^{\p{1}}_{s,q}$ is the seminorm on the truncated interpolation space $\big(L^p\p{\R^n;\mathbb{K}},\dot{W}^{1,p}\p{\R^n;\mathbb{K}}\big)_{s,q}^{\p{1}}$. Theorem~\ref{screened sobolev spaces are truncated interpolation spaces} and Corollary~\ref{sum characterization of screened sobolev spaces} ensure us that $\sb{\cdot}_{s,q}^{\p{1}}\lesssim\sb{\cdot}_{\tilde{B}^{s,p}_{q}}$. Hence, $\sb{\mathds{L}f}_{\dot{W}^{1,p}}\le c\sb{f}_{\tilde{B}^{s,p}_{q}}$ for some $c\in\R^+$ depending only on $s$, $p$, and $n$.

With another application of Minkowski's integral inequality, we find that \begin{equation}\label{bungalow}
    \norm{\mathds{H}f}_{L^p}\le\omega_p\p{1,f}\lesssim\sb{f}_{\tilde{B}^{s,p}_{q}}.
\end{equation}
Thus, to show that $\mathds{H}f\in B^{s,p}_q\p{\R^n;\mathbb{K}}$, with a good estimate, it remains to bound $\sb{\mathds{H}f}_{\dot{B}^{s,p}_q}^{\thicksim}$. Note that this seminorm is defined in Definition~\ref{difference quotient norm on besov 2}. First we note that for all $h\in\R^n$ \eqref{bungalow} implies $\norm{\Delta_h\mathds{H}f}_{L^p}\le 2\norm{\mathds{H}f}_{L^p}\lesssim\sb{f}_{\tilde{B}^{s,p}_q}$. Hence, 
\begin{multline}\label{happy}
    \Delta_h\mathds{H}f\p{x}=\int_{Q}\p{f\p{x+h}-f\p{x+h+y}-f\p{x}+f\p{x+y}}\;\m{d}y\\\imp\norm{\Delta_h\mathds{H}f}_{L^p}\lesssim\min\big\{\sb{f}_{\tilde{B}^{s,p}_q},\norm{\Delta_hf}_{L^p}\big\}.
\end{multline}
Now, if $1\le q<\infty$ we use~\eqref{happy} to estimate
\begin{multline}
    \sb{\mathds{H}f}_{\dot{B}^{s,p}_q}^{\thicksim}\le\bp{\int_{\R^n}\p{\abs{h}^{-s}\norm{\Delta_h\mathds{H}f}_{L^p}}^q\abs{h}^{-n}\;\m{d}h}^{1/q}\\\lesssim\sb{f}_{\tilde{B}^{s,p}_q}\bp{\int_{\R^n\setminus B\p{0,1}}\abs{h}^{-n-sq}\;\m{d}h}^{1/q}+\bp{\int_{B\p{0,1}}\p{\abs{h}^{-s}\norm{\Delta_hf}_{L^p}}^q\abs{h}^{-n}\;\m{d}h}^{1/q}\lesssim\sb{f}_{\tilde{B}^{s,p}_q}.
\end{multline}
The same estimates work when $q=\infty$ as well.

Now that~\eqref{really silly} is established, the embedding $\tilde{B}^{s,p}_q\p{\R^n;\mathbb{K}}\emb\Sigma(B^{s,p}_q\p{\R^n;\mathbb{K}};\dot{W}^{1,p}\p{\R^n;\mathbb{K}})$ is clear. The opposite embedding is a consequence of Corollary~\ref{sum characterization of screened sobolev spaces} item (1) and the embedding $B^{s,p}_q\p{\R^n;\mathbb{K}} \emb \dot{B}^{s,p}_q\p{\R^n;\mathbb{K}}$.
\end{proof}

As a corollary, we show that the density of compactly supported continuous functions fails in the cases not covered by Corollary~\ref{density of com}.

\begin{coro}\label{lack of density}
Let $s\in\p{0,1}$ and $1\le q\le\infty$. Then
\begin{equation}
    \tilde{B}^{s,1}_{q}\p{\R;\mathbb{K}}\setminus\Bar{C^0_c\p{\R;\mathbb{K}}\cap\tilde{B}^{s,1}_{q}\p{\R;\mathbb{K}}}^{\tilde{B}^{s,1}_{q}}\neq\es.
\end{equation}
\end{coro}
\begin{proof}
Take $\chi\in L^1\p{\R;\mathbb{K}}$ with $\int_{\R}\chi=1$ and let $u:\R\to\mathbb{K}$ be defined via $u\p{x}=\int_{\p{-\infty,x}}\chi\p{t}\;\m{d}t$.  Notice that $u\in\dot{W}^{1,1}\p{\R;\mathbb{K}}$ and hence $u\in\tilde{B}^{s,1}_{q}\p{\R;\mathbb{K}}$ by Corollary~\ref{sum characterization of screened sobolev spaces}. Suppose, for the sake of contradiction, that there exists a sequence $\cb{u_m}_{m\in\N}\subset C^0_c\p{\R;\mathbb{K}}\cap\tilde{B}^{s,1}_{q}\p{\R;\mathbb{K}}$ with the property that $\lim_{m\to\infty}\sb{u_m-u}_{\tilde{B}^{s,1}_{q}}=0$. Theorem~\ref{fundamental decomposition of screened sobolev spaces} then implies that $\lim_{m\to\infty}\sb{\mathds{L}u_m-\mathds{L}u}_{\dot{W}^{1,1}}=0$. The compact support of each $u_m$ would then tell us that $\mathds{L}u_m\in C^1_c\p{\R;\mathbb{K}}$.  Hence, the fundamental theorem of calculus implies that
\begin{equation}
    1=\int_{\R}\p{\mathds{L}u}'=\lim_{m\to\infty}\int_{\R}\p{\mathds{L}u_m}'=0,
\end{equation}
a contradiction.  This shows that $u$ cannot belong to the closure of $C^0_c\p{\R;\mathbb{K}}\cap\tilde{B}^{s,1}_{q}\p{\R;\mathbb{K}}$.
\end{proof}
\subsection{Frequency space characterizations}

Our goal in this subsection is to synthesize the sum characterization of the screened Besov spaces and the frequency characterizations of the Riesz potential and Besov-Lipschitz spaces. We find that the `low mode' part of the function behaves no worse than a general $\dot{W}^{1,p}$ function while the `high mode' part behaves like a general $B^{s,p}_q$ function. To achieve this we will generalize yet again and characterize the frequency behavior of truncated interpolation spaces between certain pairs of Riesz potential space pairs. We then read off the specifics for the screened Besov spaces.

\begin{defn}\label{fourier seminorm for screened spaces general p}
Recall that for $j\in\Z$ the operators $\{\uppi_j\}_{j \in \Z}$ are given in Definition \ref{dyadic localization}.  Let  $1<p<\infty$, $1\le q\le\infty$, and $r,s\in\R$.  We define $\sb{\cdot}_{\tilde{B}^{r,s}_{p,q}},\sb{\cdot}_{\tilde{H}^{r,s}_{p,q}}:\mathscr{S}^\ast\p{\R^n;\mathbb{K}}\to\sb{0,\infty}$ via
\begin{equation}
    \sb{f}_{\tilde{B}^{r,s}_{p,q}} = \bnorm{\bp{\textstyle{\sum_{j\in\Z\setminus\N}}\p{2^{sj}\abs{\uppi_j f}}^2}^{1/2}}_{L^p}+\snorm{\cb{2^{rj}\norm{\uppi_jf}_{L^p}}_{j\in\N}}_{\ell^q\p{\N;\R}}
\end{equation}
and
\begin{equation}
    \sb{f}_{\tilde{H}^{r,s}_{p,q}} =\snorm{\cb{2^{rj}\norm{\uppi_jf}_{L^p}}_{j\in\Z\setminus\N}}_{\ell^q\p{\Z\setminus\N;\R}}+\bnorm{\bp{\textstyle{\sum_{j\in\N}}\p{2^{sj}\abs{\uppi_j f}}^2}^{1/2}}_{L^p}.
\end{equation}
We define $\tilde{B}^{r,s}_{p,q}\p{\R^n;\mathbb{K}}$ and $\tilde{H}^{r,s}_{p,q}\p{\R^n;\mathbb{K}}$ to be the subspaces of $\mathscr{S}^\ast\p{\R^n;\mathbb{K}}$ on which $\sb{\cdot}_{\tilde{B}^{r,s}_{p,q}}$ and $\sb{\cdot}_{\tilde{H}^{r,s}_{p,q}}$ are finite, respectively.  We refer to these scales as the generalized screened Besov spaces and the generalized screened Riesz-potential spaces.
\end{defn}

The following theorem characterizes these spaces as interpolation spaces.

\begin{thm}[Truncated interpolation of Riesz potential spaces]\label{truncated interpolation of Riesz potential spaces}
Let $1<p<\infty$, $1\le q\le\infty$, $\al\in\p{0,1}$, $\sigma \in \R^+$, and $r,s\in\R$ with $r<s$. Then we have the equality of spaces with equivalence of seminorms:
\begin{equation}\label{wow so many indicies}
    \p{\dot{H}^{r,p}\p{\R^n;\mathbb{K}},\dot{H}^{s,p}\p{\R^n;\mathbb{K}}}^{\p{\sigma}}_{\al,q}=\tilde{B}^{t,s}_{p,q}\p{\R^n;\mathbb{K}},\text{ where }t=\p{1-\al}r+\al s.
\end{equation}
If, on the other hand, we suppose that $s<r$, then we have the equality of spaces with equivalence of seminorms:
\begin{equation}\label{wow so many indicies AGAIN}
    \p{\dot{H}^{r,p}\p{\R^n;\mathbb{K}};\dot{H}^{s,p}\p{\R^n;\mathbb{K}}}_{\al,q}^{\p{\sigma}}=\tilde{H}^{t,s}_{p,q}\p{\R^n;\mathbb{K}},\text{ where }t=\p{1-\al}r+\al s.
\end{equation}
\end{thm}
\begin{proof}
We will prove only \eqref{wow so many indicies}, as \eqref{wow so many indicies AGAIN} follows from similar arguments. 

Let  $\sb{\cdot}_{\Sigma} : \Sigma\big({_\wedge}\dot{B}^{t,p}_q\p{\R^n;\mathbb{K}},\dot{H}^{s,p}\p{\R^n;\mathbb{K}}\big) \to [0,\infty]$ via
\begin{equation}\label{sum sum sum}
\sb{f}_{\Sigma} =\inf\cb{\sb{f_0}_{\dot{B}^{t,p}_q}+\sb{f_1}_{\dot{H}^{s,p}}\;:\;f=f_0+f_1,\;\p{f_0,f_1}\in{_\wedge}\dot{B}^{t,p}_q\p{\R^n;\mathbb{K}}\times\dot{H}^{s,p}\p{\R^n;\mathbb{K}}}.
\end{equation}
The sum characterization of the truncated interpolation spaces, Theorem~\ref{label sum characterization of the truncated method}, and the interpolation theorem of Riesz potential spaces, Theorem~\ref{fourier characterization of besov spaces part 1}, ensure the equality of seminormed spaces with equivalence of seminorms:
\begin{equation}\label{kitty}
    \p{\dot{H}^{r,p}\p{\R^n;\mathbb{K}};\dot{H}^{s,p}\p{\R^n;\mathbb{K}}}^{\p{\sig}}_{\al,q}=\Sigma\p{{_\wedge}\dot{B}^{t,p}_q\p{\R^n;\mathbb{K}},\dot{H}^{s,p}\p{\R^n;\mathbb{K}}}.
\end{equation}
Therefore, $\sb{\cdot}_{\Sigma}$ is an equivalent seminorm on the truncated interpolation space on the left side of~\eqref{kitty}.

Now let $f\in\big(\dot{H}^{r,p}(\R^n;\mathbb{K}),\dot{H}^{s,p}(\R^n;\mathbb{K})\big)^{\p{\sig}}_{\al,q}$ and decompose $f=g+h$ with $g\in{_\wedge}\dot{B}^{t,p}_q(\R^n;\mathbb{K})$ and $h\in\dot{H}^{s,p}(\R^n;\mathbb{K})$. We estimate each factor in the space $\tilde{B}^{t,s}_{p,q}(\R^n;\mathbb{K})$, beginning with $g$:
\begin{multline}\label{expression for g}
    \sb{g}_{\tilde{B}^{t,s}_{p,q}}=\bnorm{\bp{\textstyle{\sum_{j\in\Z\setminus\N}}\p{2^{sj}\abs{\uppi_jg}}^2}^{1/2}}_{L^p}+\snorm{\cb{2^{tj}\norm{\uppi_jg}_{L^p}}_{j\in\N}}_{\ell^q\p{\N}}\\\le\bnorm{\bp{\textstyle{\sum_{j\in\Z\setminus\N}}\p{2^{sj}\abs{\uppi_jg}}^2}^{1/2}}_{L^p}+\sb{g}_{{_\wedge}\dot{B}^{t,p}_{q}}.
\end{multline}
To handle the remaining term controlling the low modes, we can break into cases on the size of $p$. First, suppose that $1<p\le 2$. Then the mapping $\R^+\cup\cb{0}\ni\eta\to\eta^{p/2}\in\R^+\cup\cb{0}$ is subadditive, and hence $t<s$ implies that
\begin{multline}
    \bnorm{\bp{\textstyle{\sum_{j\in\Z\setminus\N}}\p{2^{sj}\abs{\uppi_jg}}^2}^{1/2}}_{L^p}=\bp{\displaystyle{\int_{\R^n}}\p{\textstyle{\sum_{j\in\Z\setminus\N}}\p{2^{js}\abs{\uppi_j g}}^2}^{p/2}}^{1/p}\le\bp{\textstyle{\sum_{j\in\Z\setminus\N}}\p{2^{js}\norm{\uppi_jg}_{L^p}}^p}^{1/p}\\\le\begin{cases}\p{\sum_{j\in\Z\setminus\N}\p{2^{tj}\norm{\uppi_jg}_{L^p}}^q}^{1/q}&q\le p\\
    \snorm{\cb{2^{j\p{s-t}}}_{j\in\Z\setminus\N}}_{\ell^u\p{\Z\setminus\N}}\snorm{\cb{2^{jt}\norm{\uppi_jg}_{L^p}}_{j\in\Z\setminus\N}}_{\ell^q\p{\Z\setminus\N}}&p<q,\;1/p=1/q+1/u\end{cases}\lesssim\sb{g}_{{_\wedge}\dot{B}^{t,p}_q}.\end{multline}
    On the other hand, in the case that $2<p<\infty$, we can apply Minkowski's integral inequality to switch the sum and integral:
    \begin{multline}
        \bnorm{\bp{\textstyle{\sum_{j\in\Z\setminus\N}}\p{2^{sj}\abs{\uppi_jg}}^2}^{1/2}}_{L^p}\le\p{\textstyle{\sum_{j\in\Z\setminus\N}}\p{2^{\p{s-t}j}\norm{2^{tj}\uppi_jg}_{L^p}}^2}^{1/2}\\\le\begin{cases}\p{\sum_{j\in\Z\setminus\N}\p{2^{tj}\norm{\uppi_jg}_{L^p}}^q}^{1/q}&q\le 2\\
        \snorm{\cb{2^{j\p{s-t}}}_{j\in\Z\setminus\N}}_{\ell^u\p{\Z\setminus\N}}\snorm{\cb{2^{jt}\norm{\uppi_jg}_{L^p}}_{j\in\Z\setminus\N}}_{\ell^q\p{\Z\setminus\N}}&2<q,\;1/2=1/q+1/u\end{cases}\lesssim\sb{g}_{{_\wedge}\dot{B}^{t,p}_q}.
    \end{multline}
    Next, let us show the estimates of $h$.
    \begin{equation}
        \sb{h}_{\tilde{B}^{t,s}_{p,q}}=\bnorm{\bp{\textstyle{\sum_{j\in\Z\setminus\N}}\p{2^{sj}\abs{\uppi_j h}}^2}^{1/2}}_{L^p}+\snorm{\cb{2^{tj}\norm{\uppi_jh}_{L^p}}_{j\in\N}}_{\ell^q\p{\N}}\le\sb{h}_{\dot{H}^{s,p}}^{\thicksim}+\snorm{\cb{2^{tj}\norm{\uppi_jh}_{L^p}}_{j\in\N}}_{\ell^q\p{\N}}.
    \end{equation}
    Again, we break into cases based on the size of $p$ to control the high mode term. Let $w\in\R$ satisfy $t<w<s$. If $1<p<2$, then
    \begin{multline}
        \snorm{\cb{2^{tj}\norm{\uppi_jh}_{L^p}}_{j\in\N}}_{\ell^q\p{\N}}\le\textstyle{\sum_{j\in\N}}2^{tj}\norm{\uppi_j h}_{L^p}\le\bp{\textstyle{\sum_{j\in\N}}2^{\f{p}{p-1}\p{t-w}j}}^{(p-1)/p}\bp{\displaystyle{\int_{\R^n}}\textstyle{\sum_{j\in\N}}\abs{2^{wj}\uppi_jh}^p}^{1/p}\\=\bp{\textstyle{\sum_{j\in\N}}2^{\f{p}{p-1}\p{t-w}j}}^{(p-1)/p}\bp{\displaystyle{\int_{\R^n}}\textstyle{\sum_{j\in\N}}2^{p\p{w-s}j}\abs{2^{js}\uppi_jh}^p}^{1/p}\\\le\bp{\textstyle{\sum_{j\in\N}}2^{\f{p}{p-1}\p{t-w}j}}^{(p-1)/p}\bp{\textstyle{\sum_{j\in\N}}\p{2^{p\p{w-s}j}}^{\f{2}{2-p}}}^{(2-p)/2p}\bp{\displaystyle{\int_{\R^n}}\p{\textstyle{\sum_{j\in\N}}\p{2^{js}\abs{\uppi_jh}}^2}^{p/2}}^{1/p}\lesssim\sb{h}_{\dot{H}^{s,p}}^{\thicksim}.
    \end{multline}
    In the case that $2\le p<\infty$ we bound
    \begin{multline}
        \snorm{\cb{2^{tj}\norm{\uppi_jh}_{L^p}}_{j\in\N}}_{\ell^q\p{\N}}\le\textstyle{\sum_{j\in\N}}2^{tj}\norm{\uppi_j h}_{L^p}\le\bp{\textstyle{\sum_{j\in\N}}2^{\f{p}{p-1}\p{t-s}j}}^{(p-1)/p}\bp{\displaystyle{\int_{\R^n}}\textstyle{\sum_{j\in\N}}\abs{2^{sj}\uppi_jh}^p}^{1/p}\\\le\bp{\textstyle{\sum_{j\in\N}}2^{\f{p}{p-1}\p{t-s}j}}^{(p-1)/p}\bp{\displaystyle{\int_{\R^n}}\p{\textstyle{\sum_{j\in\N}}\p{2^{sj}\abs{\uppi_jh}}^2}^{p/2}}^{1/p}\lesssim\sb{h}_{\dot{H}^{s,p}}^{\thicksim}.
    \end{multline}
    Thus we have shown that there exists a constant $C\in\R^+$ depending on $\al$, $r$, $s$, $n$, $q$ and $p$ such that 
    \begin{equation}
        \sb{f}_{\tilde{B}^{t,s}_{p,q}}\le\sb{g}_{\tilde{B}^{t,s}_{p,q}}+\sb{h}_{\tilde{B}^{t,s}_{p,q}}\le C\big(\sb{g}_{{_\wedge}\dot{B}^{t,p}_q}+\sb{h}_{\dot{H}^{s,p}}\big).
    \end{equation}
    Upon taking the infimum over all decompositions of $f$, we arrive at the embedding
    \begin{equation}
        \p{\dot{H}^{r,p}\p{\R^n;\mathbb{K}},\dot{H}^{s,p}\p{\R^n;\mathbb{K}}}^{\p{\sig}}_{\al,q}\emb\tilde{B}^{t,s}_{p,q}\p{\R^n;\mathbb{K}}.
    \end{equation}
    
    On the other hand, let $f\in\tilde{B}^{t,s}_{p,q}(\R^n;\mathbb{K})$. Set $h=\sum_{j\in\N}\uppi_jf$ and $g=f-h$. We will prove that $h\in\dot{B}^{t,p}_{q}(\R^n;\mathbb{K})$ and $g\in\dot{H}^{s,p}(\R^n;\mathbb{K})$. Note first that the series defining $h$ is a well-defined tempered distribution, thanks to Lemma~\ref{convergence}. We now compute the action of the family $\cb{\uppi_j}_{j\in\Z}$ on $h$ using almost orthogonality, see Lemma~\ref{almost idempotent}. For $j\in\N^+$ it holds $\uppi_jh=\uppi_jf$, $\uppi_0h={\uppi_0}^2f+\uppi_0\uppi_1f$, $\uppi_{-1}h=\uppi_{-1}\uppi_0f$, and finally if $\Z\ni j<-1$ then $\uppi_jh=0$. This allows us to then estimate 
    \begin{equation}\label{this one is indeed an equation}
        \sb{h}_{{_\wedge}{B}^{t,p}_{q}}=\snorm{\cb{2^{tj}\norm{\uppi_jh}_{L^p}}_{j\in\Z}}_{\ell^q\p{\Z}}\le2^{-t}\norm{\uppi_{-1}\uppi_0f}_{L^p}+\norm{\p{\uppi_0+\uppi_1}\uppi_0f}_{L^p}+\snorm{\cb{2^{tj}\norm{\uppi_jf}_{L^p}}_{j\in\N^+}}_{\ell^q\p{\N^+}}.
    \end{equation}
    We apply Theorem~\ref{mihlin multiplier theorem} and Lemma~\ref{scaling invariance of fourier multipliers} to the first two terms on the right hand side  to find a constant $\bar{c}$, depending only on $n$, $p$, and $\psi$, such that for $\ell\in\cb{-1,0,1}$,  $\norm{\uppi_\ell\uppi_0f}_{L^p}\le\bar{c}\norm{\uppi_0f}_{L^p}$. Plugging this into~\eqref{this one is indeed an equation} yields the bound
    \begin{equation}\label{number 1}
        \sb{h}_{{_\wedge}\dot{B}^{t,p}_{q}}\lesssim\norm{\uppi_0f}_{L^p}+\snorm{\cb{2^{tj}\norm{\uppi_jf}_{L^p}}_{j\in\N^+}}_{\ell^q\p{\N^+}}\le2\sb{f}_{\tilde{B}^{t,s}_{p,q}}.
    \end{equation}
    
    We now handle the estimates of $g$. Again we use Lemma~\ref{convergence} to see that for $j\in\N^+$ we have $\uppi_jg=0$, $\uppi_0g=\uppi_{-1}\uppi_0f$, $\uppi_{-1}g={\uppi_{-1}}^2f+\uppi_{-2}\uppi_{-1}f$, while for $j\in\Z\setminus\p{\N\cup\cb{-1}}$ we have $\uppi_jg=\uppi_jf$. Thus with $\bar{c}$ as before, we find that
    \begin{multline}\label{number two}
        \sb{g}_{\dot{H}^{s,p}}=\bnorm{\bp{\textstyle{\sum_{j\in\Z}}\p{2^{js}\abs{\uppi_jg}}^2}^{1/2}}_{L^p}\le\norm{\uppi_{0}\uppi_{-1}f}_{L^p}+2^{-s}\norm{\p{\uppi_{-2}+\uppi_{-1}}\uppi_{-1}f}_{L^p}\\+\bnorm{\bp{\textstyle{\sum_{j=-\infty}^{-1}}\p{2^{sj}\abs{\uppi_jf}}^2}^{1/2}}_{L^p}\le\p{1+\p{1+2^{-s+1}}\Bar{c}}\sb{f}_{\tilde{B}^{t,s}_{p,q}}
    \end{multline}
    Together, estimates~\eqref{number 1} and~\eqref{number two} prove the other embedding: $\tilde{B}^{t,s}_{p,q}\p{\R^n;\mathbb{K}}\emb\big(\dot{H}^{r,p}\p{\R^n;\mathbb{K}},\dot{H}^{s,p}\big)^{\p{\sig}}_{\al,q}$.
\end{proof}

The following result should be contrasted with Theorem~\ref{fundamental decomposition of screened sobolev spaces}

\begin{coro}[Fundamental decomposition of generalized screened Besov and Riesz potential spaces]\label{2nd sum characterization} 
Let $r,s\in\R$, $1<p<\infty$, and $1\le q\le\infty$. 
Consider the high and low pass filters $\mathds{P}^+,\mathds{P}^-:\mathscr{S}^\ast\p{\R^n;\mathbb{K}}\to\mathscr{S}^\ast\p{\R^n;\mathbb{K}}$ defined by $\mathds{P}^+f=\textstyle{\sum_{j\in\N}}\uppi_jf$ and $\mathds{P}^-f =f - \mathds{P}^+ f = (I-\mathds{P}^+)f =\big(I-\textstyle{\sum_{j\in\N}}\uppi_j\big)f$.
These are well defined thanks to Lemmas~\ref{convergence} and~\ref{mult lemma}.  The following  hold:
 \begin{enumerate}
     \item If $r<s$, then for all $f\in\tilde{B}^{r,s}_{p,q}\p{\R^n;\mathbb{K}}$ we have $\mathds{P}^+f\in{_\wedge}\dot{B}^{r,p}_q\p{\R^n;\mathbb{K}}$, $\mathds{P}^-f\in\dot{H}^{s,p}\p{\R^n;\mathbb{K}}$, and
     \begin{equation}\label{danger math ahead}
         \sb{f}_{\tilde{B}^{r,s}_{p,q}}\asymp\sb{\mathds{P}^+f}_{{_\wedge}\dot{B}^{r,p}_q}+\sb{\mathds{P}^-f}_{\dot{H}^{s,p}}.
     \end{equation}
     \item If $s<r$, then for all $f\in\tilde{H}^{r,s}_{p,q}\p{\R^n;\mathbb{K}}$ we have $\mathds{P}^+f\in\dot{H}^{s,p}\p{\R^n;\mathbb{K}}$ and $\mathds{P}^-f\in{_\wedge}\dot{B}^{r,p}_q\p{\R^n;\mathbb{K}}$, and
     \begin{equation}
         \sb{f}_{\tilde{H}^{r,s}_{p,q}}\asymp\sb{\mathds{P}^+f}_{\dot{H}^{s,p}}+\sb{\mathds{P}^-f}_{{_\wedge}\dot{B}^{r,p}_q}
     \end{equation}
 \end{enumerate}
\end{coro}
\begin{proof}
Again we only prove the first item, as the second item follows from similar arguments. A consequence of Theorem~\ref{truncated interpolation of Riesz potential spaces} is the sum characterization: $\Sigma\big({_\wedge}\dot{B}^{r,p}_q\p{\R^n;\mathbb{K}},\dot{H}^{s,p}\p{\R^n;\mathbb{K}}\big)=\tilde{B}^{r,s}_{p,q}\p{\R^n;\mathbb{K}}$. Therefore the `$\lesssim$' inequality in~\eqref{danger math ahead} is handled. As for the `$\gtrsim$' inequality, we see that this is covered in the latter half of the proof of Theorem~\ref{truncated interpolation of Riesz potential spaces}.  There we showed that for $f\in\tilde{B}^{r,s}_{p,q}\p{\R^n;\mathds{K}}$ we can decompose  $f=\mathds{P}^+f+\mathds{P}^-f$, and the seminorms of the factors in ${_\wedge}\dot{B}^{r,p}_q\p{\R^n;\mathbb{K}}$ and $\dot{H}^{s,p}\p{\R^n;\mathbb{K}}$, respectively, can be bounded above by a universal constant times $\sb{f}_{\tilde{B}^{r,s}_{p,q}}$.

\end{proof}

Next we obtain another characterization of the screened Besov spaces.

\begin{coro}[Frequency space characterization of screened Besov spaces]\label{general fourier char of screened spaces}
Let $s\in\p{0,1}$, $1<p<\infty$, and $1\le q\le\infty$. The following hold:
\begin{enumerate}
    \item $\tilde{B}^{s,p}_{q}\p{\R^n;\mathbb{K}}\emb\tilde{B}^{s,1}_{p,q}\p{\R^n;\mathbb{K}}$, where the latter space is from Definition~\ref{screened besov spaces defn}.
    \item If $f\in\tilde{B}^{s,1}_{p,q}\p{\R^n;\mathbb{K}}$, then $f$ is identified with a locally integrable function and there exists a polynomial $Q$ whose coefficients, aside from the constant term, are uniquely determined, with the property that $f-Q\in\tilde{B}^{s,p}_{q}\p{\R^n;\mathbb{K}}$.  Moreover, there exists a constant $c\in\R^+$ depending only on $s$, $p$, $q$, and $n$ such that: $\sb{f-Q}_{\tilde{B}^{s,p}_{q}}\le c\sb{f}_{\tilde{B}^{s,1}_{p,q}}$.
\end{enumerate}
\end{coro}
\begin{proof}
Corollary~\ref{sum characterization of screened sobolev spaces} and Theorem~\ref{truncated interpolation of Riesz potential spaces} gave us the identities:
\begin{equation}
    \big(L^p\p{\R^n;\mathbb{K}};\dot{W}^{1,p}\p{\R^n;\mathbb{K}}\big)^{\p{1}}_{s,q}=\tilde{B}^{s,p}_q\p{\R^n;\mathbb{K}}\text{ and }\big(\dot{H}^{0,p}\p{\R^n;\mathbb{K}};\dot{H}^{1,p}\p{\R^n;\mathbb{K}}\big)^{\p{1}}_{s,q}=\tilde{B}^{s,1}_{p,q}\p{\R^n;\mathbb{K}}.
\end{equation}
Theorem \ref{littlewood-paley characterization of Lp} shows that $L^p(\R^n;\mathbb{K}) \emb\dot{H}^{0,p}(\R^n;\mathbb{K})$, and Theorem~\ref{frequency space characterization of homogeneous sobolev spaces} shows $\dot{W}^{1,p}(\R^n;\mathbb{K})\emb\dot{H}^{1,p}(\R^n;\mathbb{K})$. These combine to prove the first item.

On the other hand, if $f\in\tilde{B}^{s,1}_{p,q}\p{\R^n;\mathbb{K}}$, then Corollary~\ref{2nd sum characterization} tells us that $\mathds{P}^+f\in{_\wedge}\dot{B}^{s,p}_q\p{\R^n;\mathbb{K}}$ and $\mathds{P}^-f\in\dot{H}^{1,p}\p{\R^n;\mathbb{K}}$. The former has Fourier transform supported away from the origin and hence (by Corollary~\ref{besov fourier char}) $\mathds{P}^+f\in\dot{B}^{s,p}_q\p{\R^n;\mathbb{K}}$. The second conclusion of Theorem~\ref{frequency space characterization of homogeneous sobolev spaces} gives us a polynomial $Q$ such that $\mathds{P}^-f-Q\in\dot{W}^{1,p}\p{\R^n;\mathbb{K}}$, and hence $f-Q\in\tilde{B}^{s,p}_q\p{\R^n;\mathbb{K}}$. Moreover, from Corollary~\ref{besov fourier char}, Theorem~\ref{frequency space characterization of homogeneous sobolev spaces}, and Corollary~\ref{2nd sum characterization} we obtain the universal bound
\begin{equation}
    \sb{f-Q}_{\tilde{B}^{s,p}_q}\le\sb{\mathds{P}^-f-Q}_{\dot{W}^{1,p}}+\sb{\mathds{P}^+f}_{\dot{B}^{s,p}_q}\lesssim\sb{\mathds{P}^-f}_{\dot{H}^{1,p}}+\sb{\mathds{P}^+f}_{{_\wedge}\dot{B}^{s,p}_q}\lesssim\sb{f}_{\tilde{B}^{s,1}_{p,q}}.
\end{equation}
Finally, if $P$, $Q$ are both polynomials such that $f-Q,f-P\in\tilde{B}^{s,p}_q\p{\R^n;\mathbb{K}}$, then $P-Q\in\tilde{B}^{s,p}_q\p{\R^n;\mathbb{K}}$, which then implies that $P-Q$ is a constant.
\end{proof}
\subsection{Embeddings}\label{sec:embeddings}

In this subsection we shed some light on the nature of the Sobolev embeddings for the screened Besov spaces. Recall the notation for the homogeneous H\"older spaces, $\dot{C}^{0,\alpha}$, defined in Section~\ref{notation stuff}. We start in the subcritical case.

\begin{prop}[Subcritical embedding]\label{subcritical embedding}
Suppose that $n\in\N\setminus\cb{0,1}$, $1\le p<n$, $1\le u\le\infty$, and $s\in\p{0,1}$. Set $q=\f{np}{n-p}$ and $r=\f{np}{n-sp}$. Then there is a constant $c\in\R^+$ such that for all $f\in\tilde{B}^{s,p}_{u}\p{\R^n;\mathbb{K}}$ there exists $a\in\mathbb{K}$ such that $f-a\in\Sigma\p{L^q\p{\R^n;\mathbb{K}},L^{r,u}\p{\R^n;\mathbb{K}}}$,  with the estimate $\norm{f-a}_{\Sigma\p{L^q,L^{r,u}}}\le c\sb{f}_{\tilde{B}^{s,p}_{u}}$.
\end{prop}
\begin{proof}
First, we claim that if $f\in\Sigma\big(L^p\p{\R^n;\mathbb{K}};\dot{W}^{1,p}\p{\R^n;\mathbb{K}}\big)$ there is a unique $a\p{f}\in\mathbb{K}$ such that $f-a\p{f}\in\Sigma\p{L^p\p{\R^n; \mathbb{K}},L^q\p{\R^n;\mathbb{K}}}$. Uniqueness is clear since if $f-a,f-b\in\Sigma\p{L^p\p{\R^n;\mathbb{K}},L^q\p{\R^n;\mathbb{K}}}$, for $a,b\in\mathbb{K}$, then $a-b\in\Sigma\p{L^p\p{\R^n;\mathbb{K}},L^q\p{\R^n;\mathbb{K}}}$, which can only happen if $a=b$. Existence is a consequence of the Gagliardo-Nirenberg-Sobolev inequality (see, for instance, Theorem 12.9 in~\cite{MR3726909}): there is a constant $c\in\R^+$ such that for all $w\in\dot{W}^{1,p}\p{\R^n;\mathbb{K}}$ there exists $a\p{w}\in\mathbb{K}$ such that $w-a\p{w}\in L^q\p{\R^n;\mathbb{K}}$ and $\norm{w-a\p{w}}_{L^q}\le c\sb{w}_{\dot{W}^{1,p}}$. Thus, if $u\in\Sigma\big(L^p\p{\R^n;\mathbb{K}},\dot{W}^{1,p}\p{\R^n;\mathbb{K}}\big)$, then we can take $a\p{u}=a\p{w}$ for any decomposition $u=v+w$, with $v\in L^p\p{\R^n;\mathbb{K}}$ and $w\in\dot{W}^{1,p}\p{\R^n;\mathbb{K}}$. 

Next we define $\iota:\Sigma\big(L^p\p{\R^n;\mathbb{K}},\dot{W}^{1,p}\p{\R^n;\mathbb{K}}\big)\to\Sigma\big(L^p\p{\R^n;\mathbb{K}},L^q\p{\R^n;\mathbb{K}}\big)$ via $ \iota f=f-a\p{f}$.  The dependence of $a\p{f}$ on $f$ is linear, so $\iota$ is linear.  It is also the case that $\iota$ is continuous. Indeed, for any $f\in\Sigma\big(L^p\p{\R^n;\mathbb{K}};\dot{W}^{1,p}\p{\R^n;\mathbb{K}}\big)$ and any decomposition $f=v+w$ for $v\in L^p\p{\R^n;\mathbb{K}}$ and $\dot{W}^{1,p}\p{\R^n;\mathbb{K}}$, we may estimate
\begin{equation}
    \norm{\iota f}_{\Sigma\p{L^p,L^q}}\le\norm{v}_{L^p}+\norm{w-a\p{w}}_{L^q}
     \le\p{1+c}\sb{f}_{\Sigma\p{L^p,\dot{W}^{1,p}}}.
\end{equation}
Similar arguments show that $\iota$ continuously maps $L^p\p{\R^n;\mathbb{K}}$ to itself (and, in fact, equals the identity mapping) and continuously maps $\dot{W}^{1,p}\p{\R^n;\mathbb{K}}$ to $L^q\p{\R^n;\mathbb{K}}$. 

Now we use the fact that the screened spaces are interpolation spaces (see Theorem~\ref{truncated and bounded linear mappings}). This implies, by the abstract sum characterization in Theorem~\ref{label sum characterization of the truncated method}, that
\begin{equation}
    \iota:\tilde{B}^{s,p}_{u}\p{\R^n;\mathbb{K}}\to\p{L^p\p{\R^n;\mathbb{K}},L^q\p{\R^n;\mathbb{K}}}_{s,u}^{\p{1}}=\Sigma\p{L^q\p{\R^n;\mathbb{K}},L^{r,u}\p{\R^n;\mathbb{K}}}
\end{equation}
is a continuous linear mapping. Here we have used the fact that $\tilde{B}^{s,p}_{u}\p{\R^n\mathbb{K}}=\big(L^p\p{\R^n;\mathbb{K}},\dot{W}^{1,p}\p{\R^n;\mathbb{K}}\big)^{\p{1}}_{s,u}$ by Corollary~\ref{sum characterization of screened sobolev spaces} and Theorem~\ref{screened sobolev spaces are truncated interpolation spaces}, and that $\p{L^p\p{\R^n;\mathbb{K}},L^q\p{\R^n;\mathbb{K}}}_{s,u}^{\p{1}}=\Sigma\p{L^q\p{\R^n;\mathbb{K}},L^{r,u}\p{\R^n;\mathbb{K}}}$ by Theorem~\ref{label sum characterization of the truncated method} and Example~\ref{lorentz}.
\end{proof}

Next we consider a first mixed case.

\begin{prop}[Mixed subcritical/critical embedding]\label{mixed subcritical/ciritcal embedding}
Let $1\le q\le\infty$ and $s\in\p{0,1}$. Then there exists $c\in\R^+$ such that for all $f\in\tilde{B}^{s,n}_{q}\p{\R^n;\mathbb{K}}$ we have the bound $\sb{f}_{X}\le c\sb{f}_{\tilde{B}^{s,n}_q}$, where $X=\Sigma\big(L^{r,q}\p{\R^n;\mathbb{K}},\m{BMO}\p{\R^n;\mathbb{K}}\big)$ and $r=n/(1-s)$.
\end{prop}
\begin{proof}
By Theorem 12.31 in~\cite{MR3726909} we have the continuous embedding $\dot{W}^{1,n}\p{\R^n;\mathbb{K}}\emb\m{BMO}\p{\R^n;\mathbb{K}}$. Example~\ref{BMO} shows that $    \p{L^p\p{\R^n;\mathbb{K}},\m{BMO}\p{\R^n;\mathbb{K}}}_{s,q}^{\p{1}}= \Sigma\big(L^{r,q}\p{\R^n;\mathbb{K}},\m{BMO}\p{\R^n;\mathbb{K}}\big)$.
The result now follows from Theorem~\ref{truncated and bounded linear mappings} applied to the inclusion mapping from $\Sigma\big(L^p\p{\R^n;\mathbb{K}};\dot{W}^{1,n}\p{\R^n;\mathbb{K}}\big)$ to $\Sigma\big(L^p\p{\R^n;\mathbb{K}},\m{BMO}\p{\R^n;\mathbb{K}}\big)$.
\end{proof}

The next mixed case follows.

\begin{prop}[Mixed super/subcritical embedding]\label{super/sub embedding}
Let $s\in\p{0,1}$ and $n<p$, but $sp<n$, and $1\le q\le\infty$.  Set $r=\f{np}{n-sp}$ and $\al=1-\f{n}{p}$. Then we have the continuous embedding
\begin{equation}
    \tilde{B}^{s,p}_{q}\p{\R^n;\mathbb{K}}\emb\Sigma\big(L^{r,q}\p{\R^n;\mathbb{K}},\dot{C}^{0,\al}\p{\R^n;\mathbb{K}}\big).
\end{equation}
\end{prop}
\begin{proof}
Theorem~\ref{fundamental decomposition of screened sobolev spaces} tells us that we have the equality of spaces with equivalence of seminorms:
\begin{equation}
    \tilde{B}^{s,p}_{q}\p{\R^n;\mathbb{K}}=\Sigma\big(\dot{W}^{1,p}\p{\R^n;\mathbb{K}},B^{s,p}_q\p{\R^n;\mathbb{K}}\big).
\end{equation}The Morrey embedding yields $\dot{W}^{1,p}\p{\R^n;\mathbb{K}}\emb\dot{C}^{0,\al}\p{\R^n;\mathbb{K}}$ and the subcritical embedding of the Besov space yields $B^{s,p}_q\p{\R^n;\mathbb{K}}\emb L^{r,q}\p{\R^n;\mathbb{K}}$ (see Lemma 12.47 and Theorem 17.49 in~\cite{MR3726909}).
\end{proof}

We now handle the final mixed case.

\begin{prop}[Mixed critical, supercritical embedding]\label{critical/supercritical}
Let $s\in\p{0,1}$, $n<p$, with $sp=n$, and $1\le q\le\infty$. Set $p\le r<\infty$ and $\al=1-\f{n}{p}$. Then we have the continuous embedding
\begin{equation}
    \tilde{B}^{s,p}_{q}\p{\R^n;\mathbb{K}}\emb\Sigma\big(L^r\p{\R^n;\mathbb{K}},\dot{C}^{0,\al}\p{\R^n;\mathbb{K}}\big).
\end{equation}
\end{prop}
\begin{proof}
Again Morrey's embedding implies that $\dot{W}^{1,p}\p{\R^n;\mathbb{K}}\emb\dot{C}^{0,\al}\p{\R^n;\mathbb{K}}$. The critical embedding of Besov spaces (see Theorem 17.55 in~\cite{MR3726909}) implies that $B^{s,p}_q\p{\R^n;\mathbb{K}}\emb L^r\p{\R^n;\mathbb{K}}$.
\end{proof}

Finally, we consider the supercritical case.

\begin{prop}[Supercritical embedding]\label{supercritical embedding}
Let $s\in\p{0,1}$, $n<p$ with $n<sp$, and $1\le q\le\infty$. Set $\al=1-\f{n}{p}$ and $\be=s-\f{n}{p}$. Then  we have the continuous embedding
\begin{equation}
    \tilde{B}^{s,p}_{q}\p{\R^n;\mathbb{K}}\emb\Sigma\big(\dot{C}^{0,\al}\p{\R^n;\mathbb{K}},C^{0,\be}\p{\R^n;\mathbb{K}}\big).
\end{equation}
\end{prop}
\begin{proof}
The hypothesis on  $s$ and $p$ ensure that $\dot{W}^{1,p}\p{\R^n;\mathbb{K}}\emb\dot{C}^{0,\al}\p{\R^n;\mathbb{K}}$ and $B^{s,p}_q\p{\R^n;\mathbb{K}}\emb C^{0,\be}\p{\R^n;\mathbb{K}}$, thanks to Morrey's embedding and Theorem 17.52 in~\cite{MR3726909}.
\end{proof}

\subsection{Behavior on spaces of codimension 1}

We now prove a a result characterizing how restriction to codimension $1$ subspaces behaves in screened Besov spaces.

\begin{thm}[Restriction]\label{restriction theorem}
Let $n\ge 2$, $1<p<\infty$, $1\le q\le\infty$, and $s\in\p{0,1}$.  Suppose that $p^{-1}<s$ and set
\begin{equation}\label{label label}
    X=\Sigma\big(B^{s-1/p,p}_q\p{\R^{n-1};\mathbb{K}},\dot{B}^{{1-1/p},p}_p\p{\R^{n-1};\mathbb{K}}\big).
\end{equation}
Define the restriction map $R:\mathscr{S}\p{\R^n;\mathbb{K}} \to \mathscr{S}\p{\R^{n-1};\mathbb{K}}$ via $Rf(y) = f(y,0)$ for $y \in \R^{n-1}$.  Then there exists a continuous linear map $\mathcal{R} : \tilde{B}^{s,p}_{q}\p{\R^n;\mathbb{K}} \to X$  with the property that $\mathcal{R}=R$ on $\mathscr{S}\p{\R^n;\mathbb{K}}$.  Moreover, there exists a continuous linear lifting map $\mathcal{E}:X\to\tilde{B}^{s,p}_{q}\p{\R^n;\mathbb{K}}$ such that $\mathcal{R}\mathcal{E}=I$ on $X$.
\end{thm}
\begin{proof}
The trace theory in Section 2.7.2 of~\cite{MR3024598} and Chapter 18 of~\cite{MR3726909} provides continuous linear maps $\mathcal{R}^+ : B^{s,p}_q\p{\R^n;\mathbb{K}} \to B^{s-1/p,p}_p\p{\R^{n-1};\mathbb{K}}$ and $\mathcal{R}^-: \dot{W}^{1,p}\p{\R^n;\mathbb{K}}\to \dot{B}^{1-1/p,p}_p\p{\R^{n-1};\mathbb{K}}$ such that $\mathcal{R}^+u=Ru = \mathcal{R}^{-} u$ for all $u \in \mathscr{S}\p{\R^n;\mathbb{K}}$.  Moreover, the restriction of $\mathcal{R}^-$ to $W^{1,p}\p{\R^n;\mathbb{K}}$ maps continuously into $B^{1-1/p,p}_p\p{\R^{n-1};\mathbb{K}}$.

We claim that $\mathcal{R}^+=\mathcal{R}^-$ on $\Delta\big(B^{s,p}_q\p{\R^n;\mathbb{K}}, \dot{W}^{1,p}\p{\R^n;\mathbb{K}}\big) = W^{1,p}\p{\R^n;\mathbb{K}}$. Let $u\in W^{1,p}\p{\R^n}$ and take $\cb{u_m}_{m\in\N}\subset\mathscr{S}\p{\R^n;\mathbb{K}}$ such that $u_m\to u$ in  $W^{1,p}\p{\R^n;\mathbb{K}}$ as $m\to\infty$. Then $\mathcal{R}^+u_m=Ru_m=\mathcal{R}^-u_m$ converges to both $\mathcal{R}^+u$ in $B^{s-1/p,p}_q\p{\R^{n-1};\mathbb{K}}$ and $\mathcal{R}^-u$ in $B^{1-1/p,p}_p\p{\R^n;\mathbb{K}}$ as $m\to\infty$. These are strongly compatible Hausdorff vector spaces, so $\mathcal{R}^+u=\mathcal{R}^-u$, and the claim is proved.

Corollary~\ref{sum characterization of screened sobolev spaces} showed that $\tilde{B}^{s,p}_{q}\p{\R^n;\mathbb{K}}=\Sigma\big(\dot{W}^{1,p}\p{\R^n;\mathbb{K}},B^{s,p}_q\p{\R^n;\mathbb{K}}\big)$, so if $u\in\tilde{B}^{s,p}_{q}\p{\R^n;\mathbb{K}}$ we may decompose it as $u=v+w$ for $v \in \dot{W}^{1,p}\p{\R^n;\mathbb{K}}$ and $w\in B^{s,p}_q\p{\R^n;\mathbb{K}}$ as above.  Given two such decompositions, $u=v+w=\tilde{v}+\tilde{w}$, we have that $v-\tilde{v} = \tilde{w}-w \in\Delta\big(\dot{W}^{1,p}\p{\R^n;\mathbb{K}},B^{s,p}_q\p{\R^n;\mathbb{K}}\big)$, and so the above claim shows that $\mathcal{R}^-\p{v-\tilde{v}}=\mathcal{R}^+\p{\tilde{w}-w}$, and hence $\mathcal{R}^-\tilde{v}+\mathcal{R}^+\tilde{w}=\mathcal{R}^-v+\mathcal{R}^+w$.   This allows us to define the linear map $\mathcal{R}:\tilde{B}^{s,p}_{q}\p{\R^n;\mathbb{K}}\to X$ via $\mathcal{R}u = \mathcal{R}^- v + \mathcal{R}^+ w$ where $u= v+w$ is a decomposition as above.  This map takes values in $X$ and is continuous due to the mapping properties of $\mathcal{R}^\pm$.  Clearly, $\mathcal{R} = R$ on the Schwartz class.

We now construct $\mathcal{E}$. The results in Chapter 18 of~\cite{MR3726909} and Section 2.7.2 in~\cite{MR3024598} provide continuous linear maps $\mathcal{E}^-:\dot{B}^{1-1/p,p}_p\p{\R^{n-1};\mathbb{K}}\to\dot{W}^{1,p}\p{\R^n;\mathbb{K}}$ and $\mathcal{E}^+:B^{s-1/p,p}_q\p{\R^{n-1};\mathbb{K}}\to B^{s,p}_q\p{\R^n;\mathbb{K}}$
with the property that for all $v \in \dot{B}^{1-1/p,p}_p\p{\R^{n-1};\mathbb{K}}$ and $w \in B^{s-1/p,p}_q\p{\R^{n-1};\mathbb{K}}$ we have $\mathcal{R}^-\mathcal{E}^-v=v$ and $\mathcal{R}^+\mathcal{E}^+w=w$.  We paste $\mathcal{L}^\pm$ together with the use of the high and low pass filters $\mathds{P}^{\pm}$ from Corollary~\ref{2nd sum characterization}.  Indeed, we define  $\mathcal{E}:X\to\tilde{B}^{s,p}_{q}\p{\R^n;\mathbb{K}}$ via $\mathcal{E}f=\mathcal{L}^{-}\mathds{P}^-f+\mathcal{L}^{+}\mathds{P}^+f$.

Arguing as in the Corollary~\ref{2nd sum characterization}, we have that $\mathds{P}^+:X\to B^{s-1/p,p}_q\p{\R^n;\mathbb{K}}$ and $\mathds{P}^-:X\to\dot{B}^{1-1/p,p}_p\p{\R^n;\mathbb{K}}$ are continuous linear maps.   Hence, $\mathcal{E}$ is as well.  Moreover, for any $w\in X$ we can compute
\begin{equation}
    \mathcal{R}\mathcal{L}w=\mathcal{R}^-\mathcal{L}^-\mathds{P}^-w+\mathcal{R}^+\mathcal{L}^+\mathds{P}^+w=\p{\mathds{P}^-+\mathds{P}^+}w=w,
\end{equation}
and so $\mathcal{E}$ is the desired right inverse for $\mathcal{R}$.
\end{proof}
\appendix
\section{Vector topologies}\label{seminorm topology}

In this appendix we recall notions of topology in vector spaces with a particular interest in seminormed spaces. We will state several elementary facts and not attempt to provide proofs. The interested reader is referred to Taylor's book~\cite{MR564653} or to Section 2.4 of \cite{leoni2019traces}.

\begin{defn}[Topological vector spaces and annihilators]\label{defn of topological vector space}
Suppose that $X$ is a vector space over $\mathbb{K}\in\cb{\R,\C}$ equipped with a topology $\tau$.\begin{enumerate}
    \item We say that $\p{X,\tau}$ is a topological vector space if the vector operations, $+:X\times X\to X$ and $\cdot:\mathbb{K}\times X\to X$, are continuous.
    \item We define the annihilator of $\tau$ to be the set $\mathfrak{A}\p{X}=\Bar{\cb{0}}^{\tau}$, i.e. the $\tau$-closure of $0$.
\end{enumerate}
\end{defn}

Note that some authors enforce that topological vector spaces are a priori Hausdorff. We do not build this into our definition since our interests include vector spaces topologized by seminorms. The annihilator of a topological vector space measures how far away the space is from being Hausdorff. The following proposition quantifies this precisely.

\begin{prop}\label{kernel measures failure of a space to be hausdorff} 
Suppose that $\p{X,\tau}$ is a topological vector space in the sense of Definition~\ref{defn of topological vector space}. Then $\mathfrak{A}\p{X}$ is a closed vector subspace of $X$ that  measures the failure of $\p{X,\tau}$ to be a Hausdorff space in the following sense.  If $\p{Y,\tilde{\tau}}$ is a topological space, $f:Y\to X$, $y\in Y$,  $x_0,x_1\in X$, and  $f\p{z}\to x_0$ and  $f\p{z}\to x_1$  as $z\to y$, then $x_0-x_1\in\mathfrak{A}\p{X}$.
\end{prop}

A topological vector space may be equipped with various notions of size. On way of doing this is through a seminorm.

\begin{defn}[Seminormed spaces]\label{defn of seminorm}
Suppose that $\mathbb{K}\in\cb{\R,\C}$ and $X$ is a vector space over $\mathbb{K}$. We say that a function $\sb{\cdot}:X\to\R$ is a seminorm if the following properties hold: \emph{nonnegativity}:  $\sb{x}\ge0$ for all $x\in X$; \emph{subadditivity}:  $\forall\;x,y\in X$, $\sb{x+y}\le\sb{x}+\sb{y}$; \emph{homogeneity}:  $\forall\;\al\in\mathbb{K}$ and $\forall\;x\in X$,  $\sb{\al x}=\abs{\al}\sb{x}$. We say that the pair $\p{X,\sb{\cdot}}$ is a seminormed space. This space is endowed with the topology $\tau=\big\{\sb{\cdot}^{-1}\p{U}\;:\;U\subseteq\R\;\text{is open}\big\}$. Note that this is the smallest topology in which $\sb{\cdot}$ is a continuous mapping.
\end{defn}
Seminormed spaces are topological vector spaces and we have the following realization of their annihilators.
\begin{prop}\label{seminormed spaces are TVS} 
If $\p{X,\sb{\cdot}}$ is a seminormed space, then the topology $\tau$ from Definition~\ref{defn of seminorm} makes $\p{X,\tau}$ a topological vector space in the sense of Definition~\ref{defn of topological vector space}. Moreover, $\mathfrak{A}\p{X}=\cb{x\in X\;:\;\sb{x}=0}$ is a closed vector subspace.
\end{prop}
Note that seminormed spaces are, in particular, semimetric spaces. Hence, we can quotient out be the annihilator of the seminorm and the resulting structure is, at the least, a metric space; but actually, this quotient space results in a normed vector space.

\begin{prop}[Quotient by annihilator]\label{quotient space of seminormed space}
Let $\p{X,\sb{\cdot}}$ be a seminormed space. We make the following definitions:
\begin{enumerate}
    \item For $x,y \in X$ we say that $x\sim y$  if $x-y\in\mathfrak{A}\p{X}$. This obviously defines an equivalence relation on $X$. Let $X/\mathfrak{A}\p{X}$ denote the resulting set of equivalence classes.
    \item We define the function $\abs{\sb{\cdot}}:X/\mathfrak{A}\p{X}\to\R$ via $\abs{\sb{Y}}=\sb{y}, \text{ where }y\in Y\in X/\mathfrak{A}\p{X}$.
\end{enumerate}
Then, it holds that $\abs{\sb{\cdot}}$ is well-defined and equipping $X/\mathfrak{A}\p{X}$ with $\abs{\sb{\cdot}}$ results in a normed vector space.
\end{prop}
This leads us to a natural notion of completeness in seminormed spaces.
\begin{defn}[Semi-Banach spaces]\label{semi-banach spaces} We say that a seminormed space $\p{X,\sb{\cdot}}$ is semi-Banach or complete if the normed quotient space $\p{X/\mathfrak{A}\p{X},\abs{\sb{\cdot}}}$ is complete or Banach as a normed vector space.
\end{defn}
The following characterization of completeness in seminormed spaces spaces is often useful.
\begin{lem}\label{characterizations of completeness in seminormed spaces}
Suppose that $\p{X,\sb{\cdot}}$ is a seminormed space. Then, $\p{X,\sb{\cdot}}$ is semi-Banach if and only if $\forall\;\cb{x_k}_{k=0}^\infty\subset X$ Cauchy, there exists $x\in X$ such that $\lim_{k\to\infty}\sb{x-x_k}=0$ if and only if $\forall\;\cb{x_k}_{k=0}^\infty\subset X$ such that $\sum_{k=0}^\infty\sb{x_k}<\infty$ there exists $x\in X$ such that $\big[x-\sum_{k=0}^Mx_k\big]\to0$ as $M\to\infty$.
\end{lem}
We now turn our attention to linear mappings between seminormed spaces.
\begin{prop}[Properties of linear mappings]\label{properties of linear mappings}
Let $\p{X_0,\sb{\cdot}_0}$ and  $\p{X_1,\sb{\cdot}_1}$ be seminormed spaces, and $T:X_0\to X_1$ be a linear mapping. Then $T$ is continuous if and only if there is a constant $c\in\R^+$ such that for all $x\in X_0$ we have the bound $\sb{Tx}_1\le c\sb{x}_0$.  If either condition holds, then $T\mathfrak{A}\p{X_0}\subseteq\mathfrak{A}\p{X_1}$.
\end{prop}
A particularly common linear mapping between seminormed spaces is an embedding.
\begin{defn}[Continuous embeddings and equivalent seminormed spaces]\label{continuous embeddings and equivalence of seminormed spaces} 
Let $\p{X_0,\sb{\cdot}_0}$ and  $\p{X_1,\sb{\cdot}_1}$ be seminormed spaces related via the inclusion $X_0\subseteq X_1$. We say that $X_0$ is continuously embedded into $X_1$ if the inclusion mapping $\iota:X_0\to X_1$ is continuous; we write $X_0\emb X_1$ in this case. If, in addition, we assume that $X_1\subseteq X_0$ and the opposite inclusion $X_1\to X_0$ is continuous, we say that $X_0$ and $X_1$ are equivalent as seminormed spaces.
\end{defn}

Note that Proposition~\ref{properties of linear mappings} implies that if $X_0\emb X_1$, then the annihilator of $X_1$ must be at least as large as the annihilator of $X_0$; in particular, a non-Hausdorff space cannot be continuously embedded into a Hausdorff space.

Lastly, we note that seminormed spaces are equivalent if and only if their seminormed are uniformly comparable.
\begin{prop}[Equivalent seminormed spaces have the same topology]\label{characterization of equivalent seminormed spaces}
Let $\p{X_0,\sb{\cdot}_0}$ and  $\p{X_1,\sb{\cdot}_1}$ be seminormed spaces with $X_0\subseteq X_1\subseteq X_0$. Then $\p{X_0,\sb{\cdot}_0}$ and $\p{X_1,\sb{\cdot}_1}$ are equivalent as seminormed spaces if and only if
    the seminorms $\sb{\cdot}_0$ and $\sb{\cdot}_1$ are equivalent in the sense that there is $c\in\R^+$ such that 
    \begin{equation}
        c^{-1}\sb{x}_0\le\sb{x}_1\le c\sb{x}_0 \text{ for all } x\in X_0=X_1.
    \end{equation}
\end{prop}
\section{Miscellaneous facts from analysis}\label{Miscellaneous facts from analysis}
This appendix serves to collect some analysis results used in this paper.
\subsection{Integral inequalities and identities}
\begin{lem}[Averaging on cubes]\label{regularity promotion via averaging} Let $\es\neq Q\subset\R^n$ be an open cube. If $f:\R^n\to\mathbb{K}$ is a locally integrable function, we define $f_Q:\R^n\to\mathbb{K}$ via $f_Q\p{x}=\int_{Q}f\p{x+y}\;\m{d}y$. Then  $f_Q\in W^{1,1}_\loc\p{\R^n;\mathbb{K}}$, with the distributional gradient identified with the (a.e. defined) function $\grad f_Q\p{x}=\int_{\pd Q}f\p{x+y}\nu\p{y}\;\m{d}\mathcal{H}^{n-1}\p{y}$,
where $\nu:\pd Q\to\R^n$ is the outward unit normal.
\end{lem}
\begin{proof}
This follows from Fubini's theorem, differentiation under the integral, and the fundamental theorem of calculus.
\end{proof}

\begin{lem}[Hardy's inequalities]\label{truncated hardy inequality}
Suppose that $s\in\R^+$, $\sig\in(0,\infty]$, and $1\le p<\infty$.  Then for all measurable functions $\Omega:\p{0,\sig}\to\sb{0,\infty}$ we have the estimates
\begin{equation}\label{truncated hardy inequality 0}
    \bp{\int_{\p{0,\sig}}\bp{t^{-s}\int_{\p{0,t}}\Omega\p{\rho}\rho^{-1}\;\m{d}\rho}^pt^{-1}\;\m{d}t}^{1/p}\le s^{-1}\bp{\int_{\p{0,\sig}}\p{t^{-s}\Omega\p{t}}^pt^{-1}\;\m{d}t}^{1/p}
\end{equation}
and 
\begin{equation}\label{truncated hardy inequality 1}
    \bp{\int_{\R^+}\bp{t^s\int_{\p{t,\infty}}\Omega\p{\rho}\rho^{-1}\;\m{d}\rho}^pt^{-1}\;\m{d}t}^{1/p}\le s^{-1}\bp{\int_{\R^+}\p{t^s\Omega\p{t}}^pt^{-1}\;\m{d}t}^{1/p}
\end{equation}
\end{lem}
\begin{proof}
When $\sigma = \infty$, these are the classical Hardy inequalities: see, for instance, Theorem C.41 in~\cite{MR3726909}.  To prove \eqref{truncated hardy inequality 0} when $\sigma < \infty$, we apply \eqref{truncated hardy inequality 0} on $\R^+$ to the function $\Omega\chi_{\p{0,\sig}}$.
\end{proof}
\subsection{Harmonic Analysis}\label{harmonic analysis}
Here are some important notions from the theory of real and complex valued tempered distributions.

\begin{lem}[Real valued distributions]\label{real valued distributions}
For tempered distributions $f\in\mathscr{S}^\ast\p{\R^n;\C}$ we define the complex conjugate distribution, $\Bar{f}\in \mathscr{S}^\ast\p{\R^n;\mathbb{C}}$ via $\br{\Bar{f},\varphi}=\Bar{\br{f,\Bar{\varphi}}}$, $\varphi\in\mathscr{S}\p{\R^n;\C}$. We say that $f$ is $\R$-valued if $f=\Bar{f}$ and write $f\in\mathscr{S}^\ast\p{\R^n;\R}$. The following hold.
\begin{enumerate}
    \item For $f\in\mathscr{S}^\ast\p{\R^n;\C}$ we have that $f\in\mathscr{S}^\ast\p{\R^n;\R}$ if and only if $\Bar{\hat{f}}=\del_{-1}\hat{f}$.
    \item For $f\in L^1_\loc\p{\R^n;\C}\cap\mathscr{S}^\ast\p{\R^n;\C}$ we have that $f\in\mathscr{S}^\ast\p{\R^n;\R}$ if and only if $f\p{x}\in\R$ for $\Le^n$-a.e. $x\in\R^n$.
\end{enumerate}
\end{lem}
\begin{proof}
If $f\in\mathscr{S}^\ast\p{\R^n;\R}$ and $\varphi\in\mathscr{S}\p{\R^n;\C}$, then we equate
\begin{equation}
\br{f,\varphi}=\br{\Bar{f},\varphi}\lra\br{\del_{-1}\hat{f},\hat{\varphi}}=\Bar{\br{\hat{f},\del_{-1}\hat{\Bar{\varphi}}}}=\Bar{\br{\hat{f},\Bar{\hat{\varphi}}}}=\br{\Bar{\hat{f}},\hat{\varphi}}.
\end{equation}
This gives the first item. If $f$ is in $L^1_\loc\p{\R^n;\C}$ as in the second item, then
\begin{equation}
    \br{f-\Bar{f},\varphi}=\int_{\R^n}\p{f-\Bar{f}}\varphi=0,\;\forall\;\varphi\in\mathscr{S}^\ast\p{\R^n;\C}\lra f\p{x}\in\R\text{ for a.e }x\in\R^n.
\end{equation}
\end{proof}

\begin{defn}[Space of Fourier Multipliers]\label{space of fourier multipliers}
Given $1\le p<\infty$, the space of $L^p\p{\R^n;\mathbb{K}}$-Fourier multipliers, denoted $\mathscr{M}_p\p{\R^n;\mathbb{K}}$, is the set of all $m\in L^\infty\p{\R^n;\C}$ such that 
\begin{equation}
    \norm{m}_{\mathscr{M}_p}=\sup\cb{\norm{\mathscr{F}^{-1}\p{m\mathscr{F}f}}_{L^p}\;:\;f\in\mathscr{S}\p{\R^n;\mathbb{K}},\;\norm{f}_{L^p}=1} < \infty
\end{equation}
and for all $f\in\mathscr{S}\p{\R^n;\mathbb{K}}$ we have that $\mathscr{F}^{-1}\p{m\mathscr{F}}f$ is $\mathbb{K}$-valued.
\end{defn}

We will need a few facts about Fourier multipliers.

\begin{lem}[Invariance under scaling]\label{scaling invariance of fourier multipliers}
Let $1\le p<\infty$. If $m\in\mathscr{M}_p\p{\R^n;\mathbb{K}}$, then for all $\xi\in\R\setminus\cb{0}$ it holds that $\del_{\xi}m\in\mathscr{M}_p\p{\R^n;\mathbb{K}}$ with the equality $\norm{\del_\xi m}_{\mathscr{M}_p}=\norm{m}_{\mathscr{M}_p}$; note that $\del_\xi m$ is the isotropic $\xi$-dilate of $m$, defined via $\del_\xi m\p{\zeta}=m\p{\zeta\xi^{-1}}$.
\end{lem}
\begin{proof}
See Proposition 2.5.14 in~\cite{MR3243734}.
\end{proof}

Next we state sufficient conditions for an essentially bounded function to be a Fourier multiplier.

\begin{thm}[H\"ormander-Mihlin Multiplier Theorem]\label{mihlin multiplier theorem}
Suppose that $m\in L^\infty\p{\R^n;\C}$ is a function whose distributional derivatives on $\R^n\setminus\cb{0}$ of orders less than $\lfloor{n/2}\rfloor+1$ can be identified with locally integrable functions on $\R^n\setminus\cb{0}$. Define the number $A\in\sb{0,\infty}$ via
\begin{equation}
    A=\max\cb{\m{esssup}\cb{\abs{\xi}^{\abs{\al}}\abs{\pd^\al m\p{\xi}}\;:\;\xi\in\R^n\setminus\cb{0}}\;:\;\al\in\N^n,\;\abs{\al}\le\lfloor{n/2}\rfloor+1}.
\end{equation}
It holds that for all $1<p<\infty$ there exists a constant $C_{n,p}\in\R^+$, independent of $m$, such that $\norm{m}_{\mathscr{M}_p}\le C_{n,p}A$.
\end{thm}
\begin{proof}
See Theorem 6.2.7 in~\cite{MR3243734}.
\end{proof}
We can also use Lemma~\ref{real valued distributions} to characterize which multipliers preserve the property being $\R$-valued.

\begin{lem}\label{mult lemma}
Let $1<p<\infty$.  Then $m\in\mathscr{M}_p\p{\R^n;\R}$ if and only if $m\in\mathscr{M}_p\p{\R^n;\C}$ and $\del_{-1}m=\Bar{m}$.  Moreover, for each $\varphi\in\mathscr{S}\p{\R^n;\C}$ we have that $\big(\varphi\hat{f}\big)^{\vee}\in\mathscr{S}^\ast\p{\R^n;\R}$ for all $f\in\mathscr{S}^\ast\p{\R^n;\R}$ if and only if $        \del_{-1}\varphi=\Bar{\varphi}$.
\end{lem}
\begin{proof}
This is clear given Lemma~\ref{real valued distributions}.
\end{proof}

\section*{Acknowledgments}
The authors would like to express their gratitude to the anonymous referee for their thorough reading of the first draft and for the many helpful suggestions.

\bibliography{screened_sobolev_1.bib}
\bibliographystyle{alpha}
\end{document}